\documentclass[12pt]{amsart}
\usepackage[english]{babel}
\usepackage{amsmath}
\usepackage{amsthm}
\usepackage{amssymb}
\usepackage{mathrsfs}
\usepackage{enumerate}
\usepackage[notcite, final, notref]{showkeys}
\usepackage{amsfonts}
\usepackage{dsfont}
\usepackage{float}
\usepackage{tikz}
\usepackage[siunitx]{circuitikz}
\usetikzlibrary{arrows,calc,chains,shapes,dsp}
\def\C{\mathbb C}
\def\R{{\mathbb R}}
\def\T{\hfill{\triangledown\triangledown\triangledown}}
\newtheorem{Pa}{Paper}[section]
\newtheorem{Tm}[Pa]{{\bf Theorem}}
\newtheorem{La}[Pa]{{\bf Lemma}}

\newtheorem{Cy}[Pa]{{\bf Corollary}}
\newtheorem{Rk}[Pa]{{\bf Remark}}
\newtheorem{Pn}[Pa]{{\bf Proposition}}

\newtheorem{Ex}[Pa]{{\bf Example}}
\newtheorem{Dn}[Pa]{{\bf Definition}}
\usepackage{xcolor}

\title[Canonical Hyper-Positive Real Functions]
{Quantitatively Hyper-Positive real Rational\\[0.2cm]
Functions IV: the Canonical Case}

\author[D. Alpay]{Daniel Alpay}
\address{(DA)
Faculty of Mathematics, Physics, and Computation\\
Schmidt College of Science and Technology\\
Chapman University\\
One University Drive
Orange, California 92866\\
USA}
\email{alpay@chapman.edu}
\thanks{Daniel Alpay thanks the Foster G. and Mary McGraw Professorship in
Mathematical Sciences, which supported this research.}

\author[I. Lewkowicz]{Izchak Lewkowicz}
\address{(IL) School of Electrical and Computer Engineering\\
Ben-Gurion University of the Negev\\ P.O.B. 653\\ Beer-Sheva, 84105\\
Israel}
\email{izchak@bgu.ac.il}

\begin{document}
\bibliographystyle{plain}

\begin{abstract}
In the linear time-invariant framework, passive systems are modeled by positive real
functions. Dissipative systems can be modeled by the subset of quantitatively
Hyper-positive real functions, related through nested inclusions. This family was
introduced and studied in our three previous works.
\smallskip

\noindent
Here, we further focus our attention on the proper subset of {\em canonical}
Hyper-Positive functions. Although this family is ``small", its exploration is well
motivated: First, this set turns to be associated with absolute stability (the Lurie
problem). Then, a systematic parametrization of all {\em canonical} Hyper-Positive
functions is introduced. Moreover each {\em canonical} Hyper-Positive function can
be viewed as an extreme point of the convex set of Hyper-Positive functions.
Specifically a convex combination of {\em canonical} Hyper-Positive is Hyper-Positive,
but not {\em canonical}. These observations hold in both frameworks: of analytic
functions and of state-space realization arrays.
\smallskip

\noindent
Technically, some of the analysis is facilitated by employing Quadratic Matrix
Inclusions of both, matrices and of matrix-valued rational functions.
\end{abstract}
\maketitle


\date{today}
\tableofcontents

\section{Introduction and Main Results}
\setcounter{equation}{0}

The set $\mathcal{P}$ serves as a model for passive, continuous-time, linear,
time-invariant systems, see e.g. \cite[Theorem 2.7.1]{AnderVongpa1973},
\cite[Section 3.18]{Belev1968}, \cite[Corollary 2.39]{BroLozaMasEge2020},
\cite[Section 6.3]{Khalil2000}, \cite[Proposition 1]{Will1976} and
\cite[Theorem 11]{YoulCastCarl1959}. Specifically, let $\mathbf{P}_m$
($\overline{\mathbf P}_m$) denotes the set of constant positive (semi-)definite
$m\times m$ matrices, where $m$ is a parameter. A $m\times m$-valued rational
function $F(s)$ is said to be Positive real, i.e. $F\in\mathcal{P}$, whenever,
\begin{equation}\label{eq:Def_P}
F(s)+{F(s)}^*\in\overline{\mathbf P}_m\quad\quad\forall s\in\C_R~,
\end{equation}
where $\C_R$ stands for the open right-half of the complex plane
(and $\C_L$ for the left-half).
\smallskip

The set of $\mathcal{P}$ functions was characterized, through its structure,
in \cite[Theorem 5.3]{Lewk2021a}.

\begin{Tm}\label{Tm:Set_Of_P_Functions}
The set $\mathcal{P}$ is a maximal matrix-convex cone of matrix-valued real rational functions
(of all dimensions), closed under inversion, where each element is analytic in $\C_R~$.
\smallskip

Conversely, a maximal matrix-convex cone of matrix-valued rational functions (of various
dimensions), closed under inversion, which are analytic in $\C_R~$ and containing the zero
degree function \mbox{$F_o(s)\equiv I$}, is the set $\mathcal{P}$.
\end{Tm}

In \cite{AlpayLew2021}, \cite{AlpayLew2024} and \cite{AlpayLew2025a}, we explored subsets of
dissipative systems within $\mathcal{P}$. Specifically, we parametrized sub-families where,
in a rigorous sense, one is ``more stable" than the other. Here are the details.

\begin{Dn}\label{Dn:Hyper_Positive}
{\rm Let \mbox{${\scriptstyle\color{blue}T}$} be a matricial parameter so that
\[
I_m\succ{\color{blue}T}\succcurlyeq 0.
\]
We shall call a $m\times m$-valued rational function $F(s)$,
$\begin{smallmatrix}{\color{blue}T}\end{smallmatrix}$-{\em Hyper-Positive}, i.e.
\mbox{$F\in\mathcal{HP}_{\color{blue}T}$,} if\begin{footnote}{To emphasize the
fact that ${\color{blue}T}$ is a {\em constant} parameter, whenever
next to a function, a smaller font is employed.}\end{footnote},
\begin{equation}\label{eq:HP_Delta}
F(s)+(F(s))^*\succcurlyeq\begin{smallmatrix}{\color{blue}T}\end{smallmatrix}+
(F(s))^*\begin{smallmatrix}
{\color{blue}T}\end{smallmatrix}F(s)\quad\forall s\in\C_R~.
\end{equation}
}
$\T$
\end{Dn}
\smallskip

Already at this stage, the introduction of Hyper-Positive functions, may be viewed as
a refinement of the set $\mathcal{P}$. Specifically, we have the following.

\begin{Pn}\label{Pn:HP_Delta_Order}
Definition \ref{Dn:Hyper_Positive} induces a partial order among family of
functions, namely if
\[
I_m\succ{\color{cyan}T_2}\succcurlyeq{\color{blue}T_1}\succcurlyeq 0
\]
then
\begin{equation}\label{eq:Partial_Order}
\mathcal{HP}_{\color{cyan}T_2}\subset\mathcal{HP}_{\color{blue}T_1}~.
\end{equation}
In addition there is a boundary condition,
\[
{\color{blue}T_1}=0_{m\times m}\quad\Longrightarrow\quad\mathcal{HP}_{{\color{blue}T_1}=0}
=\mathcal{P}.
\]
\end{Pn}

\begin{Rk}
{\rm
{\bf a.}~ Following Eq. \eqref{eq:Partial_Order}, one can say that $
\mathcal{HP}_{\color{cyan}T_2}$ is ``more Lurie-stable" than
$\mathcal{HP}_{\color{blue}T_1}$. This point is further discussed in
items {\bf b.} and {\bf c.} of Remark \ref{Rk:Circle_Criterion} below.
\bigskip

{\bf b.}~
In the special case where in Eq. \eqref{eq:HP_Delta} \mbox{${\color{blue}T}
=\begin{smallmatrix}{\color{blue}\beta}\end{smallmatrix}I_m$,} where
\mbox{$\begin{smallmatrix}{\color{blue}\beta}\end{smallmatrix}\in[0,~1)$} we
shall say that \mbox{$F\in\mathcal{HP}_{\color{blue}\beta}$} if
\begin{equation}\label{eq:Basic_Def_HP_Beta}
F(s)+{F(s)}^*\succcurlyeq\begin{smallmatrix}{\color{blue}\beta}\end{smallmatrix}
(I_m+{F(s)}^*F(s))\quad\quad\forall s\in\C_R~.
\end{equation}
{\bf c.}~
Strictly speaking, Eq. \eqref{eq:HP_Delta} describes {\em Right}~
$\mathcal{HP}_{\color{blue}T}$ functions. {\em Left}~ 
$\mathcal{HP}_{\color{blue}T}$ functions are given by
\begin{equation}\label{eq:Left_HP_T}
F(s)+(F(s))^*\succcurlyeq\begin{smallmatrix}{\color{blue}T}\end{smallmatrix}+F(s)
\begin{smallmatrix}{\color{blue}T}\end{smallmatrix}(F(s))^*\quad\forall s\in\C_R~.
\end{equation}
The lion share of this work focuses on {\em Right} $\mathcal{HP}_{\color{blue}T}$
functions. Some of the slight differences between ``Left" and ``Right"
$\mathcal{HP}_T$ functions were addressed in \cite[Sections 4-6]{AlpayLew2024},
\cite{AlpayLew2025a}, \cite{AlpayColLewSab2025a}. Subsection
\ref{counter-123} below, illustrates this difference.
}
$\T$
\end{Rk}
\smallskip

In the context of the set $\mathcal{HP}_{\color{blue}\beta}$ from Eq.
\eqref{eq:Basic_Def_HP_Beta} Proposition \ref{Pn:HP_Delta_Order} takes the
following form.
\smallskip

\begin{Cy}\label{Cy:HP_beta_Order} 
\[
1>\begin{smallmatrix}{\color{cyan}{\beta}_2}\end{smallmatrix}\geq
\begin{smallmatrix}{\color{blue}{\beta}_1}\end{smallmatrix}\geq 0\quad\Longrightarrow\quad
\mathcal{HP}_{\color{cyan}{\beta}_2}\subset\mathcal{HP}_{\color{blue}{\beta}_1}\subset
\mathcal{HP}_{{\color{blue}\beta}=0}=\mathcal{P}.
\]
\end{Cy}
This order is illustrated in Figure \ref{Fig:Degree_One_HP} below.
\bigskip

In this work we further narrow our scope to focus on the subset of {\em canonical}
$\mathcal{HP}_{\color{blue}T}$ functions.

\begin{Dn}\label{Dn:Canonical_Hyper_Positive}
{\rm
For a given \mbox{$I_m\succ{\color{blue}T}\succcurlyeq 0$,} we shall say that a
$m\times m$-valued \mbox{$\mathcal{HP}_{\color{blue}T}$} function, $F(s)$ is
{\rm canonical}, if 
\begin{equation}\label{eq:Canonical_HP_T}
F(s)+(F(s))^*-(\begin{smallmatrix}{\color{blue}T}\end{smallmatrix}+(F(s))^*
\begin{smallmatrix}{\color{blue}T}\end{smallmatrix}F(s))\in\left\{\begin{matrix}
\overline{\mathbf P}_m&&\forall s\in\C_R\\~\\0&&\forall s\in{i}\R.\end{matrix}\right.
\end{equation}
As before, for a prescribed \mbox{$\begin{smallmatrix}{\color{blue}\beta}
\end{smallmatrix}\in[0,~1)$} we say that \mbox{$F(s)$} in
\mbox{$\mathcal{HP}_{\color{blue}\beta}$} is {\em canonical}, if
}
\begin{equation}\label{eq:Def_Canoinal_Quad_HP_Beta}
F(s)+{F(s)}^*-\begin{smallmatrix}{\color{blue}\beta}\end{smallmatrix}\left(I_m+
{F(s)}^*F(s)\right)\in\left\{\begin{matrix}\overline{\mathbf P}_m&&\forall s\in\C_R
\\~\\0&&\forall s\in{i}\R.
\end{matrix}\right.
\end{equation}
$\T$
\end{Dn}

We start with the main structural properties of this family.

\begin{Tm}\label{Tm:Set_Of_HP_Functions}
Let ${\color{blue}T}$, where \mbox{$I_m\succ{\color{blue}T}\succcurlyeq 0$,} be given.

{\bf A.}
Let the set $\mathcal{HP}_{\color{blue}T}$ of 
functions be as in Definition \ref{Dn:Hyper_Positive}.
\begin{itemize}
\item[(i)~~~]{}
This set is closed under inversion, i.e.  \mbox{$F(s)\in\mathcal{HP}_{\color{blue}T}$}
implies that \mbox{${F(s)}^{-1}\in\mathcal{HP}_{\color{blue}T}$} as well.
\smallskip

\item[(ii)~~]{}
This set is convex.
\smallskip

\item[(iii)~]{}
For given \mbox{$\begin{smallmatrix}{\color{blue}\beta}\end{smallmatrix}\in[0,~1)$},
the set $\mathcal{HP}_{\color{blue}\beta}$ is matrix-convex.
\end{itemize}

{\bf B.}
Let the subset {\em canonical} $\mathcal{HP}_{\color{blue}T}$ functions, be as in 
Definition \ref{Dn:Canonical_Hyper_Positive}.
\begin{itemize}

\item[(i)~~~]{}
Let $F_0(s)$ and $F_1(s)$ be a pair of {\em canonical} $\mathcal{HP}_{\color{blue}T}$
functions. Then the convex-hull of $F_0(s)$ and $F_1(s)$ is comprized of
{\em canonical} $\mathcal{HP}_{\color{blue}T}$ functions as well, if and only if,
\begin{equation}\label{eq:Condition_Canonical_Convex}
((F_0-F_1)(s))^*{\scriptstyle\color{blue}T}(F_0-F_1)(s)\equiv 0\quad\quad\forall s\in{i}\R.
\end{equation}
Else functions in this convex-hull are not {\em canonical} and belong to
$\mathcal{HP}_{\color{cyan}\hat{T}}$ for some
\mbox{$I_m\succ{\color{cyan}\hat{T}}\succcurlyeq{\color{blue}T}$.}
\smallskip

\item[(ii)~~]{}
Whenever \mbox{$F\in\mathcal{HP}_{\color{blue}T}$} is {\em canonical}, its inverse
$\left(F(s)\right)^{-1}$ is 
another {\em canonical} function
within the same \mbox{$\mathcal{HP}_{\color{blue}T}$.}
\smallskip

\item[(iii)~]{}
Whenever $f(s)$ is a scalar {\em canonical} \mbox{$\mathcal{HP}_{\color{blue}\beta}$}
function, \mbox{$f_1(s):=\frac{1}{2}(f(s)+{f(s)}^{-1})$} is another {\em canonical}
function within \mbox{$\mathcal{HP}_{\color{blue}{\beta}_1}$,} where
\mbox{${\color{blue}\begin{smallmatrix}{\beta}_1\end{smallmatrix}}=
\begin{smallmatrix}\left(\frac{1}{2}({\color{blue}\beta}+\frac{1}{{\color{blue}\beta}})
\right)^{-1}\end{smallmatrix}$.}
\end{itemize}
\end{Tm}

\begin{Rk}\label{Rk:Structure_Canonical}
{\rm
{\bf a.}~ {\em Maximality}, is the key to the gap between
Theorem \ref{Tm:Set_Of_P_Functions} and part {\bf A.} of 
Theorem \ref{Tm:Set_Of_HP_Functions}.
\smallskip

{\bf b.}~
Part {\bf A.} of Theorem \ref{Tm:Set_Of_HP_Functions} is part of
\cite[Therorem 1.9]{AlpayLew2025a}.
\smallskip

{\bf c.}~ 
Roughly speaking, for a given \mbox{$I_m\succ{\color{blue}T}\succcurlyeq 0$,} we shall
say that a $m\times m$-valued $F(s)$ is a {\em boundary}
\mbox{$\mathcal{HP}_{\color{blue}T}$} function, if
\[
F(s)+(F(s))^*-(\begin{smallmatrix}{\color{blue}T}\end{smallmatrix}+(F(s))^*
\begin{smallmatrix}{\color{blue}T}\end{smallmatrix}F(s))\in\left\{
\begin{matrix}\overline{\mathbf P}_m&&
\forall s\in\C_R\\~\\ 0&&{\rm for~some}~s\in{i}\R.\end{matrix}\right.
\]
In this sense, {\em canonical} \mbox{$\mathcal{HP}_{\color{blue}T}$} functions in
Eq.  \eqref{eq:Canonical_HP_T} are the extreme points of the convex
\mbox{$\mathcal{HP}_{\color{blue}T}$} family. This point is discussed in Subsection
\ref{Subsec:Convex_Combination_Rational} below.
\smallskip

{\bf d.}~ For a proof of part {\bf B.} of Theorem \ref{Tm:Set_Of_HP_Functions}, see
Subsection \ref{Subsec:Proof_of_Theorem_1.7_B} below.

$\T$
}
\end{Rk}
\smallskip

Although the family of {\em canonical} Hyper-Positive functions is small, it is of
interest: First, in its own right, see e.g. Section \ref{Sec:Lurie_Problem} and
Figures \ref{Fig:Impedance_Degree_One}, \ref{Fig:Degree_One_HP}. Second, it is on
the boundary of the
{\em convex} family of Hyper-Positive functions, so one can expect the following.

\begin{Pn}\label{Pn:Convex_Combination_Degree_One}
For \mbox{$\begin{smallmatrix}{\color{blue}\beta}\end{smallmatrix}\in(0,~1)$,}
a scalar \mbox{$\mathcal{HP}_{\color{blue}\beta}$} function of degree one, can
always be written as a convex combination of a pair of {\em canonical}
\mbox{$\mathcal{HP}_{\color{blue}\beta}$}
functions of degree one (and degree zero).
\end{Pn}

A proof is given in Subsection \ref{Subsec:Convex_Combination_Rational} below.
\bigskip

It turns out that also in the state-space realization setup, the family of
$\mathcal{HP}_{\color{blue}T}$ functions possess a rich structure. To explore it,
one needs to first resort to the celebrated Kalman-Yakubovich-Popov Lemma
characterizing $\mathcal{P}$ rational functions through the respective state
space realization. To this end, recall that a \mbox{$m\times m$-valued} rational
function $F(s)$, with no pole at infinity, admits a realization
\begin{equation}\label{eq:Realization}
F(s)=C(sI_n-A)^{-1}B+D,
\quad\quad\quad
R_F={\footnotesize
\left(
\begin{array}{c|c}A&B\\ \hline C&D\end{array}\right)}. 
\end{equation}
Now, if there exists $H\in\mathbf{P}_n$ so that,
\begin{equation}\label{eq:Original_KYP}
\left(\begin{smallmatrix}-H&&0\\~\\0&&I_m\end{smallmatrix}\right)R_F+{R_F}^*
\left(\begin{smallmatrix}-H&&0\\~\\0&&I_m\end{smallmatrix}\right)
\in\overline{\mathbf P}_{n+m},
\end{equation}
then $F(s)$ is in $\mathcal{P}$. See e.g. \cite[Chapter 5]{AnderVongpa1973},
\cite[Chapter 7]{Belev1968}, \cite[Subsection 2.7.2]{BGFB1994},
\cite[Theorem 3]{Will1972b} and \cite[Proposition 2]{Will1976}. We can
now specialize this result to the Hyper-Positive framework.

\begin{Tm}\label{Tm:Kyp_Hyper_Pos_W}
Let $F(s)$ be a $m\times m$-valued rational function as in Eq. \eqref{eq:Realization}
\smallskip

{\bf A.}
\begin{itemize}
\item[(i)~~]{}
If there exist matrices 
$I_m\succ{\color{blue}T}\succcurlyeq 0$\mbox{(${\rm rank}({\color{blue}T})\geq 1$)}
and $H\in\mathbf{P}_n$, satisfying
\begin{equation}\label{eq:HP_Delta_KYP}
\left(\begin{smallmatrix}-H&&0\\~\\0&&I_m\end{smallmatrix}\right)R_F+{R_F}^*
\left(\begin{smallmatrix}-H&&0\\~\\0&&I_m\end{smallmatrix}\right)\succcurlyeq
\left(\begin{smallmatrix}C~&&D\\~\\0_{m\times n}&&I_m\end{smallmatrix}\right)^*
\left(\begin{smallmatrix}{\color{blue}T}&&0\\~\\0&&{\color{blue}T}\end{smallmatrix}\right)
\left(\begin{smallmatrix}C~&&D\\~\\0_{m\times n}&&I_m\end{smallmatrix}\right)
\end{equation}
then the function $F(s)$ is $\begin{smallmatrix}{\color{blue}T}
\end{smallmatrix}$-Hyper-Positive.\\
If the above realization is minimal, the converse is true as well.
\smallskip

\item[(ii)~~]{}
Whenever Eq. \eqref{eq:HP_Delta_KYP} holds and the realization is minimal, the
\mbox{$(n+m)\times(n+m)$}
``matrix" $R_F$ is non-singular, and denote \mbox{$R_{\hat{F}}:={R_F}^{-1}$.} 
\smallskip

\item[(iii)~]{}
Let $\hat{F}(s)$ be a \mbox{$m\times m$-valued} rational function whose realization
array is $R_{\hat{F}}$ from the previous item. Then $\hat{F}(s)$ is a
$\mathcal{HP}_{\color{blue}T}$ function, with the same ${\color{blue}T}$.
\end{itemize}
\smallskip

{\bf B.}
\begin{itemize}
\item[(i)~~]{} If Eq. \eqref{eq:HP_Delta_KYP} holds with equality, i.e.
\begin{equation}\label{eq:KYP_Canonical_T}
\left(\begin{smallmatrix}-H&&0\\~\\0&&I_m\end{smallmatrix}\right)R_F+{R_F}^*
\left(\begin{smallmatrix}-H&&0\\~\\0&&I_m\end{smallmatrix}\right)=
\left(\begin{smallmatrix}C~&&D\\~\\0_{m\times n}&&I_m\end{smallmatrix}\right)^*
\left(\begin{smallmatrix}{\color{blue}T}&&0\\~\\0&&{\color{blue}T}\end{smallmatrix}\right)
\left(\begin{smallmatrix}C~&&D\\~\\0_{m\times n}&&I_m\end{smallmatrix}\right),
\end{equation}
then $F(s)$ is a {\em canonical}.\\
If the above realization is minimal, the converse is true as well.
\smallskip

\item[(ii)~~]{}
With $R_F$ from the previous item, let \mbox{$R_{\hat{F}}:=R_F^{-1}$}, and let
$\hat{F}(s)$ be a \mbox{$m\times m$-valued} rational function whose realization
array is $R_{\hat{F}}$. Then $\hat{F}(s)$ is another {\em canonical}
$\mathcal{HP}_{\color{blue}T}$ function, with the same ${\color{blue}T}$.
\end{itemize}
\end{Tm}
\smallskip

\begin{Rk}\label{Rk:KYP_HP}
{\rm
{\bf a.}~ When one substitutes in Eq. \eqref {eq:HP_Delta_KYP} \mbox{${\color{blue}T}=0$},
the set $\mathcal{P}$ is obtained, see Eq. \eqref{eq:Original_KYP}.
\smallskip

{\bf b.}~
In the special case of $\mathcal{HP}_{\color{blue}\beta}$ (see Eq.
\eqref{eq:Basic_Def_HP_Beta}) \eqref{eq:HP_Delta_KYP} is simplified to
\begin{equation}\label{eq:M_HP_beta}
\left(\begin{smallmatrix}-H&&0\\~\\0&&I_m\end{smallmatrix}\right)R_F
+{R_F}^*\left(\begin{smallmatrix}-H&&0\\~\\0&&I_n\end{smallmatrix}\right)
\succcurlyeq\begin{smallmatrix}{\color{blue}\beta}\end{smallmatrix}
\left(\begin{smallmatrix}C~&&D\\~\\0_{m\times n}&&I_m\end{smallmatrix}\right)^*
\left(\begin{smallmatrix}C~&&D\\~\\0_{m\times n}&&I_m\end{smallmatrix}\right)
\end{equation}
and the {\em canonical} version is,
\begin{equation}\label{eq:KYP_Canonical_beta}
\left(\begin{smallmatrix}-H&&0\\~\\0&&I_m\end{smallmatrix}\right)R_F
+{R_F}^*\left(\begin{smallmatrix}-H&&0\\~\\0&&I_n\end{smallmatrix}\right)
=
\begin{smallmatrix}{\color{blue}\beta}\end{smallmatrix}
\left(\begin{smallmatrix}C~&&D\\~\\0_{m\times n}&&I_m\end{smallmatrix}\right)^*
\left(\begin{smallmatrix}C~&&D\\~\\0_{m\times n}&&I_m\end{smallmatrix}\right).
\end{equation}
{\bf c.}~
Part {\bf A.} of the theorem appeared in \cite[Theorem 1.11]{AlpayLew2025a}.
\smallskip

{\bf d.}~
For a proof of part {\bf B.} see Subsection \ref{Subsec:Structure} below.
}
$\T$
\end{Rk}
\smallskip

As a consequence of Eq. \eqref{eq:KYP_Canonical_T} we have the following.

\begin{Cy}\label{Cy:Parametrization_Of_Realizations}

\begin{itemize}
\item[(i)~~]{}
$F(s)$ is a  ~{\em canonical} $\mathcal{HP}_{\color{blue}T}$ function, for some
\mbox{$I_m\succ{\color{blue}T}\succ 0$,} if and only if, it
admits a \mbox{$(n+m)\times(n+m)$} minimal realization $R_F$, of the form
\begin{equation}\label{eq:Realization_Canonical}
R_F={\footnotesize\left(\begin{array}{c|c}A&B\\ \hline C&D\end{array}\right)}=
\left({\footnotesize\begin{array}{c|c}{\scriptstyle-\frac{1}{2}}C^*{\color{blue}T}C
+Y&C^*{\color{blue}T}^{\frac{1}{2}}U({\color{blue}T}^{-1}-{\color{blue}T})^{\frac{1}{2}}
\\ \hline C&{\color{blue}T}^{-1}-{\color{blue}T}^{-\frac{1}{2}}U({\color{blue}T}^{-1}-
{\color{blue}T})^{\frac{1}{2}}\end{array}}\right),
\end{equation}
where the parameters are,
\[
\begin{smallmatrix}U\in\C^{m\times m}&&U^*U=I_m&&&&Y\in\C^{n\times n}&&Y+Y^*=0\\~\\
C\in\C^{m\times n}&&{\rm is~of~a~full~rank}&&&&{\rm the~pair}~~Y,~C&&{\rm is~observable}.
\end{smallmatrix}
\]

\item[(ii)~]{}
$F(s)$ is a  ~{\em canonical} $\mathcal{HP}_{\color{blue}\beta}$ function, for some
\mbox{$\begin{smallmatrix}{\color{blue}\beta}\end{smallmatrix}\in(0,~1)$,} if and only if, 
$\hat{R}_F$ a
corresponding \mbox{$(n+m)\times(n+m)$} balanced realization takes the form,
\begin{equation}\label{eq:Realization_Balanced_Canonical}
\hat{R}_F=
\left({\footnotesize\begin{array}{c|c}{\scriptstyle-\frac{\color{blue}\beta}{2}}C^*C+Y&
(1-{\scriptstyle{\color{blue}\beta}^2})^{\frac{1}{4}}C^*U\\ \hline
(1-{\scriptstyle{\color{blue}\beta}^2})^{\frac{1}{4}}C&
{\scriptstyle\frac{1}{\color{blue}\beta}}(I_m-
{\scriptstyle\sqrt{1-{\color{blue}\beta}^2}}U)\end{array}}\right),
\end{equation}
where the parameters are,
\[
\begin{smallmatrix}U\in\C^{m\times m}&&U^*U=I_m&&&&Y\in\C^{n\times n}&&Y+Y^*=0\\~\\
C\in\C^{m\times n}&&{\rm is~of~a~full~rank}&&&&{\rm the~pair}~~Y,~C&&{\rm
is~observable}.\end{smallmatrix}
\]
\item[(iii)]{}
If
$F(s)$ is a  ~{\em canonical} $\mathcal{HP}_{\color{blue}\beta}$ function, for some
\mbox{$\begin{smallmatrix}{\color{blue}\beta}\end{smallmatrix}\in(0,~1)$,} of
McMillan degree $n$, then its Hankel singular values are
\begin{equation}\label{eq:Equal_Hankel_Singular_Values}
{\sigma}_1=~\ldots~={\sigma}_n=\begin{smallmatrix}\frac{\sqrt{1-{\color{blue}\beta}^2}}
{\color{blue}\beta}\end{smallmatrix}~.
\end{equation}
\item[(iv)]{}
Conversely, if $\Phi(s)$ is a $m\times m$-valued rational function, vanishing at
infinity, whose all $n$ Hankel singular values are equal, then, one can always find
an $m\times m$ matrix $D$, so that the resulting $D+\Phi(s)$ is a {\em canonical} 
$\mathcal{HP}_{\color{blue}\beta}$ function for some
\mbox{$\begin{smallmatrix}{\color{blue}\beta}\end{smallmatrix}\in(0,~1)$.}
\end{itemize}
\end{Cy}
\smallskip

\begin{Rk}
{\rm
{\bf a.}~ For proof see Subsection \ref{Subsec:Parametrization_of_Realization}.
\smallskip

{\bf b.}~ 
One can view the triple of matrices $C$, $Y$, $U$ in Corollary
\ref{Cy:Parametrization_Of_Realizations}, as parametrizing all {\em canonical}
$\mathcal{HP}_{\color{blue}\beta}$ functions, where dimensions and McMillan
degree are prescribed.
}

$\T$
\end{Rk}
\smallskip

We next examine realization of {\em families} of $\mathcal{HP}_{\color{blue}T}$ functions.

\begin{Tm}\label{Tm:Convexity_Sets_Of_Realizations}
Consider the framework of Theorem \ref{Tm:Kyp_Hyper_Pos_W}.
\smallskip

{\bf A.}
\begin{itemize}
\item[(i)~~]{}
Let the matricial parameters $I_m\succ{\color{blue}T}\succcurlyeq 0$ and
$H\in\mathbf{P}_n$ be
prescribed. Then the set of all \mbox{$(n+m)\times(n+m)$} realization
arrays $R_F$ satisfying Eq. \eqref{eq:HP_Delta_KYP}, is convex.
\smallskip

\item[(ii)~]{}
Consider the set of all \mbox{$(n+m)\times(n+m)$} realization array $R_F$ satisfying Eq.
\eqref{eq:M_HP_beta}, where $H=I_n$ and the parameter
\mbox{$\begin{smallmatrix}{\color{blue}\beta}\end{smallmatrix}\in[0,~1)$,} is prescribed.
Then this set is \mbox{$n, m$-matrix-convex.}
\end{itemize}
\smallskip

{\bf B.}
\begin{itemize}
\item[(i)~~]{}
Let $F_0(s)$ and $F_1(s)$  be a pair of {\em canonical} $\mathcal{HP}_{\color{blue}T},$
functions, for some \mbox{$I_m\succ{\color{blue}T}\succcurlyeq 0$,} and let
\mbox{$R_0=\left({\footnotesize\begin{array}{c|c}A_0&B_0\\ \hline C_0&D_0\end{array}}
\right)$,} \mbox{$R_1=\left({\footnotesize\begin{array}{c|c}A_1&B_1\\ \hline C_1&D_1
\end{array}}\right)$} be corresponding realizations, satisfying Eq.
\eqref{eq:KYP_Canonical_T} with the same $H\in\mathbf{P}_n~$.
\smallskip

\noindent
For $\begin{smallmatrix}\alpha\end{smallmatrix}\in[0,~1]$ denote by \mbox{$R_{\alpha}:=
\begin{smallmatrix}\alpha\end{smallmatrix}{R_1}+\begin{smallmatrix}(1-\alpha)
\end{smallmatrix}R_0$,} a realization of a function\begin{footnote}{This $F_{\alpha}(s)$
should not be confused with a
convex combination of $F_0(s)$ and $F_1(s)$.}\end{footnote}$F_{\alpha}(s)$.
\smallskip

\noindent
The resulting $F_{\alpha}(s)$ is a {\em canonical} \mbox{$\mathcal{HP}_{\color{blue}T}$}
function, $\forall\begin{smallmatrix}\alpha\end{smallmatrix}\in[0,~1],$ if and only if,
\begin{equation}\label{eq:Convex_Realization_Canonical}
\underbrace{
\left(\begin{matrix}C_1-C_0&&D_1-D_0\end{matrix}\right)
{\color{blue}T}
\left(\begin{matrix}C_1^*-C_0^*\\~\\D_1^*-D_0^*\end{matrix}\right)
}_X
=0_{m\times m}~.
\end{equation}
\smallskip

\noindent
In particular, for \mbox{${\color{blue}T}\succ 0$,} this is the case when,
\[
A_0=A_1+iH,~~H\in\overline{\mathbf H}_n\quad\quad\quad B_0=B_1
\quad\quad\quad C_0=C_1\quad\quad\quad D_0=D_1~.
\]
\item[(ii)~]{}
If in Eq. \eqref{eq:Convex_Realization_Canonical} 
\mbox{$X:=\left(\begin{smallmatrix}C_1-C_0&&D_1-D_0\end{smallmatrix}\right)
\begin{smallmatrix}
{\color{blue}T}
\end{smallmatrix}
\left(\begin{smallmatrix}C_1^*-C_0^*\\~\\D_1^*-D_0^*\end{smallmatrix}\right)
\not=0_{m\times m}$,}
then
\mbox{$\forall\begin{smallmatrix}\alpha\end{smallmatrix}\in(0,~1)$,} the resulting
$F_{\alpha}(s)$ is a non {\em canonical} $\mathcal{HP}_{\color{blue}T}$ function.
\end{itemize}
\end{Tm}

\begin{Rk}
{\rm
{\bf a.}~
See \cite[Proposition 1.15]{AlpayLew2025a}, for proof of part {\bf A.} 
\smallskip

{\bf b.}~ A proof of part {\bf B.} is given in Subsection 
\ref{Subsec:Convexity_of_Realizations} below.
}
$\T$
\end{Rk}
\smallskip

All results which to the best of our knowledge, have previously appeared (including
ours) are explicitly indicated.
\smallskip

The outline of this work can be inferred from the table of contents, 
appearing immediately ahead of Section 1.

\section{Background}
\label{Sec:Background}
\setcounter{equation}{0}

\subsection{Rational Functions in Quadratic Form}
In the sequel we find it convenient to adopt quadratic formulation, which will enable
us to technically unify the treatment of rational functions and state-space
realizations.
\smallskip

First, the family $\mathcal{HP}_{\color{blue}T}$ with 
\mbox{$I_m\succ{\color{blue}T}\succcurlyeq 0$}, (originally defined in Eq.
\eqref{eq:HP_Delta}, can be described as,
\begin{equation}\label{eq:Quadratic_HP_Delta}
\mathcal{HP}_{\color{blue}T}=\left\{~\begin{smallmatrix}F(s)\end{smallmatrix}~:~\left(
\begin{smallmatrix}F(s)\\~\\I_m\end{smallmatrix}\right)^*\left(\begin{smallmatrix}-
{\color{blue}T}&&~~I_m\\~\\~~I_m&&-{\color{blue}T}\end{smallmatrix}\right)\left(
\begin{smallmatrix}F(s)\\~\\I_m\end{smallmatrix}\right)\in\overline{\mathbf P}_m\quad
\forall s\in\C_R~\right\}.
\end{equation}
Substituting ${\color{blue}T}=0$, yields the quadratic
form of $\mathcal{P}$ functions from Eq. \eqref{eq:Def_P}, namely
\begin{equation}\label{eq:Quad_Def_P}
\mathcal{P}=\left\{~\begin{smallmatrix}F(s)\end{smallmatrix}~:~
\left(\begin{smallmatrix}F(s)\\~\\I_m\end{smallmatrix}\right)^*
\left(\begin{smallmatrix}0&&~~I_m\\~\\~~I_m&&~0\end{smallmatrix}\right)
\left(\begin{smallmatrix}F(s)\\~\\I_m\end{smallmatrix}\right)
\in\overline{\mathbf P}_m\quad\forall s\in\C_R~\right\}.
\end{equation}
Within the family $\mathcal{P}$, of particular interest is the subset of $\mathcal{PO}$
functions (a.k.a. Foster or Lossless) described as,
\begin{equation}\label{eq:Def_PO}
\mathcal{PO}:=\{F(s)\in\mathcal{P}~:~F(s)=-\left(F(-s^*)\right)^*~\}.
\end{equation}
For details see e.g. \cite[Theorem 2.7.4]{AnderVongpa1973},
\cite[Ch 8, items 36-50]{Belev1968}, \cite[Section 4.2]{CohenLew2007}, \cite{Lewk2021a} and
\cite[p. 36]{Wohl1969}. A quadratic form of $\mathcal{PO}$ functions is,
\begin{equation}\label{eq:Quad_Def_PO}
\mathcal{PO}=\left\{~\begin{smallmatrix}F(s)\end{smallmatrix}~:~\left(\begin{smallmatrix}
F(s)\\~\\I_m\end{smallmatrix}\right)^*\left(\begin{smallmatrix}0&&~~I_m\\~\\~~I_m&&~0
\end{smallmatrix}\right)\left(\begin{smallmatrix}F(s)\\~\\I_m\end{smallmatrix}\right)
\in\left\{\begin{smallmatrix}\overline{\mathbf P}_m&~&\forall s\in\mathbb{C}_L\\~\\
0&~&\forall s\in{i}\mathbb{R}\\~\\
\overline{\mathbf P}_m&~&~\forall s\in\mathbb{C}_R~.\end{smallmatrix}\right.~\right\}
\end{equation}
Now, the quadratic form of Definition \ref{Dn:Hyper_Positive} is:\\
For a prescribed $I_m\succ{\color{blue}T}\succcurlyeq 0$ we shall say that $F(s)$ is a
{\em canonical} $\mathcal{HP}_{\color{blue}T}$ function if
\begin{equation}\label{eq:Def_Canoinal_Quad_HP_W}
\begin{smallmatrix}\left(\begin{smallmatrix}F(s)\\~\\I_m\end{smallmatrix}\right)^*\left(
\begin{smallmatrix}-{\color{blue}T}&&~~I_m\\~\\~~I_m&&-{\color{blue}T}\end{smallmatrix}
\right)\left(\begin{smallmatrix}F(s)\\~\\I_m\end{smallmatrix}\right)&\in\left\{
\begin{smallmatrix}\overline{\mathbf P}_m&&\forall s\in\C_R\\~\\0&&\forall s\in{i}\R.
\end{smallmatrix}\right.\end{smallmatrix}
\end{equation}

\begin{Rk}\label{Rk:PO_Canonical_HP_Delta}
{\rm 
In Proposition \ref{Pn:HP_Delta_Order} it was stated that the set $\mathcal{P}$ is
recovered from $\mathcal{HP}_{\color{blue}T}$, when ${\color{blue}T}=0$. We here point out
that in particular, (when ${\color{blue}T}=0$) the subset of $\mathcal{PO}$ functions is
obtained from ~{\em canonical} $\mathcal{HP}_{\color{blue}T}$ functions. Indeed compare Eq. 
\eqref{eq:Quad_Def_PO} with Eq. \eqref{eq:Def_Canoinal_Quad_HP_W}. This point is further
discussed in Remark \ref{Rk:PO_not_Canonical} below.
}
$\T$
\end{Rk}
\smallskip



\subsection{First Examples}

\begin{Ex}
{\rm 
Without loss of generality, a scalar hyper-positive function of degree one
can always be written in of the form\begin{footnote}{For completeness, note
that when the coefficients are not restricted to be real, then
Eq. \eqref{eq:HP_Deg_One}
takes the form of
\[
\tilde{\phi}_1(s)=\begin{smallmatrix}\frac{1}{{\beta}_1}\end{smallmatrix}
+e^{i{\theta}_1}
\begin{smallmatrix}\frac{\sqrt{1-{\beta}_2^2}}{{\beta}_2}\end{smallmatrix}\frac{s-a}{s+a}
\quad\quad{\rm or}\quad\quad
\tilde{\phi}_2(s)=\begin{smallmatrix}\frac{1}{{\beta}_1}\end{smallmatrix}
-e^{i{\theta}_2}
\begin{smallmatrix}\frac{\sqrt{1-{\beta}_2^2}}{{\beta}_2}\end{smallmatrix}\frac{s-b}{s+b}
\quad\quad\begin{smallmatrix}
1>{\beta}_2\geq{\beta}_1>0\\~\\a, b>0
\\~\\
{\theta}_1,
{\theta}_2\in[0.~2\pi).
\end{smallmatrix}
\]
}\end{footnote}
\begin{equation}\label{eq:HP_Deg_One}
\tilde{\phi}_1(s)=\begin{smallmatrix}\frac{1}{{\beta}_1}\end{smallmatrix}
+\begin{smallmatrix}\frac{\sqrt{1-{\beta}_2^2}}{{\beta}_2}\end{smallmatrix}\frac{s-a}{s+a}
\quad\quad{\rm or}\quad\quad
\tilde{\phi}_2(s)=\begin{smallmatrix}\frac{1}{{\beta}_1}\end{smallmatrix}
-\begin{smallmatrix}\frac{\sqrt{1-{\beta}_2^2}}{{\beta}_2}\end{smallmatrix}\frac{s-b}{s+b}
\quad\quad\begin{smallmatrix}
1>{\beta}_2\geq{\beta}_1>0\\~\\a, b>0.\end{smallmatrix}
\end{equation}
Indeed, a scalar hyper-positive function of degree one, is of the form
\mbox{$\frac{{\alpha}s+\gamma}{s+\delta}$} with
\mbox{$\begin{smallmatrix}\alpha, \gamma, \delta\end{smallmatrix}>0$.} Note now that
\[
\frac{{\alpha}s+\gamma}{s+\delta}=\begin{smallmatrix}\frac{1}{2}(\alpha+\frac{\gamma}
{\delta})\end{smallmatrix}+\begin{smallmatrix}\frac{\nu}{2}\left|\alpha-\frac{\gamma}
{\delta}\right|\end{smallmatrix}\frac{s-\delta}{s+\delta}\quad{\rm with}\quad
\begin{smallmatrix}\frac{1}{{\beta}_1}&=&\frac{1}{2}(\alpha+\frac{\gamma}{\delta})
\\~\\ \nu&=&\pm{1}\\~\\ \frac{\sqrt{1-{{\beta}_2^2}}}{{\beta}_2}
&=&\frac{1}{2}\left|\alpha-\frac{\gamma}{\delta}\right|.\end{smallmatrix}
\]
The particular case of {\em canonical} $\mathcal{HP}_{\color{blue}\beta}$ functions
is obtained, when in Eq. \eqref{eq:HP_Deg_One} one substitutes
\mbox{$\begin{smallmatrix}{\beta}_1\end{smallmatrix}
=\begin{smallmatrix}{\beta}_2\end{smallmatrix}=
\begin{smallmatrix}\beta\end{smallmatrix}$,} i.e.
\begin{equation}\label{eq:Canonical_Degree_One}
{\phi}_1(s)=\begin{smallmatrix}\frac{1}{\beta}\end{smallmatrix}+
\begin{smallmatrix}\frac{\sqrt{1-{\beta}^2}}{\beta}\end{smallmatrix}\frac{s-a}{s+a}
\quad\quad\quad
{\phi}_2(s)=\begin{smallmatrix}\frac{1}{\beta}\end{smallmatrix}
-\begin{smallmatrix}\frac{\sqrt{1-{\beta}^2}}{\beta}\end{smallmatrix}\frac
{s-b}{s+b}\quad\quad\quad\begin{smallmatrix}\beta\in(0,~1)\\~\\a, b>0.\end{smallmatrix}
\end{equation}
Nyquist plots of the four functions in Eqs. \eqref{eq:HP_Deg_One} and
\eqref{eq:Canonical_Degree_One}, are illustrated in Figure \ref{Fig:Degree_One_HP}
below; and of scalar {\em canonical} $\mathcal{HP}_{\color{blue}\beta}$ functions
in Figures \ref{Fig:Degree_One_HP}, \ref{Figure:Convex_Two_Canonical_1} and
\ref{Figure:Convex_Two_Canonical_3} below.
\smallskip

It is easy to verify that in fact,
\begin{equation}\label{eq:Inverse_Caninical_Degree_One}
{\phi}_2(s)=\left({\phi}_1(s)\right)^{-1}\quad{\rm with}\quad
b=a\left(\begin{smallmatrix}\frac{1-\sqrt{1-{\beta}^2}}{\beta}\end{smallmatrix}\right)^2.
\end{equation}
See item  {\bf B} (ii) of Theorem \ref{Tm:Set_Of_HP_Functions}.
and item {\bf b} in Example \ref{Ex:Combinations_in_Figure}
\smallskip

Sometimes, we find it convenient to re-write Eq. \eqref{eq:Canonical_Degree_One} as
\begin{equation}\label{eq:Alternative_Canonical_Degree_One}
\phi_1(s)=\frac{\begin{smallmatrix}r\end{smallmatrix}s+\begin{smallmatrix}\frac{1}{r}
\end{smallmatrix}a}{s+a}\quad\quad\quad
\phi_2(s)=\frac{\begin{smallmatrix}\frac{1}{r}\end{smallmatrix}s+\begin{smallmatrix}r
\end{smallmatrix}b}{s+b}\quad\quad\quad\begin{smallmatrix}\beta\in(0,~1)\\~\\
r:=\frac{1+\sqrt{1-{\beta}^2}}{\beta}\\~\\a, b>0.\end{smallmatrix}
\end{equation}
\smallskip

In a way similar to Eq. \eqref{eq:Canonical_Degree_One}, we now present a pair of
scalar {\em canonical} $\mathcal{HP}_{\beta}$ functions of degree two,
\begin{align}
{\phi}_3(s)=&\begin{smallmatrix}\frac{1}{\beta}\end{smallmatrix}
+\begin{smallmatrix}\frac{\sqrt{1-{\beta}^2}}{\beta}\end{smallmatrix}
\frac{(s-c)(s-d)}{(s+c)(s+d)}
&&c, d>0
\label{al:Phi_3}
\\~\nonumber \\
{\phi}_4(s)=&
\begin{smallmatrix}\frac{1}{\beta}\end{smallmatrix}
-\begin{smallmatrix}\frac{\sqrt{1-{\beta}^2}}{\beta}\end{smallmatrix}
\frac{(s-\gamma)(s-\delta)}{(s+\gamma)(s+\delta)}&&\gamma, \delta>0.
\label{al:Phi_4}
\end{align}
\smallskip

Note that the above functions ${\phi}_1(s)$, ${\phi}_2(s)$, ${\phi}_3(s)$,
${\phi}_4(s)$ along with the zero degree function
\mbox{${\phi}_0(s)\equiv\begin{smallmatrix}\frac{1}{\beta}\end{smallmatrix}$,} may serve
as a parametrization of all scalar {\em canonical} functions of degree up to two.
}
$\T$
\end{Ex}
\smallskip

As a first illustration of an application of Hyper-Positive functions we have the
following.

\begin{figure}[H]
\centering
\begin{minipage}{0.48\linewidth}
{\rm
For all $C, R_1, R_2\geq 0$ \mbox{($R_1+R_2>0$),} $Z_{\rm in}(s)$, the driving point
impedance of this circuit, is in $\mathcal{P}$.  Whenever $C, R_1, R_2\in(0,~\infty)$,
$Z_{\rm in}(s)$ is in $\mathcal{HP}_{\color{blue}\beta}$ with 
\mbox{$\begin{smallmatrix}{\color{blue}\beta}
\end{smallmatrix}\in(0,~2(R_2^2+4)^{-\frac{1}{2}}]$.}
\smallskip 

This $Z_{\rm in}(s)$ is a {\em canonical} $\mathcal{HP}_{\color{blue}\beta}$ function,
when \mbox{$R_1=\begin{smallmatrix}\frac{1}{2}(\sqrt{R_2^2+4}-R_2)\end{smallmatrix}$.}
In this case, \mbox{$\begin{smallmatrix}{\color{blue}\beta}\end{smallmatrix}$} is maximal.
}
\end{minipage}\quad\begin{minipage}{0.47\linewidth}
\begin{tikzpicture}[scale=1.3]
	   \draw[color=black, thick]
		      (0,0) to [short,o-] (3.6,0){} 
		      (-0.1,0.6) node[]{\large{$\mathbf{Z}_{\rm\bf in}~~\mathbf{\rightarrow}$}}
		     (0,1.2) to [short,o-] (0.1,1.2)
		     (0.1,1.2)  to [R,l=$\mathbf{R_1}$,](2,1.2)
		     (2,1.2)   to node[short]{} (3.6,1.2)
		     (3.6,0) to [C,l=$\mathbf{C}$,*-*] (3.6,1.2)
		     (2,0) to [R, l=$\mathbf{R_2}$, *-*] (2,1.2)
		     ;
\end{tikzpicture}
\caption{${\rm\bf ~Z}_{\rm\bf in}(s)=R_1+\frac{\frac{1}{C}}{s+\frac{1}{R_2C}}$}
\label{Fig:Impedance_Degree_One}
\end{minipage}
\end{figure}
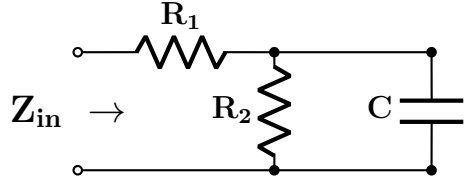
$\T$

\begin{figure}[H]
\centering
\begin{minipage}{0.44\linewidth}
{\rm
Consider four $\mathcal{HP}_{\beta}$ functions:

$\begin{matrix}
{\color{blue}f_1(s)=\frac{2}{5}+\frac{\frac{21}{10}a_1}{s+a_1}}&&
{\color{red}f_2(s)=\frac{2}{5}+\frac{\frac{14}{15}a_2}{s+a_2}}
\end{matrix}$

$\begin{matrix}
{\color{orange}f_3(s)=\frac{5}{2}-\frac{\frac{7}{4}a_3}{s+a_3}}&&
{\color{teal}f_4(s)=\frac{3}{4}-\frac{\frac{7}{12}a_4}{s+a_4}}
\end{matrix}$
\smallskip

${\color{blue}f_1(s)}$ ${\color{red}f_2(s)}$ ${\color{orange}f_3(s)}$
are with ${\scriptstyle\beta}=\frac{20}{29}$.
\smallskip

In contrast ${\color{teal}f_4(s)}$ is with 
${\scriptstyle\beta}=\frac{24}{25}$.
\smallskip

${\color{red}f_2(s)}$ and ${\color{orange}f_3(s)}$, are of the form of
$\tilde{\phi}_1(s)$ and $\tilde{\phi}_2(s)$, respectively
(see Eq. \eqref{eq:HP_Deg_One}).\\

Both ${\color{blue}f_1(s)}$ and ${\color{teal}f_4(s)}$ are {\em canonical} 

\hfill{(see Eqs. \eqref{eq:Canonical_Degree_One}
\eqref{eq:Alternative_Canonical_Degree_One}).}
}
\end{minipage}\quad\quad\begin{minipage}{0.48\linewidth}
 \begin{tikzpicture}[scale=2.1,cap=round]
	    \tikzstyle{axes}=[]
		    \tikzstyle{important line}=[very thick]
		    \tikzstyle{information text}=[rounded corners,fill=cyan!10,inner sep=1ex]
		    \begin{scope}[style=axes]
		  \
		      \draw[->] (-0.2,0) -- (2.9,0) node[above] {${\rm Real}$};
		      \draw[->] (0,-1.2) -- (0,1.2) node[right] {${\rm Imaginary}$};

		      \foreach \x/\xtext in {0.4/{\frac{2}{5}}, 0.75/{\frac{3}{4}},
		1.333333/{\frac{4}{3}},
2.5/{\frac{5}{2}}}
		\draw[xshift=\x cm] (0pt,1pt) -- (0pt,-1pt) node[below,fill=white]
			 {$\xtext$};

		      \foreach \y/\ytext in {-1.05/{-\frac{21}{20}}, 
		-0.46666/{-\frac{7}{15}}, 0.2916666/{\frac{7}{24}}, 0.875/{\frac{7}{8}}}
		       \draw[yshift=\y cm] (1pt,0pt) -- (-1pt,0pt) node[left,fill=white]
			      {$\ytext$};
		    \end{scope}
		   \draw[arrows=->,style=important line, blue] (1.45,0) circle (1.05);
		\draw[color=blue, thick] [->] (1.46,-1.05) -- (1.44,-1.05){};
		   \draw[arrows=->,style=important line, red] (13/15,0) circle (7/15);
		\draw[color=red, thick] [->] (0.87,-7/15) -- (0.85,-7/15){};
		   \draw[arrows=->,style=important line, orange] (13/8,0) circle (7/8);
		\draw[color=orange, thick] [->] (1.63,-7/8) -- (1.62,-7/8){};
		   \draw[arrows=->,style=important line, teal] (25/24,0) circle (7/24);
		\draw[color=teal, thick] [->] (1.05,-7/24) -- (1.03,-7/24){};
	 \end{tikzpicture}
\end{minipage}
\caption{
The Nyquist plots of
\mbox{
$
{\color{blue}f_1(s)=\frac{2}{5}+\frac{\frac{21}{10}a_1}{s+a_1}}\quad
{\color{red}f_2(s)=\frac{2}{5}+\frac{\frac{14}{15}a_2}{s+a_2}}\quad
{\color{orange}f_3(s)=\frac{5}{2}-\frac{\frac{7}{4}a_3}{s+a_3}}\quad
{\color{teal}f_4(s)=\frac{3}{4}-\frac{\frac{7}{12}a_4}{s+a_4}}$}
}
\label{Fig:Degree_One_HP}

The Nyquist plots are independent of the value of the parameters
$a_1$, $a_2$, $a_3$, $a_4$.
$\T$
\end{figure}
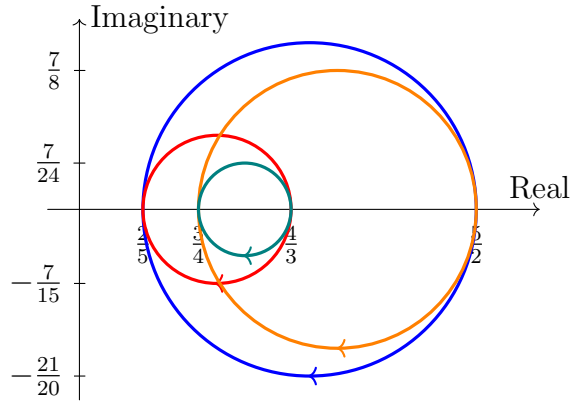

\section{A sample application: The Lurie Problem - absolute stability}
\label{Sec:Lurie_Problem}

To better motivate the study of Hyper-Positive functions, we next recall, see
\cite[Section 2]{AlpayLew2024}, that the classical ``Circle stability
criterion"\begin{footnote}{The graphical interpretation leading to the name
``The circle criterion", is beyond the scope of this work.}\end{footnote} can
be formulated in terms of a pair of $\mathcal{HP}_{\beta}$ functions. Here are
the details.
\smallskip

In control theory, the {\em Lurie problem} (a.k.a. the {\em absolute stability}
problem) is classical. For simplicity of exposition we here focus on the scalar
case. For more information, see e.g. \cite[Section 3.13]{BroLozaMasEge2020},
\cite[Section 7.1]{Khalil2000}, \cite{Popov1973}.

\begin{figure}[ht!]
\begin{tikzpicture}[scale=1.1]
\draw[color=black, thick] [->] (0,2.5) node[left]{\mbox{\boldmath${\rm In}\equiv 0$}} -- (1.15,2.5) {};
		\draw[color=black, thick] [->] (1.85,2.5) -- (3.2,2.5){};
		\draw[color=black, thick] [->] (4.2,2.5) -- (5.6,2.5){};
		\draw[color=black, thick] [->] (5.6,2.5) -- (7.0,2.5) node[right] {${\rm Out}$};
		\draw[color=black, thick] [->]  (1.5,1) -- (1.5,2.15){};
		\draw[color=black, thick] [->]  (5.6,2.5) -- (5.6,1.0){};
		\draw[color=black, thick] [->]  (3.1,1) -- (1.5,1){};
		\draw[color=black, thick] [->]  (5.6,1) -- (4.3,1){};
		\draw[color=black, thick] [-]  (3.2,2.85) -- (4.2,2.85){};
		\draw[color=black, thick] [-]  (3.2,2.15) -- (4.2,2.15){};
		\draw[color=black, thick] [-]  (3.2,2.15) -- (3.2,2.85){};
		\draw[color=black, thick] [-]  (3.2,1.35) -- (3.2,0.65){};
		\draw[color=black, thick] [-]  (4.2,2.15) -- (4.2,2.85){};
		\draw[color=black, thick] [-]  (3.1,1.45) -- (4.3,1.45){};
		\draw[color=black, thick] [-]  (3.2,1.35) -- (4.2,1.35){};
		\draw[color=black, thick] [-]  (3.1,1.45) -- (3.1,0.55){};
		\draw[color=black, thick] [-]  (4.2,1.35) -- (4.2,0.65){};
		\draw[color=black, thick] [-]  (4.3,1.45) -- (4.3,0.55){};
		\draw[color=black, thick] [-]  (3.2,0.65) -- (4.2,0.65){};
		\draw[color=black, thick] [-]  (3.1,0.55) -- (4.3,0.55){};
		\draw[color=black, thick] (1.5,2.5) circle (0.35){};
		\draw[color=black, thick] (3.7,1.0)node {\mbox{$\psi$}};
		\draw[color=black, thick] (3.7,2.5)node {\mbox{$h(s)$}};
		\draw[color=black, thick] (1.05,2.5)node[below] {\bf +};
	\draw[color=black, thick] (2.3,1.0)node[above] {\mbox{${\scriptstyle\psi({\rm t, Out})}$}};
		\draw[color=black, thick] (1.5,2.05)node[right] {\bf -};
\end{tikzpicture}
\caption{The Lurie Problem feedback loop}
\label{Fig:NonLinearFeedbackLoop}
\end{figure}
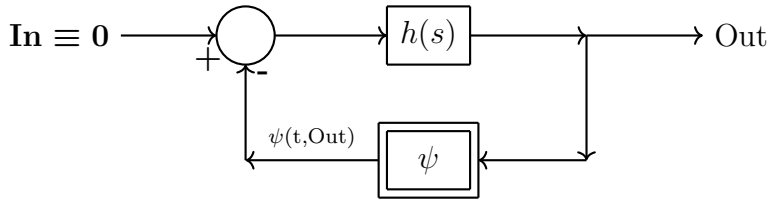
\mbox{}\smallskip

\noindent
{\bf Absolute stability problem:}~ For a given pair of real parameters
\begin{center}
$M\geq m,$
\end{center}
let
\begin{center}
$
\psi=\psi(t, {\rm Out})
$
\end{center}
be an {\em unknown time-dependent, sector-bounded, non-linearity}, satisfying
\begin{equation}\label{eq:SectorNonLinearity}
\left(M\cdot{\rm Out}-\psi\right)\left(\psi-m\cdot{\rm Out}\right)\geq 0
\quad\quad\quad\begin{smallmatrix}\forall~{\rm Out}\\~\\ \forall t\geq 0.
\end{smallmatrix}
\end{equation}
Given a feedback loop as in Figure \ref{Fig:NonLinearFeedbackLoop}. Find
conditions, based {\em only} on the rational function $h(s)$, and the constants
$M$ and $m$, so that the origin of the overall system is uniformly asymptotically
stable for any non-linearity\begin{footnote}{For $~M=m$, this reduces to a
question on stability of a~ {\em linear-time-invariant} system, so we actually
focus ourselves on the case $~M>m$.}\end{footnote}
$~\psi~$ satisfying Eq. \eqref{eq:SectorNonLinearity}.
$\T$
\smallskip

\noindent
There are several absolute stability conditions see e.g.
\cite[Theorem 5.6.3]{AnderVongpa1973} \cite[Sections 7.1]{Khalil2000},
\cite{Popov1973}, \cite[Subsection 2.3.5]{SepJanKok1996}. We here
refer only to the Circle stability criterion.
\smallskip

{\bf The Circle Stability Criterion}
\smallskip

\noindent
Consider the closed-loop system in Figure \ref{Fig:NonLinearFeedbackLoop}, along with
the condition in Eq. \eqref{eq:SectorNonLinearity} and assume that \mbox{$\infty>M
\geq m>0$}.\quad This system is absolutely stable whenever there exists $f(s)$ a
{\em canonical} $\mathcal{HP}_{\beta}$ function of degree one (of the form of
$\phi_2(s)$ as in Eq. \eqref{eq:Alternative_Canonical_Degree_One}), with
the parameters 
\begin{equation}\label{eq:Lurie}
f(s)=\frac{\begin{smallmatrix}\frac{1}{r}\end{smallmatrix}s+\begin{smallmatrix}r
\end{smallmatrix}b}{s+b}~,\quad\quad\quad\begin{smallmatrix} b&={\scriptstyle
\frac{1}{M}}\\~\\r&=\sqrt{\frac{M}{m}}\\~\\
{\color{blue}\beta}&=\frac{\sqrt{Mm}}{\frac{M+m}{2}}\end{smallmatrix}
\end{equation}
so that
\[
\underbrace{
\frac{\begin{smallmatrix}\frac{1}{r}\end{smallmatrix}h+\begin{smallmatrix}r
\end{smallmatrix}b}{h+b}}_{f(h)}
\in\mathcal{HP}_{\color{blue}\hat{\beta}}\quad\quad{\rm 
for~some}~~\begin{smallmatrix}{\color{blue}\hat{\beta}}\end{smallmatrix}\in(0,~1).
\]
$\T$

\begin{Rk}\label{Rk:Circle_Criterion}
{\rm
{\bf a.}~
Interestingly, already 45 years ago, in \cite[Eq. (9.44)]{FauCleGer1979}, the solution to
the Lurie Problem was associated with a scalar Hyper-Positive function (not under this name).
\smallskip

{\bf b.}~
In Corollary \ref{Cy:HP_beta_Order} it was stated that having
\mbox{$1>\begin{smallmatrix}{\color{teal}{\beta}_2}\end{smallmatrix}\geq
\begin{smallmatrix}{\color{blue}{\beta}_1}\end{smallmatrix}\geq 0$} implies that
\mbox{$\mathcal{HP}_{\color{teal}{\beta}_2}\subset\mathcal{HP}_{\color{blue}{\beta}_1}$.}
We now show that in the framework of {\em canonical} functions, one can say that indeed
$\mathcal{HP}_{\color{teal}{\beta}_2}$ is ``more
Lurie-stable" than $\mathcal{HP}_{\color{blue}{\beta}_1}$ 
\smallskip

Eq. \eqref{eq:Lurie} suggests that  the larger $\begin{smallmatrix}{\color{blue}\beta}
\end{smallmatrix}$
is, the greater is the sector of uncertainty (the gap between $M$ and $m$) that
the closed loop system in Figure \ref{Fig:NonLinearFeedbackLoop}, can withstand. 
\smallskip

{\bf c.}~ 
Let \mbox{${\color{blue}f_1}\in\mathcal{HP}_{\color{blue}{\beta}_1}$,} 
\mbox{${\color{teal}f_2}\in\mathcal{HP}_{\color{teal}{\beta}_2}$,} be a pair of
scalar {\em canonical} functions, where \mbox{$1>\begin{smallmatrix}{\color{teal}{\beta}_2}
\end{smallmatrix}\geq\begin{smallmatrix}{\color{blue}{\beta}_1}\end{smallmatrix}\geq 0$.} 
From Figure \ref{Fig:Degree_One_HP}, it is clear that the Nyquist plot of
${\color{teal}f_2(s)}$ is contained in that of ${\color{blue}f_1(s)}$.
\smallskip

Thus one can say that if in $\C_R$, a Nyquist plot of a function is contained in that
of another one, it is ``more stable". In fact, already more than forty years ago, in
\cite{Reza1984}, F.M. Reza called it ``power dominance". Adopting this point of view,
following Figures \ref{Fig:Degree_One_HP}, \ref{Figure:Convex_Two_Canonical_1},
\ref{Figure:Convex_Two_Canonical_3} one can say if $f_1$ and $f_2$ belong to the same
$\mathcal{HP}_{\color{blue}\beta}$ and $f_1$ is {\em canonical}, then
$f_2(s)$ is ``more stable".
}
$\T$
\end{Rk}

\section{Parametrization of {\em Canonical} $\mathcal{HP}_{\color{blue}T}$
Rational Functions}
\label{Sec:B_and_HP_T}
\setcounter{equation}{0}

In this section, we construct, in stages, a parametrization of all
$m\times m$-valued {\em canonical} $\mathcal{HP}_{\color{blue}T}$ functions with 
\mbox{$I_m\succ{\color{blue}T}\succ 0$.} To gain intuition, we start with scalars.

\subsection{An analogy with $\C$}
To proceed, we need to resort to the classical Cayley transform.

\begin{Dn}\label{Dn:Cayley_Transform}
{\rm
We denote by $\mathcal{C}(A)$ the Cayley transform of a matrix
$A\in\C^{n\times n}$, where \mbox{$-1\not\in{\rm spec}(A)$},
\[
\mathcal{C}\left(A\right):=\left(I_n-A\right)\left(I_n+A\right)^{-1}=
-I_n+2\left(I_n+A\right)^{-1}.
\]
$\T$
}
\end{Dn}

Recall that the Cayley transform is involutive in the sense that,
whenever well defined,
\[
\mathcal{C}\left(\mathcal{C}\left(A\right)\right)=A.
\]
To gain intuition, we first examine disks in $\mathbb{C}$ of the the form
\mbox{$\mathbb{D}(\begin{smallmatrix}{\rm Center}\end{smallmatrix},~
\begin{smallmatrix}{\rm Radius}\end{smallmatrix})$.} Recall that the Cayley
transform forms a bijection between \mbox{$\mathbb{D}(0,~1)$,} the open unit disk,
and $\mathbb{C}_R$, the open right-half of the complex plane.\\
To refine the analysis, we examine disks of the form
\[
\mathbb{D}_{\rm Center}(\begin{smallmatrix}{\color{blue}\beta}\end{smallmatrix})
:=\mathbb{D}\left(0+i0,~{\scriptstyle\frac{\sqrt{1-{\color{blue}\beta}}}
{\sqrt{1+{\color{blue}\beta}}}}\right)\quad\quad
\begin{smallmatrix}{\color{blue}\beta}\end{smallmatrix}\in[0,~1),
\]
which are illustrated on the left-hand side of Figure \ref{Fig:Sub-Unit_Disk}.
Clearly, for
\mbox{$\begin{smallmatrix}{\color{blue}\beta}\end{smallmatrix}=0$,} the unit
disk is recovered.\quad 
\smallskip

Consider now the the Cayley transform of the above sub-unit disks,
\begin{equation}\label{eq:Caylaey_D_center}
\mathcal{C}\left(\mathbb{D}_{\rm Center}(\begin{smallmatrix}{\color{blue}\beta}
\end{smallmatrix})\right)=\mathbb{D}_{\rm Inv}(\begin{smallmatrix}{\color{blue}\beta}
\end{smallmatrix}):=\mathbb{D}\left({\scriptstyle\frac{1}{\color{blue}\beta}}+i0,~
{\scriptstyle\frac{\sqrt{1-{\color{blue}\beta}^2}}{\color{blue}\beta}}\right)
\quad\quad\begin{smallmatrix}{\color{blue}\beta}\end{smallmatrix}\in[0,~1).
\end{equation}
Expectedly, $\mathbb{C}_R$ is recovered for \mbox{${\scriptstyle\beta}=0$.~} It
turns out that under inversion, each \mbox{$\mathbb{D}_{\rm Inv}(\begin{smallmatrix}
{\color{blue}\beta}\end{smallmatrix})$}
disk is mapped {\em onto} itself (and hence the subscript ``Inv").\\
In Figure  \ref{Fig:Sub-Unit_Disk}, these disks, along with their image under the
Cayley transform, are illustrated, where the color is preserved.

\begin{figure}[h]
\begin{minipage}[b]{0.38\linewidth}
\begin{tikzpicture}[scale=3.2,cap=round]
\tikzstyle{axes}=[]
\tikzstyle{important line}=[very thick]
\tikzstyle{information text}=[rounded corners,fill=red!10,inner sep=1ex]
\begin{scope}[style=axes]
		  \
		    \draw[->] (-1.10,0) -- (0.9,0) node[above] {Re};
		      \draw[->] (0,-0.85) -- (0,0.85) node[right] {Im};

		      \foreach \x/\xtext in{
		   -1/{\mbox{\boldmath$-{\scriptstyle 1}$}},
		  -0.666666/{\mbox{\boldmath$-{\scriptstyle\frac{2}{3}}$}},
		      -0.2/{\mbox{\boldmath$-{\scriptstyle\frac{1}{5}}$}}, 
			0.2/{\mbox{\boldmath${\scriptstyle\frac{1}{5}}$}}, 
			0.5/{\mbox{\boldmath${\scriptstyle\frac{1}{2}}$}},
		   0.75/{\mbox{\boldmath${\scriptstyle\frac{3}{4}}$}}
		}
		\draw[xshift=\x cm] (0pt,1pt) -- (0pt,-1pt) node[below,fill=white]
      {$\xtext$}; 

		      \foreach \y/\ytext in{
		   -0.75/{\mbox{\boldmath$-{\scriptstyle\frac{3}{4}}$}},
			-0.5/{\mbox{\boldmath$-{\scriptstyle\frac{1}{2}}$}},
			 0.2/{\mbox{\boldmath${\scriptstyle\frac{1}{5}}$}}, 
		    0.666666/{\mbox{\boldmath${\scriptstyle\frac{2}{3}}$}}
		}
		       \draw[yshift=\y cm] (1pt,0pt) -- (-1pt,0pt) node[left,fill=white]
		      {$\ytext$};
		    \end{scope}
		   \draw[arrows=->,style=important line, cyan] (0,0) circle (0.75);
		   \draw[arrows=->,style=important line, blue] (0,0) circle (0.666666666);
		   \draw[arrows=->,style=important line, red] (0,0) circle (0.5);
		   \draw[arrows=->,style=important line, green] (0,0) circle (0.2);
	\end{tikzpicture}
	\begin{center}
	$\begin{matrix}{\rm sub-unit~disk}\\ \mathbb{D}_{\rm Center}({\scriptstyle\beta}):=
	\mathbb{D}\left(0+i0,~{\scriptstyle\frac{\sqrt{1-\beta}}{\sqrt{1+\beta}}}\right)\end{matrix}$
	\end{center}
	\end{minipage}
		\quad\quad\quad
	\begin{minipage}[b]{0.42\linewidth}
	\begin{tikzpicture}[scale=0.67,cap=round]
	    \tikzstyle{axes}=[]
	    \tikzstyle{important line}=[very thick]
	    \tikzstyle{information text}=[rounded corners,fill=red!10,inner sep=1ex]
	    \begin{scope}[style=axes]
	  \
	      \draw[->] (-0.3,0) -- (7.8,0) node[above] {Re};
	      \draw[->] (0,-3.5) -- (0,4.1) node[right] {Im};
	     \foreach \x/\xtext in {
	       1.5/{\mbox{\boldmath${\scriptstyle\frac{3}{2}}$}}, 
		3/{\mbox{\boldmath${\scriptstyle 3}$}},
		4/{\mbox{\boldmath${\scriptstyle 4}$}},
		5/{\mbox{\boldmath${\scriptstyle 5}$}},
		6/{\mbox{\boldmath${\scriptstyle 6}$}},
		7/{\mbox{\boldmath${\scriptstyle 7}$}}
		}
	\draw[xshift=\x cm] (0pt,1pt) -- (0pt,-1pt) node[below,fill=white]
		      {$\xtext$}; 

		     \foreach \y/\ytext in {
		   -3.43/{\mbox{\boldmath$-{\scriptstyle\frac{24}{7}}$}},
		-1.33/{\mbox{\boldmath$-{\scriptstyle\frac{4}{3}}$}},
		     0.41666/{\mbox{\boldmath${\scriptstyle\frac{5}{12}}$}},
		   2.4/{\mbox{\boldmath${\scriptstyle\frac{12}{5}}$}}
		}
		       \draw[yshift=\y cm] (1pt,0pt) -- (-1pt,0pt) node[left,fill=white]
			      {$\ytext$};
	 \end{scope}
	   \draw[arrows=->,style=important line, cyan] (25/7,0) circle (24/7);
	   \draw[arrows=->,style=important line, blue] (2.6,0) circle (2.4);
	   \draw[arrows=->,style=important line, red] (5/3,0) circle (4/3);
	   \draw[arrows=->,style=important line, green] (13/12,0) circle (5/12);
	   \draw[dashed, gray] (0,0) circle (1);
	\end{tikzpicture}
\begin{center}
$\begin{matrix}{\rm inverible~disk}\\
\mathbb{D}_{\rm Inv}({\scriptstyle\beta}):=\mathbb{D}\left({\scriptstyle\frac{1}{\beta}}+i0,~
{\scriptstyle\frac{\sqrt{1-{\beta}^2}}{\beta}}\right)\end{matrix}$
\end{center}
\end{minipage}
\smallskip

\begin{center}
${\scriptstyle\color{cyan}\beta}=\frac{7}{25}\quad\quad\quad{\scriptstyle\color{blue}
\beta}=\frac{5}{13}\quad\quad\quad{\scriptstyle\color{red}\beta}=\frac{3}{5}
\quad\quad\quad{\scriptstyle\color{green}\beta}=\frac{12}{13}$
\end{center}
\caption{Sub-Unit Disks and their Image under the Cayley Transform (color is preserved).}
\label{Fig:Sub-Unit_Disk}
\end{figure}
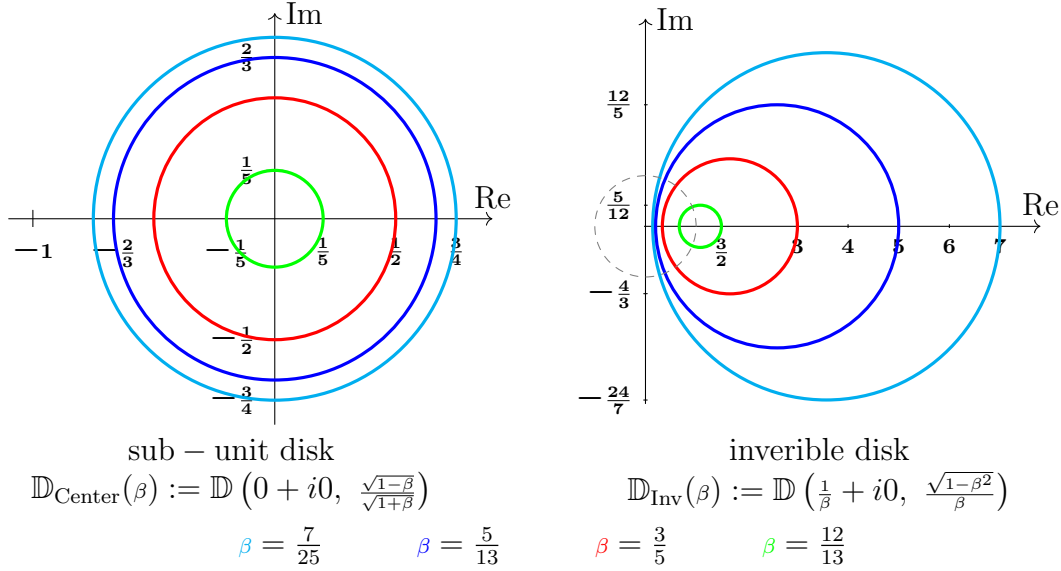
\smallskip

From Eq. \eqref{eq:Caylaey_D_center} one has that
\[
\mathcal{C}\left(\mathbb{D}_{\rm Center}(\begin{smallmatrix}{\color{blue}\beta}
\end{smallmatrix})\right)=\mathbb{D}_{\rm Inv}(\begin{smallmatrix}{\color{blue}\beta}
\end{smallmatrix})\quad\quad
\begin{smallmatrix}{\color{blue}\beta}\end{smallmatrix}\in[0,~1).
\]
From Figure \ref{Fig:Sub-Unit_Disk} it is straightforward to see that for
\mbox{$\begin{smallmatrix}{\color{blue}\beta}\end{smallmatrix}>0$,} beyond the Cayley
transform, there is a pair of additional maps\begin{footnote}{
We have exploited the fact that 
\mbox{$\mathbb{D}_{\rm Center}({\scriptstyle\color{blue}\beta})=e^{i\theta}
\mathbb{D}_{\rm Center}({\scriptstyle\color{blue}\beta})$} for all
\mbox{$\begin{smallmatrix}\theta\end{smallmatrix}\in[0,~2\pi)$} and in particular,
\mbox{$\mathbb{D}_{\rm Center}({\scriptstyle\color{blue}\beta})=-
\mathbb{D}_{\rm Center}({\scriptstyle\color{blue}\beta})$}.
}\end{footnote} between these disks,
\begin{equation}\label{eq:Direct_Map}
\begin{matrix}
\mathbb{D}_{\rm Inv}({\scriptstyle\color{blue}\beta})&=&{\scriptstyle\frac{1}
{\color{blue}\beta}}\pm\left(1+{\scriptstyle\frac{1}{\color{blue}\beta}}\right)
\mathbb{D}_{\rm Center}({\scriptstyle\color{blue}\beta})\\~\\
\mathbb{D}_{\rm Center}({\scriptstyle\color{blue}\beta})&=&\pm
{\scriptstyle\frac{\color{blue}\beta}{1+{\color{blue}\beta}}}
(\mathbb{D}_{\rm Inv}({\scriptstyle\color{blue}\beta})-{\scriptstyle\frac{1}
{\color{blue}\beta}})
\end{matrix}
\quad\quad\quad\begin{smallmatrix}{\color{blue}\beta}\end{smallmatrix}\in[0,~1).
\end{equation}
In contrast to the Cayley Transform, the formulation in Eq. \eqref{eq:Direct_Map} has the
advantage of being {\em affine}-linear. In the rest of this section, this is extended to
matrix-valued functions, and then exploited.

\subsection{Hyper-Bounded functions}

We now find it convenient to recall the family $\mathcal{B}$ of Bounded real
functions\begin{footnote}{To ease the reading we typically denote by $F(s)$
(~$G(s)$~) functions analytic in $\C_R$ (unit disk).}\end{footnote},
\begin{equation}\label{eq:Def_B}
\mathcal{B}=\{ G(s):~I_m-{G(s)}^*G(s)\succcurlyeq 0\quad\quad\quad
\forall s\in\C_R~\},
\end{equation}
and the corresponding quadratic form is
\begin{equation}\label{eq:Quad_Def_B}
\mathcal{B}=\left\{~\begin{smallmatrix}G(s)\end{smallmatrix}~:~
\left(\begin{smallmatrix}G(s)\\~\\I_m\end{smallmatrix}\right)^*
\left(\begin{smallmatrix}-I_m&&0\\~\\0&&I_m\end{smallmatrix}\right)
\left(\begin{smallmatrix}G(s)\\~\\I_m\end{smallmatrix}\right)
\in\overline{\mathbf P}_m\quad\forall s\in\C_R~\right\}.
\end{equation}
In particular, $G_o(s)$ is said to be a {\em canonical} Bounded function whenever
\begin{equation}\label{eq:Def_Canonical_B}
\left(\begin{smallmatrix}G_o(s)\\~\\I_m\end{smallmatrix}\right)^*
\left(\begin{smallmatrix}-I_m&&0\\~\\0&&I_m\end{smallmatrix}\right)
\left(\begin{smallmatrix}G_o(s)\\~\\I_m\end{smallmatrix}\right)~
\left\{\begin{smallmatrix}\in\overline{\mathbf P}_m&&\forall s\in\C_R\\~\\=0&&
\forall s\in{i}\R.\end{smallmatrix}\right.
\end{equation}
\smallskip

As before, for an arbitrary prescribed
\mbox{$I_m\succ{\color{blue}T}\succcurlyeq 0$,} one can
define the subset of ${\color{blue}T}$-Hyper-Bounded functions as,
\begin{equation}\label{eq:Quad_Def_HB_T}
\mathcal{HB}_{\color{blue}T}=\left\{~\begin{smallmatrix}G(s)\end{smallmatrix}~:~
\left(\begin{smallmatrix}G(s)\\~\\I_m\end{smallmatrix}\right)^*
\left(\begin{smallmatrix}-(I_m+{\color{blue}T})&&0\\~\\0&&I_m+{\color{blue}T}
\end{smallmatrix}\right)
\left(\begin{smallmatrix}G(s)\\~\\I_m\end{smallmatrix}\right)
\in\overline{\mathbf P}_m\quad\forall s\in\C_R~\right\}.
\end{equation}
Finally, {\em canonical} $\mathcal{HB}_{\color{blue}T}$ are described as
functions $G(s)$ satisfying,
\begin{equation}\label{eq:Canonical_Quad_Def_HB_T}
\left(\begin{smallmatrix}G(s)\\~\\I_m\end{smallmatrix}\right)^*
\left(\begin{smallmatrix}-(I_m+{\color{blue}T})&&0\\~\\0&&I_m+{\color{blue}T}
\end{smallmatrix}\right)\left(\begin{smallmatrix}G(s)\\~\\I_m\end{smallmatrix}
\right)~\left\{\begin{smallmatrix}\in\overline{\mathbf P}_m&&\forall
s\in\C_R\\~\\=0&&\forall s\in{i}\R.\end{smallmatrix}\right.
\end{equation}
The above quadratic forms will turn to be useful in the sequel.
\smallskip

We now return to the Cayley transform, in the framework of
\mbox{$m\times m$-valued} rational functions, where it takes the form
\begin{equation}\label{eq:Cayley_F}
G(s)=\mathcal{C}\left(F(s)\right):=(I_m-F(s))(I_m+F(s))^{-1}\quad\quad
{\rm det}(I_m+F(s))\not\equiv 0.
\end{equation}
We next list four pairs of functions, related through the Cayley transform.

\begin{Pn}\label{Pn:Cayley_Functions}
For arbitrary \mbox{$I_m\succ{\color{blue}T}\succcurlyeq 0$}
(including \mbox{${\color{blue}T}=0$),}
\[
\begin{matrix}
\mathcal{C}\left(\mathcal{P}\right)&=&\mathcal{B}\\~\\ 
\mathcal{C}\left(\mathcal{PO}\right)&=&{\rm canonical}~\mathcal{B}\\~\\ 
\mathcal{C}\left(\mathcal{HP}_{
\color{blue}T}\right)&=&\mathcal{HB}_{\color{blue}T}\\~\\ 
\mathcal{C}\left({\rm canonical}~\mathcal{HP}_{
\color{blue}T}\right)&=&{\rm canonical}~\mathcal{HB}_{\color{blue}T}
\end{matrix}
\]
\end{Pn}

The first relation is classical,
see e.g. \cite[Example 2.7.1]{AnderVongpa1973},
\cite[Eq. (44), Section 6]{Belev1968}.\\
\smallskip

{\bf Proof :}~
Consider the following pairs of function,
\[
\begin{array}{ccc}
\mathcal{P}, \mathcal{B}~ {\rm in~ Eqs.~ \eqref{eq:Quad_Def_P},~
\eqref{eq:Quad_Def_B} } &~&\mathcal{PO},~ {\rm canonical}~ \mathcal{B}~
{\rm in~ Eqs.~\eqref{eq:Quad_Def_PO},~\eqref{eq:Def_Canonical_B}}\\~\\
\mathcal{HP}_{\color{blue}T},~\mathcal{HB}_{\color{blue}T}~
{\rm in~ Eqs.~\eqref{eq:Quadratic_HP_Delta},~\eqref{eq:Quad_Def_HB_T}}
&~&{\rm canonical}~\mathcal{HP}_{\color{blue}T},~{\rm canonical}~ 
\mathcal{HB}_{\color{blue}T}~{\rm in~Eqs.~\eqref{eq:Def_Canoinal_Quad_HP_W},~
\eqref{eq:Canonical_Quad_Def_HB_T}}
\end{array}
\]
\smallskip

All four relations follow from the description of the Cayley transform of
quadratic form. Namely, if one denotes
\[
\left(\begin{smallmatrix}F\\I\end{smallmatrix}\right)^*W
\left(\begin{smallmatrix}F\\I\end{smallmatrix}\right)
\quad\quad{\rm and}\quad\quad
\left(\begin{smallmatrix}G\\I\end{smallmatrix}\right)^*V
\left(\begin{smallmatrix}G\\I\end{smallmatrix}\right),
\]
then
\[
\begin{matrix}
\left(\begin{smallmatrix}F\\I\end{smallmatrix}\right)^*W
\left(\begin{smallmatrix}F\\I\end{smallmatrix}\right)
&=&
\left(\begin{smallmatrix}\mathcal{C}(G)\\I\end{smallmatrix}\right)^*W
\left(\begin{smallmatrix}\mathcal{C}(G)\\I\end{smallmatrix}\right)
\\~\\&=&
\left(\begin{smallmatrix}(I-G)(I+G)^{-1}\\I_m\end{smallmatrix}\right)^*W
\left(\begin{smallmatrix}(I-G)(I+G)^{-1}\\I_m\end{smallmatrix}\right)
\\~\\&=&
\begin{smallmatrix}(I+G^*)^{-1}\end{smallmatrix}
\left(\begin{smallmatrix}I-G\\I+G\end{smallmatrix}\right)^*W
\left(\begin{smallmatrix}I-G\\I+G\end{smallmatrix}\right)
\begin{smallmatrix}(I+G)^{-1}\end{smallmatrix}
\\~\\&=&
\begin{smallmatrix}(I+G^*)^{-1}\end{smallmatrix}
\underbrace{\left(\begin{smallmatrix}G\\I\end{smallmatrix}\right)^*
\left(\begin{smallmatrix}-I&I\\~I&I\end{smallmatrix}\right)}_{
\left(\begin{smallmatrix}I-G\\I+G\end{smallmatrix}\right)^*}W
\underbrace{\left(\begin{smallmatrix}-I&I\\~I&I\end{smallmatrix}\right)
\left(\begin{smallmatrix}G\\I\end{smallmatrix}\right)}_{
\left(\begin{smallmatrix}I-G\\I+G\end{smallmatrix}\right)}
\begin{smallmatrix}(I+G)^{-1}\end{smallmatrix}\\~\\&=&
\begin{smallmatrix}(I+G^*)^{-1}\end{smallmatrix}
\left(\begin{smallmatrix}G\\I\end{smallmatrix}\right)^*
\underbrace{\left(\begin{smallmatrix}-I&I\\~I&I\end{smallmatrix}\right)W
\left(\begin{smallmatrix}-I&I\\~I&I\end{smallmatrix}\right)}_V
\left(\begin{smallmatrix}G\\I\end{smallmatrix}\right)
\begin{smallmatrix}(I+G)^{-1}\end{smallmatrix}.
\end{matrix}
\]
Namely in all four relations,
\[
V=\left(\begin{smallmatrix}-I&I\\~I&I\end{smallmatrix}\right)W
\left(\begin{smallmatrix}-I&I\\~I&I\end{smallmatrix}\right),
\]
and the proof is complete.
\qed
\smallskip

We next extend the affine-linear maps from Eq. \eqref{eq:Direct_Map} to rational
functions. We start with an example of a scalar function of degree one. First, 
\[
g_o(s)=\begin{smallmatrix}\frac{s-a}{s+a}\end{smallmatrix}\quad\quad a>0,
\]
is a {\em canonical} $\mathcal{B}$ function. Next,
\[
{\color{red}g(s)}=\begin{smallmatrix}\frac{\sqrt{1-{\color{blue}\beta}}}
{\sqrt{1+{\color{blue}\beta}}}\end{smallmatrix}
\underbrace{\frac{s-a}{s+a}}_{g_o(s)}\quad\quad
\begin{smallmatrix}\beta\in(0,~1)\\~\\a>0,\end{smallmatrix}
\]
is a {\em canonical} $\mathcal{HB}_{\color{blue}\beta}$ function, i.e.
a {\em canonical} $\mathcal{HB}_{\color{blue}T}$ function with
\mbox{${\color{blue}T}=\begin{smallmatrix}{\color{blue}\beta}\end{smallmatrix}
I_m~$.} Now, with this ${\color{red}g(s)}$ one can associate three
{\em canonical} $\mathcal{HP}_{\color{blue}\beta}$ functions:
\[
\begin{matrix}
f(s)&=&\mathcal{C}({\color{red}g(s)})&=&{\scriptstyle\frac{1}{\beta}}+
{\scriptstyle\frac{\sqrt{1-{\beta}^2}}{\beta}}\frac{s-b}{s+b}&~&b:=
\begin{smallmatrix}\frac{1}{\beta}\left(1+\sqrt{1-{\beta}^2}\right)
\end{smallmatrix}a\\~\\{\phi}_1(s)&=&{\scriptstyle\frac{1}{\beta}}+
{\scriptstyle\frac{1+\beta}{\beta}}{\color{red}g(s)}&=&
{\scriptstyle\frac{1}{\beta}}+{\scriptstyle\frac{\sqrt{1-{\beta}^2}}{\beta}}
\frac{s-a}{s+a}&~&~\\~\\{\phi}_2(s)&=&{\scriptstyle\frac{1}{\beta}}-
{\scriptstyle\frac{1+\beta}{\beta}}{\color{red}g(s)}&=&
{\scriptstyle\frac{1}{\beta}}-{\scriptstyle\frac{\sqrt{1-{\beta}^2}}{\beta}}
\frac{s-a}{s+a}~,&~&~\end{matrix}
\]
where ${\phi}_1(s)$, ${\phi}_2(s)$ are as in Eq. \eqref{eq:Canonical_Degree_One}.
These four functions are illustrated in Figure \ref{Fig:Two_Maps}: The Nyquist
plots are quite similar, but the functions are indeed different.

\begin{figure}[H]
\centering
\begin{minipage}[b]{0.38\linewidth}
	 \begin{tikzpicture}[scale=2.4,cap=round]
		    \tikzstyle{axes}=[]
		    \tikzstyle{important line}=[very thick]
		    \tikzstyle{information text}=[rounded corners,fill=red!10,inner sep=1ex]
		    \begin{scope}[style=axes]
		  \
		      \draw[->] (-0.7,0) -- (0.7,0) node[above] {Re};
		      \draw[->] (0,-0.7) -- (0,0.7) node[right] {Im};

		      \foreach \x/\xtext in{
			-0.5/{\mbox{\boldmath$-{\scriptstyle\frac{\sqrt{1-\beta}}{\sqrt{1+\beta}}}$}}, 
			 0.5/{\mbox{\boldmath${\scriptstyle\frac{\sqrt{1-\beta}}{\sqrt{1+\beta}}}$}}
		}
		\draw[xshift=\x cm] (0pt,1pt) -- (0pt,-1pt) node[below,fill=white]
		      {$\xtext$}; 

		      \foreach \y/\ytext in{
			 0.5/{\mbox{\boldmath${\scriptstyle\frac{\sqrt{1-\beta}}{\sqrt{1+\beta}}}$}}
		}
		       \draw[yshift=\y cm] (1pt,0pt) -- (-1pt,0pt) node[left,fill=white]
	     {$\ytext$};
		    \end{scope}
		   \draw[arrows=->,style=important line, red] (0,0) circle (0.5);
		\draw[color=red, thick] [->] (-0.04,0.5) -- (0.04,0.5){};
		 \end{tikzpicture}
		\begin{center}
		$\begin{matrix}
		{\color{red}g(s)}=
		\begin{smallmatrix}\frac{\sqrt{1-\beta}}{\sqrt{1+\beta}}\end{smallmatrix}\frac{s-a}{s+a}
		\\~\\
		\end{matrix}$
		\end{center}
		\end{minipage}
		\quad\quad\quad
		\begin{minipage}[b]{0.48\linewidth}
		\begin{tikzpicture}[scale=0.85,cap=round]
		    \tikzstyle{axes}=[]
		    \tikzstyle{important line}=[very thick]
		    \tikzstyle{information text}=[rounded corners,fill=red!10,inner sep=1ex]
		    \begin{scope}[style=axes]
		  \
		      \draw[->] (-0.1,0) -- (3.8,0) node[above] {Re};
		      \draw[->] (0,-1.5) -- (0,1.8) node[right] {Im};

		      \foreach \x/\xtext in {
		       1.6666666/{\mbox{\boldmath${\scriptstyle\frac{1}{\beta}}$}} 
		}
		\draw[xshift=\x cm] (0pt,1pt) -- (0pt,-1pt) node[below,fill=white]
			      {$\xtext$}; 

		      \foreach \y/\ytext in {
		   -1.333333/{\mbox{\boldmath$-{\scriptstyle\frac{\sqrt{1-{\beta}^2}}{\beta}}$}},
		     1.33333/{\mbox{\boldmath${\scriptstyle\frac{\sqrt{1-{\beta}^2}}{\beta}}$}}
		}
		       \draw[yshift=\y cm] (1pt,0pt) -- (-1pt,0pt) node[left,fill=white]
		      {$\ytext$};
		    \end{scope}
		   \draw[arrows=->,style=important line, blue] (5/3,0) circle (4/3);
		\draw[color=blue, thick] [->] (1.66,4/3) -- (1.67,4/3){};
		\end{tikzpicture}
		\begin{center}
		$
		f(s)=\mathcal{C}({\color{red}g(s)})={\scriptstyle\frac{1}{\color{blue}\beta}}+
		{\scriptstyle\frac{\sqrt{1-{\color{blue}\beta}^2}}{\color{blue}\beta}}\frac{s-b}{s+b}
		$
		\end{center}
		\end{minipage}

		\begin{minipage}[b]{0.45\linewidth}
		\begin{tikzpicture}[scale=0.85,cap=round]
		    \tikzstyle{axes}=[]
		    \tikzstyle{important line}=[very thick]
		    \tikzstyle{information text}=[rounded corners,fill=red!10,inner sep=1ex]
		    \begin{scope}[style=axes]
		  \
		      \draw[->] (-0.1,0) -- (3.8,0) node[above] {Re};
		      \draw[->] (0,-1.5) -- (0,1.8) node[right] {Im};

		      \foreach \x/\xtext in {
		       1.6666666/{\mbox{\boldmath${\scriptstyle\frac{1}{\beta}}$}} 
		}
		\draw[xshift=\x cm] (0pt,1pt) -- (0pt,-1pt) node[below,fill=white]
	      {$\xtext$}; 

		      \foreach \y/\ytext in {
		    -1.33333/{\mbox{\boldmath$-{\scriptstyle\frac{\sqrt{1-{\beta}^2}}{\beta}}$}},
		     1.33333/{\mbox{\boldmath${\scriptstyle\frac{\sqrt{1-{\beta}^2}}{\beta}}$}}
		}
		       \draw[yshift=\y cm] (1pt,0pt) -- (-1pt,0pt) node[left,fill=white]
		      {$\ytext$};
		    \end{scope}
		   \draw[arrows=->,style=important line, blue] (5/3,0) circle (4/3);
		\draw[color=blue, thick] [->] (1.66,4/3) -- (1.67,4/3){};
\end{tikzpicture}
\begin{center}
$\begin{smallmatrix}\frac{1}{\color{blue}\beta}\end{smallmatrix}+{\scriptstyle\frac{1+{\color{blue}
\beta}}{\color{blue}\beta}}{\color{red}g(s)}={\scriptstyle\frac{1}{\color{blue}\beta}}+
{\scriptstyle\frac{\sqrt{1-{\color{blue}\beta}^2}}{\color{blue}\beta}}\frac{s-a}{s+a}$
\end{center}
\end{minipage}\quad\quad\quad\begin{minipage}[b]{0.45\linewidth}
		\begin{tikzpicture}[scale=0.85,cap=round]
		    \tikzstyle{axes}=[]
		    \tikzstyle{important line}=[very thick]
		    \tikzstyle{information text}=[rounded corners,fill=red!10,inner sep=1ex]
		    \begin{scope}[style=axes]
		  \
		      \draw[->] (-0.1,0) -- (3.8,0) node[above] {Re};
		      \draw[->] (0,-1.5) -- (0,1.8) node[right] {Im};

		      \foreach \x/\xtext in {
		       1.6666666/{\mbox{\boldmath${\scriptstyle\frac{1}{\beta}}$}} 
		}
		\draw[xshift=\x cm] (0pt,1pt) -- (0pt,-1pt) node[below,fill=white]
			      {$\xtext$}; 

		      \foreach \y/\ytext in {
		    -1.33333/{\mbox{\boldmath$-{\scriptstyle\frac{\sqrt{1-{\beta}^2}}{\beta}}$}},
		     1.33333/{\mbox{\boldmath${\scriptstyle\frac{\sqrt{1-{\beta}^2}}{\beta}}$}}
		}
		       \draw[yshift=\y cm] (1pt,0pt) -- (-1pt,0pt) node[left,fill=white]
			      {$\ytext$};
		    \end{scope}
		   \draw[arrows=->,style=important line, blue] (5/3,0) circle (4/3);
		\draw[color=blue, thick] [->] (1.67,4/3) -- (1.66,4/3){};
		\end{tikzpicture}
\begin{center}
$\begin{smallmatrix}\frac{1}{\color{blue}\beta}\end{smallmatrix}-
{\scriptstyle\frac{1+{\color{blue}\beta}}{\color{blue}\beta}}{\color{red}g(s)}
={\scriptstyle\frac{1}{\color{blue}\beta}}-{\scriptstyle\frac
{\sqrt{1-{\color{blue}\beta}^2}}{\color{blue}\beta}}\frac{s-a}{s+a}$
\end{center}
\end{minipage}
\caption{Three {\em canonical} $\mathcal{HP}_{\color{blue}\beta}$ associated
with the same {\em canonical} $\mathcal{HB}_{\color{blue}\beta}$ function,
${\color{red}g(s)}$.
}
\label{Fig:Two_Maps}
$\T$
\end{figure}
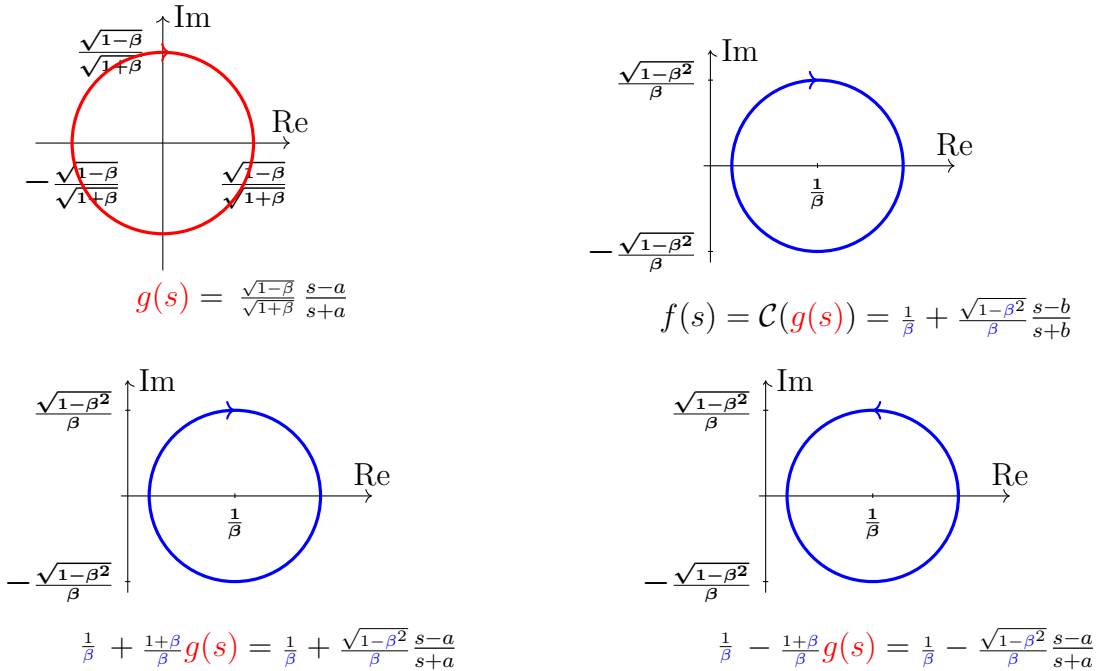
\smallskip

Naturally, the above discussion also applies to the matrix-valued case:\\
For \mbox{$I_m\succ{\color{blue}T}\succcurlyeq 0$,} let ${\color{red}G(s)}$ be a
{\em canonical} $\mathcal{HB}_{\color{blue}T}$ function as in Eq.
\eqref{eq:Canonical_Quad_Def_HB_T}. By the last item of Proposition
\ref{Pn:Cayley_Functions}, whenever
\mbox{$F(s)=\mathcal{C}\left({\color{red}G(s)}\right)$,} it
is a {\em canonical} $\mathcal{HP}_{\color{blue}T}$ function.
\smallskip

Now when \mbox{${\color{blue}T}\succ 0$,} i.e. non-singular, from the same
${\color{red}G(s)}$, two additional {\em canonical } $\mathcal{HP}_{\color{blue}T}$
functions ${\Phi}_1(s)$ and ${\Phi}_2(s)$ can be obtained, 
\begin{equation}\label{eq:Map_HP_HB}
\begin{matrix}{\Phi}_1(s)&=&\begin{smallmatrix}{\color{blue}T}^{-1}+
(I_m+{\color{blue}T}^{-1})^{\frac{1}{2}}\end{smallmatrix}{\color{red}G(s)}
\begin{smallmatrix}(I+{\color{blue}T}^{-1})^{\frac{1}{2}}\end{smallmatrix}
\\~\\{\Phi}_2(s)&=&\begin{smallmatrix}{\color{blue}T}^{-1}-(I_m+
{\color{blue}T}^{-1})^{\frac{1}{2}}\end{smallmatrix}{\color{red}G(s)}
\begin{smallmatrix}(I+{\color{blue}T}^{-1})^{\frac{1}{2}}
\end{smallmatrix}.\end{matrix}
\end{equation}
Indeed, substituting these ${\Phi}_1$ and ${\Phi}_2$  in Eq. \eqref{eq:Canonical_HP_T}
reveals that ${\color{red}G(s)}$ satisfies Eq. \eqref{eq:Canonical_Quad_Def_HB_T}.
\smallskip

For the converse direction note that in both cases
\[
\begin{matrix}
{\color{red}G(s)}
&=&
\begin{smallmatrix}(I_m+{\color{blue}T}^{-1})^{-\frac{1}{2}}\end{smallmatrix}
\left({\Phi}_1(s)-\begin{smallmatrix}{\color{blue}T}^{-1}\end{smallmatrix}\right)
\begin{smallmatrix}(I_m+{\color{blue}T}^{-1})^{-\frac{1}{2}}\end{smallmatrix}
\\~\\
{\color{red}G(s)}
&=& 
\begin{smallmatrix}(I_m+{\color{blue}T}^{-1})^{-\frac{1}{2}}\end{smallmatrix}
\left(\begin{smallmatrix}{\color{blue}T}^{-1}\end{smallmatrix}-{\Phi}_2(s)\right)
\begin{smallmatrix}(I_m+{\color{blue}T}^{-1})^{-\frac{1}{2}}\end{smallmatrix},
\end{matrix}
\]
${\color{red}G(s)}$ is a {\em canonical} $\mathcal{HB}_{\color{blue}T}$ function.
\smallskip

An affine-linear map, quite similar to the one in Eq. \eqref{eq:Map_HP_HB}, will
be exploited in the next subsection.

\subsection{Parametrization of {\em Canonical} $\mathcal{HP}_{\color{blue}T}$ 
Rational Functions}
\label{Subsec:B_and_HP_T}

In Eq. \eqref{eq:Map_HP_HB} we presented an {\em affine}-linear map between {\em canonical}
$\mathcal{HP}_{\color{blue}T}$ and {\em canonical} $\mathcal{HB}_{\color{blue}T}$
functions. In the current subsection, we modify it to an {\em affine}-linear map between
{\em canonical} $\mathcal{HP}_{\color{blue}T}$ and {\em canonical} $\mathcal{B}$
functions. Here are the details.
\smallskip

Recall that in Eq. \eqref{eq:Def_Canoinal_Quad_HP_W}, {\em canonical}
$\mathcal{HP}_{\color{blue}T}$ functions, with $I_m\succ{\color{blue}T}\succcurlyeq 0$,
are described as all $m\times m$-valued $F(s)$, satisfying
\begin{equation}\label{eq:Quadratic_HP_Delta_Again}
\left(\begin{smallmatrix}F(s)\\~\\I_m\end{smallmatrix}\right)^*\left(
\begin{smallmatrix}-{\color{blue}T}&&~~I_m\\~\\~~I_m&&-{\color{blue}T}\end{smallmatrix}
\right)\left(\begin{smallmatrix}F(s)\\~\\I_m\end{smallmatrix}\right)~~\left\{
\begin{smallmatrix}\succcurlyeq 0&&\forall s\in\C_R
\\~\\=0&&\forall s\in{i}\R.\end{smallmatrix}\right.
\end{equation}
Recall also that in Eq.  \eqref{eq:Def_Canonical_B} {\em canonical} $\mathcal{B}$
functions were described as all $m\times m$-valued $G_o(s)$ satisfying
\begin{equation}\label{eq:Canonical_Bounded}
\left(\begin{smallmatrix}G_o(s)\\~\\I_m\end{smallmatrix}\right)^*\left(
\begin{smallmatrix}-I_m&&~0\\~\\~~0&&~I_m\end{smallmatrix}\right)\left(
\begin{smallmatrix}G_o(s)\\~\\I_m\end{smallmatrix}\right)~~\left\{
\begin{smallmatrix}\succcurlyeq 0&&s\in\C_R\\~\\=0&&s\in{i}\R.\end{smallmatrix}\right.
\end{equation}
The following will turn to be useful.

\begin{La}\label{La:Affine_G_o_F}
Let $F(s)$ and $G_o(s)$ be a pair of $m\times m$-valued functions satisfying, for a
prescribed \mbox{$I_m\succ{\color{blue}T}\succ 0$,} 
\begin{equation}\label{eq:Affine_G_o_F}
\begin{matrix}G_o(s)&=&
\begin{smallmatrix}{\color{blue}T}^{\frac{1}{2}}\end{smallmatrix}\left(F(s)-
\begin{smallmatrix}{\color{blue}T}^{-1}\end{smallmatrix}\right)
\begin{smallmatrix}({\color{blue}T}^{-1}-{\color{blue}T})^{-\frac{1}{2}}
\end{smallmatrix}\\~\\F(s)&=&\begin{smallmatrix}{\color{blue}T}^{-1}\end{smallmatrix}
+\begin{smallmatrix}{\color{blue}T}^{-\frac{1}{2}}\end{smallmatrix}
G_o(s)\begin{smallmatrix}({\color{blue}T}^{-1}-{\color{blue}T})^{\frac{1}{2}}
\end{smallmatrix}.\end{matrix}
\end{equation}
Then, $F(s)$ is a {\em canonical} $\mathcal{HP}_{\color{blue}T}$ function, if and
only if, $G_o(s)$ is a {\em canonical} $\mathcal{B}$ function, see Eq.
\eqref{eq:Def_Canonical_B}
\end{La}

{\bf Proof :}~ We first find it convenient to cast the affine-linear relations in the
claim, in quadratic form, 
\begin{equation}\label{eq:Quadratic_G_function_of_F}
\left(\begin{smallmatrix}G_o(s)\\~\\I_m\end{smallmatrix}\right)=\underbrace{\left(
\begin{smallmatrix}{\color{blue}T}^{\frac{1}{2}}&&-{\color{blue}T}^{-\frac{1}{2}}
\\~\\0&&({\color{blue}T}^{-1}-{\color{blue}T})^{\frac{1}{2}}\end{smallmatrix}
\right)}_{V}\left(\begin{smallmatrix}F(s)\\~\\I_m\end{smallmatrix}\right)
\begin{smallmatrix}({\color{blue}T}^{-1}-{\color{blue}T})^{-\frac{1}{2}}\end{smallmatrix}
\end{equation}
and
\[
\left(\begin{smallmatrix}F(s)\\~\\I_m\end{smallmatrix}\right)=\underbrace{\left(
\begin{smallmatrix}{\color{blue}T}^{-\frac{1}{2}}&&{\color{blue}T}^{-1}(
{\color{blue}T}^{-1}-{\color{blue}T})^{-\frac{1}{2}}\\~\\0&&({\color{blue}T}^{-1}
-{\color{blue}T})^{-\frac{1}{2}}\end{smallmatrix}\right)}_{V^{-1}}\left(
\begin{smallmatrix}G_o(s)\\~\\I_m\end{smallmatrix}\right)\begin{smallmatrix}(
{\color{blue}T}^{-1}-{\color{blue}T})^{\frac{1}{2}}\end{smallmatrix}.
\]
Now, to verify the claim note that
\[
\begin{matrix}\underbrace{\left(\begin{smallmatrix}G_o(s)\\~\\I_m\end{smallmatrix}\right)^*
\overbrace{\left(\begin{smallmatrix}-I_m&&0\\~\\0&&I_m\end{smallmatrix}\right)}^U\left(
\begin{smallmatrix}G_o(s)\\~\\I_m\end{smallmatrix}\right)}_{\rm See~ Eq.~
\eqref{eq:Canonical_Bounded}}&=&\underbrace{\begin{smallmatrix}({\color{blue}T}^{-1}-
{\color{blue}T})^{-\frac{1}{2}}\end{smallmatrix}\left(\begin{smallmatrix}F(s)\\~\\I_m
\end{smallmatrix}\right)^*\begin{smallmatrix}V^*\end{smallmatrix}}_{{\rm by~Eq.~
\eqref{eq:Quadratic_G_function_of_F}:}~~\left(\begin{smallmatrix}G_o(s)\\~\\I_m
\end{smallmatrix}\right)^*}\overbrace{\left(\begin{smallmatrix}-I_m&&0\\~\\0&&I_m
\end{smallmatrix}\right)}^{U}\underbrace{\begin{smallmatrix}V\end{smallmatrix}\left(
\begin{smallmatrix}F(s)\\~\\I_m\end{smallmatrix}\right)\begin{smallmatrix}(
{\color{blue}T}^{-1}-{\color{blue}T})^{-\frac{1}{2}}\end{smallmatrix}}_{{\rm by~Eq.~
\eqref{eq:Quadratic_G_function_of_F}:}~~\left(\begin{smallmatrix}G_o(s)\\~\\I_m
\end{smallmatrix}\right)}\\~\\~&=&\begin{smallmatrix}({\color{blue}T}^{-1}-
{\color{blue}T})^{-\frac{1}{2}}\end{smallmatrix}\underbrace{\left(\begin{smallmatrix}F(s)
\\~\\I_m\end{smallmatrix}\right)^*\overbrace{\left(\begin{smallmatrix}-
{\color{blue}T}&&~~I_m\\~\\~~I_m&&-{\color{blue}T}\end{smallmatrix}\right)}^{V^*UV}\left(
\begin{smallmatrix}F(s)\\~\\I_m\end{smallmatrix}\right)}_{\rm See~ Eq.~\eqref
{eq:Quadratic_HP_Delta_Again}}\begin{smallmatrix}({\color{blue}T}^{-1}-{\color{blue}T})^{
-\frac{1}{2}}\end{smallmatrix}.\end{matrix}
\]
With respect to Eqs. \eqref{eq:Quadratic_HP_Delta_Again}, \eqref{eq:Canonical_Bounded},
the left-hand side is in $\overline{\mathbf P}_m$, if and only if, the right-hand side is.
\qed
\smallskip

We next parametrize all {\em canonical} $\mathcal{B}$ functions. To this end,
recall that if $\Pi\in\C^{m\times m}$ is an orthogonal projection, i.e.
\[
{\Pi}^*={\Pi}={\Pi}^2,
\]
of rank $k$, for some $k\in[1,~m]$, then if can be written as
\[
\Pi=v_1{v_1}^*+~\ldots~+v_k{v_k}^*\quad\quad {\rm where}~~{v_j}^*v_l=
\left\{\begin{matrix}1&&j=l\\0&&j\not=l.\end{matrix}\right.
\]

\begin{La}\label{La:Blaschke}
Let $G_o(s)$ be an arbitrary $m\times m$-valued {\em canonical} $\mathcal{B}$ function,
see Eq. \eqref{eq:Def_Canonical_B}. Then, it can be written as
\begin{equation}\label{eq:Blaschke_G_o}
G_o(s)=\prod\limits_{j=1}^k\begin{smallmatrix}\left(I_m-\frac{2}{1+{\psi}_j(s)}\cdot
v_j{v_j}^*\right)\end{smallmatrix}\quad with\quad k~~{\rm a~parameter},
\end{equation}
and ${\psi}_j(s)$ are scalar $\mathcal{PO}$ functions (see Eq.\eqref{eq:Def_PO}).
\end{La}

{\bf Proof :}~ We here rely on two facts:\\
(i) From Proposition \ref{Pn:Cayley_Functions} one has that
\mbox{$\mathcal{C}\left(\mathcal{PO}\right)={\rm canonical}~\mathcal{B}$.}\\
(ii)~ Since the family of $m\times m$-valued {\em canonical} $\mathcal{B}$ function
is closed under product among its elements (see Eq. \eqref{eq:Def_B}), it implies
that whenever (for some $k$) ${\Gamma}_1(s)$, $\ldots$, ${\Gamma}_k(s)$, are {\em
canonical} $\mathcal{B}$ functions, then so is
their product \mbox{$G_o(s)=\prod\limits_{j=1}^k{{\Gamma}_j(s)}$.} 
\smallskip

We next explore the nature of factors of the form of ${\Gamma}_1$, $\ldots$,
${\Gamma}_k~$.
\smallskip

First, if $\psi(s)$ is a scalar $\mathcal{PO}$ function, then so is
$\frac{1}{\psi(s)}~$. Now,
\[
g(s):=\mathcal{C}(\begin{smallmatrix}\frac{1}{\psi(s)}\end{smallmatrix})=
\begin{smallmatrix}\frac{1-\frac{1}{\psi(s)}}{1+\frac{1}{\psi(s)}}\end{smallmatrix}
=\begin{smallmatrix}1-\frac{2}{1+\psi(s)}\end{smallmatrix}~,
\]
is a (scalar) {\em canonical} $\mathcal{B}$ function.\quad
Similarly, with the same $\psi(s)$,
\begin{equation}\label{eq:Gamm_Prototype}
\Gamma(s)=\begin{smallmatrix}I_m-\frac{2}{1+{\psi}(s)}\cdot vv^*
\end{smallmatrix}\quad{\rm where}\quad v\in\C^m,~~v^*v=1,
\end{equation}
is a $m\times m$-valued {\em canonical} $\mathcal{B}$
function, sharing the same McMillan degree as $\psi(s)$.
\smallskip

To verify that directly note that for arbitrary $u\in\C^m$,
\[
u^*(I_m-{\Gamma}^*\Gamma)u=\underbrace{\begin{smallmatrix}\frac{2(\psi(s)+
{\psi(s)}^*)}{(1+\psi(s))(1+{\psi(s)}^*)}\end{smallmatrix}}_{a(s)}\cdot
\underbrace{\begin{smallmatrix}|u^*v|^2\end{smallmatrix}}_{\geq 0}~.
\]
Now since $\psi(s)$ is a $\mathcal{PO}$ function, it implies that
\[
\begin{smallmatrix}1\geq a(s):=\frac{2(\psi(s)+{\psi(s)}^*)}
{(1+\psi(s))(1+{\psi(s)}^*)}>0&&s\in\C_R\\~\\~~~a(s):=\frac{2(\psi(s)+
{\psi(s)}^*)}{(1+\psi(s))(1+{\psi(s)}^*)}=0&&s\in{i}\R.\end{smallmatrix}
\]
Thus, one can conclude that $\Gamma(s)$ is indeed a $m\times m$-valued
{\em canonical} $\mathcal{B}$
function (see Eq.  \eqref{eq:Canonical_Bounded}). Furthermore, one has that
\[
u^*u\geq |u^*v|^2\geq 0,
\]
and the left-hand side holds with equality, if and only if \mbox{$u=cv$,} for
some $c\in\C$ (i.e. $u$ and $v$ are linearly dependent), and
the right-hand side holds with equality, if and only if \mbox{$u^*v=0$,}
(i.e. $u$ is orthogonal to $v$).
\smallskip

When Eq. \eqref{eq:Gamm_Prototype} is extended to
\mbox{$G_o(s)=\prod\limits_{j=1}^k{{\Gamma}_j(s)}$,} 
similar argument holds, so the construction is complete.
\qed
\smallskip

\begin{Rk}
{\rm
{\bf a.}~ 
The description of $G_o(s)$ in Eq. \eqref{eq:Blaschke_G_o} involves the scalar
$\mathcal{PO}$ functions ${\psi}_1(s)$, $\ldots$, ${\psi}_k(s)$. Recall now that due
to the Foster parametrization (see e.g. \cite[Eq. (9), Ch. 5]{Belev1968}), a scalar
$\mathcal{PO}$ functions can always be written as, 
\[
a_os+\frac{b_o}{s}+\sum\limits_j(a_js+\frac{b_j}{s})^{-1}\quad\quad
a_o, b_o\geq 0,~~a_j, b_j>0.
\]
{\bf b.}~ The parametrization of $G_o(s)$ in Lemma \ref{La:Blaschke} is of
the nature of to the Blaschke product description of unitary functions. For
example, see
the discussion in \cite[Section 2]{ag1}.
}
$\T$
\end{Rk}
\smallskip

One can now combine Eq. \eqref{eq:Quadratic_G_function_of_F} together with
Eq. \eqref{eq:Blaschke_G_o} to formulate the main result of this section, a
systematic description of all $m\times m$-valued {\em canonical}
$\mathcal{HP}_{\color{blue}T}$ functions, where
\mbox{$I_m\succ{\color{blue}T}\succ 0$.}

\begin{Pn}
Let $F(s)$ be an arbitrary $m\times m$-valued {\em canonical}
$\mathcal{HP}_{\color{blue}T}$,
\mbox{$I_m\succ{\color{blue}T}\succ 0$,} function. Then, it can be written as
\[
F(s)=\begin{smallmatrix}{\color{blue}T}^{-\frac{1}{2}}\end{smallmatrix}\overbrace{
\left(\begin{smallmatrix}\prod\limits_{j=1}^k\left(I_m-\frac{2}{1+{\psi}_j(s)}
\cdot{v}_j{v_j}^*\right)\end{smallmatrix}\right)}^{G_o(s)}\begin{smallmatrix}(
{\color{blue}T}^{-1}-{\color{blue}T})^{\frac{1}{2}}\end{smallmatrix}+
\begin{smallmatrix}{\color{blue}T}^{-1}\end{smallmatrix}\quad with
\quad\begin{smallmatrix}k~~parameter \\~\\ v_j\in\C^m\\~\\
{v_j}^*v_j=1\end{smallmatrix}
\]
and ${\psi}_j(s)$ are scalar $\mathcal{PO}$ functions.
\end{Pn}

\section{Structural Properties of 
{\em Canonical} $\mathcal{HP}_{\color{blue}T}$ Rational Functions}
\label{Sec:Structural_Prproperties}
\setcounter{equation}{0}

\subsection{Proof of part B of Theorem \ref{Tm:Set_Of_HP_Functions}}
\label{Subsec:Proof_of_Theorem_1.7_B}

(i)~ For $j=0,~1$, let $F_j(s)$ be a pair of {\em canonical}
$\mathcal{HP}_{\color{blue}T}$
functions, then one can re-write Eq. \eqref{eq:Def_Canoinal_Quad_HP_W} as,
\[
F_j(s)+{F_j(s)}^*=\begin{smallmatrix}{\Delta}_j\end{smallmatrix}+
\begin{smallmatrix}{\color{blue}T}\end{smallmatrix}
+{F_j(s)}^*\begin{smallmatrix}{\color{blue}T}\end{smallmatrix}F_j(s)
\quad\quad\begin{smallmatrix}{\Delta}_j\end{smallmatrix}=
\begin{smallmatrix}{\Delta}_j(s)\end{smallmatrix}=\left\{\begin{smallmatrix}
\succcurlyeq 0&&s\in\C_R\\~\\=0&&s\in{i}\R.\end{smallmatrix}\right.
\]
For ${\scriptstyle\alpha}\in[0,~1]$ denote 
\mbox{$\begin{smallmatrix}{\Delta}_{\alpha}\end{smallmatrix}
:=\begin{smallmatrix}\alpha{\Delta}_1\end{smallmatrix}+
\begin{smallmatrix}(1-\alpha){\Delta}_0\end{smallmatrix}$} and
\mbox{$\begin{smallmatrix}F_{\alpha}\end{smallmatrix}:=\begin{smallmatrix}\alpha{F}_1
\end{smallmatrix}+\begin{smallmatrix}(1-\alpha){F}_0\end{smallmatrix}$.} Thus, a
straightforward computation yields, for all $s\in\overline{\C}_R$, and all 
\mbox{$\begin{smallmatrix}\alpha\end{smallmatrix}\in[0, 1]$,}
\[
F_{\alpha}(s)+{F_{\alpha}(s)}^*={\color{blue}\begin{smallmatrix}T\end{smallmatrix}}
+{F_{\alpha}(s)}^*{\color{blue}\begin{smallmatrix}T\end{smallmatrix}}{F_{\alpha}(s)}
+
\underbrace{\begin{smallmatrix}{\Delta}_{\alpha}(s)\end{smallmatrix}}_{\succcurlyeq 0}
+
\underbrace{\begin{smallmatrix}\alpha(1-\alpha)\end{smallmatrix}(F_0(s)-F_1(s))^*{
\color{blue}\begin{smallmatrix}T\end{smallmatrix}}(F_0(s)-F_1(s))}_{\succcurlyeq 0}.
\]
Consider first the case where $s\in{i}\R$: Note that 
\mbox{$\begin{smallmatrix}{\Delta}_{\alpha}(s)\end{smallmatrix}\equiv 0$,} for all 
\mbox{${\scriptstyle\alpha}\in[0,~1]$.} Thus,
\mbox{$F_{\alpha}(s)+{F_{\alpha}(s)}^*={\color{blue}\begin{smallmatrix}T\end{smallmatrix}}
+{F_{\alpha}(s)}^*{\color{blue}\begin{smallmatrix}T\end{smallmatrix}}{F_{\alpha}(s)}$,}
for all $s\in{i}\R$, if and only if, the condition in Eq.
\eqref{eq:Condition_Canonical_Convex}
is satisfied.
\smallskip

Else, \mbox{$F_{\alpha}(s)+{F_{\alpha}(s)}^*\succcurlyeq
{\color{blue}\begin{smallmatrix}T\end{smallmatrix}}
+{F_{\alpha}(s)}^*{\color{blue}\begin{smallmatrix}T\end{smallmatrix}}{F_{\alpha}(s)}$,}
for some $s\in{i}\R$, so indeed $F_{\alpha}(s)$ is not {\em canonical}.
\smallskip

Next, we address ourselves to the case where $s\in\C_R$. Then,
for ${\scriptstyle\alpha}\in(0,~1)$, one has that both
\mbox{$\begin{smallmatrix}{\Delta}_{\alpha}(s)\end{smallmatrix}\succcurlyeq 0$,}
~and~
\mbox{$\begin{smallmatrix}\alpha(1-\alpha)\end{smallmatrix}(F_0(s)-F_1(s))^*{
\color{blue}\begin{smallmatrix}T\end{smallmatrix}}(F_0(s)-F_1(s))\succcurlyeq 0$.}
Hence, one can find ${\color{cyan}\hat{T}}$,
\mbox{$I_m\succ{\color{cyan}\hat{T}}\succcurlyeq{\color{blue}T}$,} so that, for all
\mbox{$\begin{smallmatrix}\alpha\end{smallmatrix}\in(0, 1)$,}
\[
\begin{smallmatrix}{\Delta}_{\alpha}(s)\end{smallmatrix}
+\begin{smallmatrix}\alpha(1-\alpha)\end{smallmatrix}
(F_0(s)-F_1(s))^*{\color{blue}\begin{smallmatrix}T\end{smallmatrix}}(F_0(s)-F_1(s))
\succcurlyeq({\color{cyan}\begin{smallmatrix}\hat{T}\end{smallmatrix}}
-{\color{blue}\begin{smallmatrix}T\end{smallmatrix}})
+{F_{\alpha}(s)}^*({\color{cyan}\begin{smallmatrix}\hat{T}\end{smallmatrix}}
-{\color{blue}\begin{smallmatrix}T\end{smallmatrix}})F_{\alpha}(s).
\]
Thus, this part of the claim is established.
\bigskip

(ii)~ Multiply Eq. \eqref{eq:Def_Canoinal_Quad_HP_W} by $\left(F(s)^{-1}\right)^*$ and
$F(s)^{-1}$ from the left and from the right, respectively to obtain
\[
\begin{smallmatrix}\left(\begin{smallmatrix}I_m\\~\\F(s)^{-1}\end{smallmatrix}\right)^*
\left(\begin{smallmatrix}-{\color{blue}T}&&~~I_m\\~\\~~I_m&&-{\color{blue}T}\end{smallmatrix}
\right)\left(\begin{smallmatrix}I_m\\~\\F(s)^{-1}\end{smallmatrix}\right)
&\left\{\begin{smallmatrix}\succcurlyeq 0&&\forall s\in\C_R\\~\\=0&&\forall s\in{i}\R,
\end{smallmatrix}\right.\end{smallmatrix}
\]
which is equivalent to
\[
\begin{smallmatrix}\left(\begin{smallmatrix}F(s)^{-1}\\~\\I_m\end{smallmatrix}\right)^*
\left(\begin{smallmatrix}-{\color{blue}T}&&~~I_m\\~\\~~I_m&&
-{\color{blue}T}\end{smallmatrix}\right)
\left(\begin{smallmatrix}F(s)^{-1}\\~\\I_m\end{smallmatrix}\right)
&\left\{\begin{smallmatrix}\succcurlyeq 0&&\forall s\in\C_R\\~\\=0&&\forall s\in{i}\R,
\end{smallmatrix}\right.\end{smallmatrix}
\]
so this item is established.
\smallskip

(iii)~ 
For a scalar {\em canonical} $\mathcal{HP}_{\color{blue}\beta}$ function $f(s)$ one
has that
\[
\begin{smallmatrix}{\color{blue}\beta}\end{smallmatrix}=
\begin{smallmatrix}\frac{f+f^*}{1+|f|^2}
\end{smallmatrix}\quad\quad\quad\forall s\in{i}\R.
\]
Thus one can write $\forall s\in{i}\R$
\[
\begin{matrix}\begin{smallmatrix}{\color{blue}{\beta}_{\rm new}}\end{smallmatrix}&=&
\begin{smallmatrix}\frac{\frac{1}{2}(f+\frac{1}{f})+\frac{1}{2}(f+\frac{1}{f})^*}
{1+\left|\frac{1}{2}(f+\frac{1}{f})\right|^2}\end{smallmatrix}=
\begin{smallmatrix}\frac{\frac{1}{2}(f+\frac{f^*}{|f|^2})+\frac{1}{2}(f^*+\frac{f}{|f|^2})}
{1+\left|\frac{1}{2}(f+\frac{f^*}{|f|^2})\right|^2}\end{smallmatrix}=
\begin{smallmatrix}\frac{\frac{1}{2}(f+f^*)(1+\frac{1}{|f|^2})}
{1+\frac{1}{4}(|f|^2+\frac{(f)^2}{|f|^2}+\frac{(f^*)^2}{|f|^2}+
\frac{1}{|f|^2})}\end{smallmatrix}\\~\\~&=&\begin{smallmatrix}\frac{2(f+f^*)
(1+|f|^2)}{4|f|^2+|f|^4+(f)^2+(f^*)^2+1}\end{smallmatrix}=
\begin{smallmatrix}\frac{2(f+f^*)(1+|f|^2)}{(1+|f|^2)^2+(f+f^*)^2}
\end{smallmatrix}=\begin{smallmatrix}\frac{2\frac{f+f^*}{1+|f|^2}}
{1+(\frac{f+f^*}{1+|f|^2})^2}\end{smallmatrix}
=\begin{smallmatrix}\frac{2\beta}{1+{\beta}^2}\end{smallmatrix}
=\begin{smallmatrix}\left(\frac{1}{2}(\beta+\frac{1}{\beta})\right)^{-1}\end{smallmatrix},
\end{matrix}
\]
and the proof is complete.
\qed
\smallskip

\begin{Rk}\label{Rk:PO_not_Canonical}
{\rm
Following Remark \ref{Rk:PO_Canonical_HP_Delta} and the first two lines in Proposition
\ref{Pn:Cayley_Functions}, one may be tempted to say that $\mathcal{PO}$ functions play
the role of {\em canonical} $\mathcal{P}$ functions. However, the two cases differ. It
is only for \mbox{${\color{blue}T}=0$} (or \mbox{$\begin{smallmatrix}{\color{blue}\beta}
\end{smallmatrix}=0$}) that the boundary of the convex set $\mathcal{HP}_{\color{blue}T}$,
is convex by itself, see item (i) of part B of Theorem
\ref{Tm:Set_Of_HP_Functions}. Roughly, this can be pictorially viewed in Figure
\ref{Fig:Degree_One_HP}, by comparing the (convex) imaginary axis with the 
(non-convex) blue {\em circle}.
}
$\T$
\end{Rk}

\subsection{Convex Combination of {\em Canonical} $\mathcal{HP}_{\color{blue}T}$
Rational Functions}
\label{Subsec:Convex_Combination_Rational}

Item {\bf B.} (i) of Theorem \ref{Tm:Set_Of_HP_Functions} implies that if $F_0(s)$
and $F_1(s)$ is a pair of {\em canonical} $\mathcal{HP}_{\color{blue}T}$ functions,
then \mbox{$\forall {\scriptstyle\alpha}\in(0, 1)$,}
\mbox{$({\scriptstyle\alpha}F_1+{\scriptstyle(1-\alpha)}F_0)(s)$,} is a
(non-{\em canonical}) $\mathcal{HP}_{\color{blue}T}$ function. This subsection
focuses on this fact.
\smallskip

In principle some version of the converse statement holds as well: When
\mbox{$I_m\succ{\color{blue}T}\succ 0$,} an arbitrary
$\mathcal{HP}_{\color{blue}T}$ function can be written as a limit of a convex
combination of {\em canonical} $\mathcal{HP}_{\color{blue}T}$ functions. Note
however that this is true in the broader framework of the space
$\mathcal H(\mathbb C_+)$, functions analytic in the right open half-plane
endowed with the topology of uniform convergence on compact sets. Under this
topology $\mathcal H(\mathbb C_+)$ is a Fr\'echet space in which being compact
is equivalent to being bounded and closed; see e.g. \cite[p. 166]{cartan}. This
allows to apply the Krein-Milman theorem (see e.g. \cite[p. 13]{brezis},
\cite[p. 362-363]{yosida} for the latter) to the set $\mathcal{HP}_{\color{blue}T}$
functions. The details are beyond the
scope of the work.\begin{footnote} {The subset of rational functions is not
closed.}\end{footnote} For a related analysis in the setting
of realizations, see Subsection \ref{Subsec:Conex_Realizations}.
\smallskip

We start with a very simple case.
\smallskip

{\bf Proof of Proposition \ref{Pn:Convex_Combination_Degree_One} :}
\smallskip

To simplify the construction, substitute in Eq.  \eqref{eq:Canonical_Degree_One}
$b=a$ and \mbox{$\begin{smallmatrix}{\beta}\end{smallmatrix}=
\begin{smallmatrix}{\beta}_1\end{smallmatrix}$,}
to obtain the following convex combination of ${\phi}_1(s)$ and ${\phi}_2(s)$,
\[
f_{\alpha}(s):=\begin{smallmatrix}\alpha\end{smallmatrix}{\phi}_1(s)+
\begin{smallmatrix}(1-\alpha)\end{smallmatrix}{\phi}_2(s)
=\begin{smallmatrix}\frac{1}{{\beta}_1}\end{smallmatrix}+\begin{smallmatrix}
(2\alpha-1)\frac{\sqrt{1-{{\beta}_1}^2}}{{\beta}_1}\end{smallmatrix}\frac{s-a}{s+a}
\quad\quad\quad
\begin{smallmatrix}{\beta}_1\in(0,~1)\\~\\a>0\\~\\ \alpha\in[0,~1].\end{smallmatrix}
\]
Without loss of generality assume that
\mbox{$\begin{smallmatrix}{\beta}_2\end{smallmatrix}\geq\begin{smallmatrix}{\beta}_1
\end{smallmatrix}$.} Then using the notation of Eq. \eqref{eq:HP_Deg_One}),
\[
f_{\alpha}(s)=\left\{\begin{matrix}\tilde{\phi}_1(s)&{\rm when}&
\begin{smallmatrix}\alpha\end{smallmatrix}=\begin{smallmatrix}\frac{1}{2}
\end{smallmatrix}(1+\begin{smallmatrix}\frac{{\beta}_1}{{\beta}_2}
\frac{\sqrt{1-{\beta}_2^2}}{\sqrt{1-{\beta}_1^2}}\end{smallmatrix})
\\
\tilde{\phi}_2(s)&{\rm when}&\begin{smallmatrix}\alpha\end{smallmatrix}=
\begin{smallmatrix}\frac{1}{2}\end{smallmatrix}(1-\begin{smallmatrix}
\frac{{\beta}_1}{{\beta}_2}\frac{\sqrt{1-{\beta}_2^2}}
{\sqrt{1-{\beta}_1^2}}\end{smallmatrix}).
\end{matrix}\right.
\]
Namely, this is a parametrization of non-{\em canonical} functions.
Thus the claim is established.
\qed
\bigskip

In fact, with the pair of {\em canonical} functions of degree one, from Eq.
\eqref{eq:Canonical_Degree_One}, one can go beyond the framework of Proposition
\ref{Pn:Convex_Combination_Degree_One}. For example, already a convex combination of
${\phi}_2(s)_{|_{b=1}}$ along with ${\phi}_2(s)_{|_{b=20}}$, results in a function
of degree two, which is illustrated in Figure \ref{Figure:Convex_Two_Canonical_1}.

\begin{figure}[H]
\centering
\begin{minipage}{0.44\linewidth}
{\rm
Consider three $\mathcal{HP}_{\color{blue}\beta}$ functions with
\mbox{$\begin{smallmatrix}{\color{blue}\beta}\end{smallmatrix}=
\begin{smallmatrix}\frac{4}{5}\end{smallmatrix}$~:}
\smallskip

Both functions (see Eq.
\eqref{eq:Canonical_Degree_One}),

$\begin{matrix}{\color{blue}f_1(s)}:={\color{blue}{\phi}_2(s)}_{|_{b=1}}=
{\color{blue}\begin{smallmatrix}\frac{\frac{1}{2}s+2}{s+1}
\end{smallmatrix}}
\\~\\
{\color{blue}f_2(s)}:={\color{blue}{\phi}_2(s)}_{|_{b=20}}=
{\color{blue}\begin{smallmatrix}
\frac{\frac{1}{2}s+40}{s+20}\end{smallmatrix}}
\end{matrix}$

are {\em canonical}, their Nyquist plots are identical. 
\smallskip

An associated convex combination is given by, \mbox{${\color{red}f_3(s)}:=
\begin{smallmatrix}\frac{2}{5}\end{smallmatrix}{\color{blue}f_1(s)}+\begin{smallmatrix}
\frac{3}{5}\end{smallmatrix}{\color{blue}f_2(s)}=\begin{smallmatrix}\frac{\frac{1}{2}
s^2+29.1s+40}{(s+1)(s+20)}\end{smallmatrix}$.} }
\end{minipage}\quad\quad\quad\begin{minipage}{0.48\linewidth}
\includegraphics[width=0.64\textwidth]{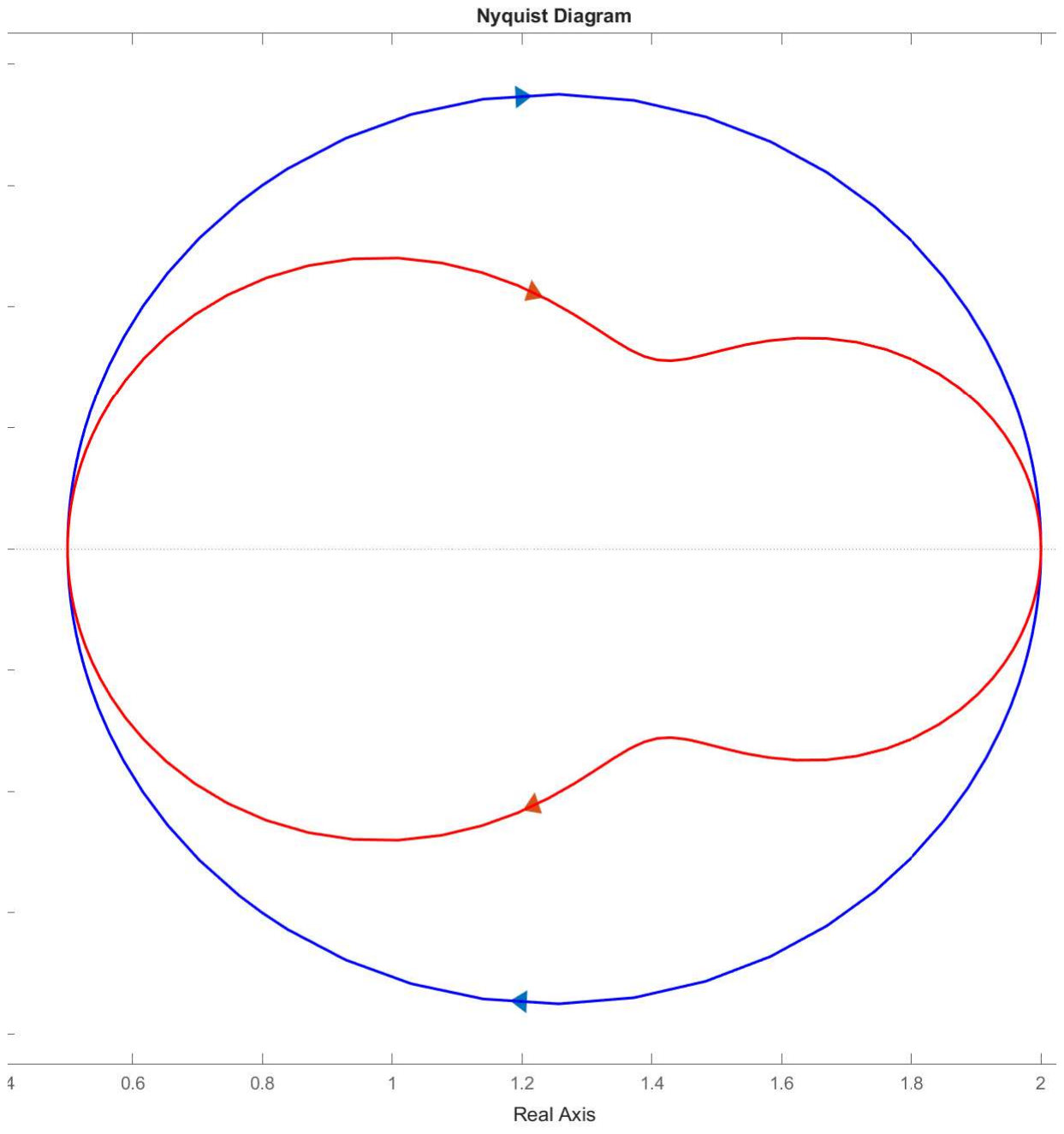}\end{minipage}
\caption{The Nyquist plot of ${\color{red}f_3(s)}$ a convex combination of
${\color{blue}f_1(s)}$, ${\color{blue}f_2(s)}$, two canonical
$\mathcal{HP}_{\color{blue}\beta}$ functions}
\label{Figure:Convex_Two_Canonical_1}
\end{figure}

A slightly richer example is given in Figure \ref{Figure:Convex_Two_Canonical_3}.

\begin{figure}[H]
\centering
\begin{minipage}{0.44\linewidth}
{\rm
Consider three $\mathcal{HP}_{\beta}$ functions with the same
${\scriptstyle\beta}$.\\
Using the {\em canonical} $\mathcal{HP}_{\color{blue}\beta}$ functions in Eqs.
\eqref{eq:Canonical_Degree_One}, \eqref{al:Phi_3} and \eqref{al:Phi_4},
we construct two convex combinations,
\mbox{${\color{blue}f_1(s)}=
\frac{1}{2}({\phi}_1+{{\phi}_4}_{|_{c=d=a}})(s)={\color{blue}
\begin{smallmatrix}\frac{1}{\beta}\end{smallmatrix}+
\begin{smallmatrix}\frac{\sqrt{1-{\beta}^2}}
{\beta}\end{smallmatrix}\frac{a(s-a)}{(s+a)^2}}$}\\
\mbox{${\color{red}f_2(s)}=\frac{1}{2}({{\phi}_2}_{|_{b=a}}+
{{\phi}_3}_{|_{c=d=a}})(s)={\color{red}\begin{smallmatrix}
\frac{1}{\beta}\end{smallmatrix}-\begin{smallmatrix}
\frac{\sqrt{1-{\beta}^2}}{\beta}\end{smallmatrix}
\frac{a(s-a)}{(s+a)^2}}$}\\ 
\mbox{${\phi}_1(s)=\begin{smallmatrix}\frac{1}{\beta}\end{smallmatrix}+
\begin{smallmatrix}\frac{\sqrt{1-{\beta}^2}}{\beta}\end{smallmatrix}
\frac{s-a}{s+a}$~,} canonical.}
\end{minipage}\quad\quad\begin{minipage}{0.50\linewidth}
\includegraphics[width=1.00\textwidth]{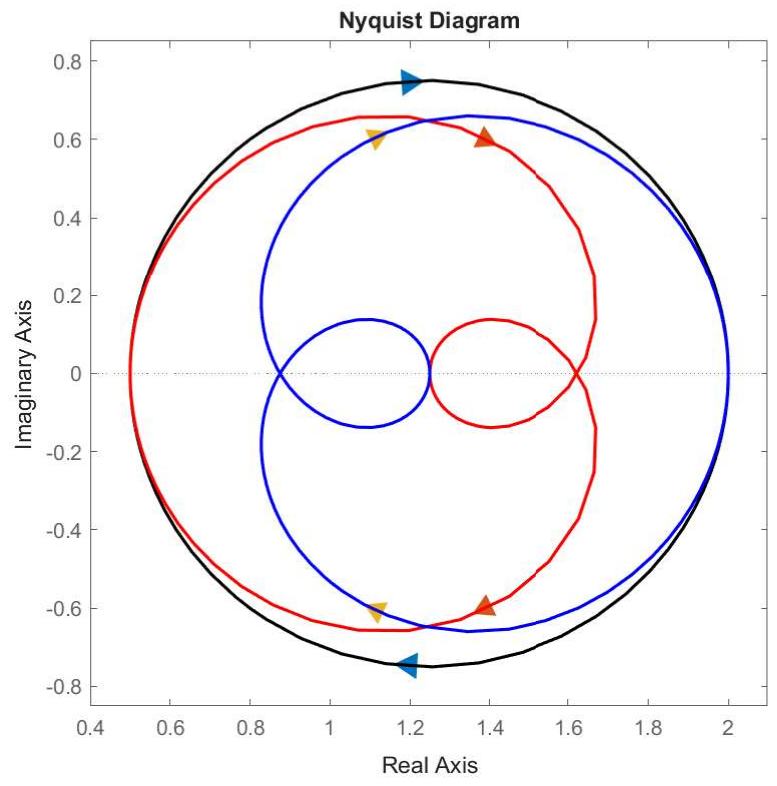}
\end{minipage}
\caption{ The Nyquist plot of ${\phi}_1(s)$,~ \mbox{${\color{blue}f_1(s)}$,}~
\mbox{${\color{red}f_2(s)}$}}
\label{Figure:Convex_Two_Canonical_3}
\end{figure}

\begin{Ex}
{\rm
We first illustrate the fact that taking a convex combination of a pair of scalar
{\em canonical} $\mathcal{HP}_{\color{blue}\beta}$ functions ``improves Lurie
stability". To simplify the construction, (as in the proof of Proposition
\ref{Pn:Convex_Combination_Degree_One}) substitute in Eq.
 \eqref{eq:Canonical_Degree_One} $b=a$, to obtain the following convex combination
of ${\phi}_1(s)$ and ${\phi}_2(s)$ (see Eqs. \eqref{eq:Canonical_Degree_One},
\eqref{eq:Alternative_Canonical_Degree_One}),
\begin{equation}\label{eq:f_{alpha}}
f_{\alpha}(s):=\begin{smallmatrix}\alpha\end{smallmatrix}{\phi}_1(s)+
\begin{smallmatrix}(1-\alpha)\end{smallmatrix}{\phi}_2(s)
=\begin{smallmatrix}\frac{1}{\beta}\end{smallmatrix}+
\begin{smallmatrix}(2\alpha-1)\frac{\sqrt{1-{\beta}^2}}{\beta}\end{smallmatrix}
\frac{s-a}{s+a}\quad\quad\quad
\begin{smallmatrix}\beta\in(0,~1)\\~\\a>0\\~\\ \alpha\in[0,~1].\end{smallmatrix}
\end{equation}
It turns out that for all $\begin{smallmatrix}\alpha\end{smallmatrix}\in(0,~1)$,
$f_{\alpha}(s)$ in Eq. \eqref{eq:f_{alpha}}, is a {\em non-canonical}
$\mathcal{HP}_{\color{blue}{\beta}_{\alpha}}$ function, with
\begin{equation}\label{eq:Beta_{alpha}}
\begin{smallmatrix}{\color{blue}{\beta}_{\alpha}}\end{smallmatrix}
=\left(\begin{smallmatrix}\frac{1}{2}\end{smallmatrix}
\left(\begin{smallmatrix}\frac{1}{\beta}(1+|2\alpha-1|\sqrt{1-{\beta}^2})\end{smallmatrix}
+\begin{smallmatrix}\frac{1}{\frac{1}{\beta}(1+|2\alpha-1|\sqrt{1-{\beta}^2})}
\end{smallmatrix}\right)\right)^{-1}.
\end{equation}
In particular for 
\mbox{$\begin{smallmatrix}\alpha\end{smallmatrix}=
\begin{smallmatrix}\frac{1}{2}\end{smallmatrix}$,}
\[
f_{\frac{1}{2}}(s)=\begin{smallmatrix}\frac{1}{2}\end{smallmatrix}({\phi}_1
+{\phi}_2)(s)\equiv\begin{smallmatrix}\frac{1}{{\beta}_{\rm NEW}}\end{smallmatrix}
\quad\forall s\in\C_R~.
\]
Namely, it is a zero degree function in
$\mathcal{HP}_{\color{blue}{\beta}_{\rm NEW}}$ where
\begin{equation}\label{eq:Beta_New}
\begin{smallmatrix}{\color{blue}{\beta}_{\rm NEW}}\end{smallmatrix}
=\left(\begin{smallmatrix}\frac{1}{2}\end{smallmatrix}\left(
\begin{smallmatrix}\frac{1}{\beta}\end{smallmatrix}+
\begin{smallmatrix}\beta\end{smallmatrix}\right)\right)^{-1}.
\end{equation}
This should be compared with item  {\bf B} (iii) of Theorem \ref{Tm:Set_Of_HP_Functions}.
}
$\T$
\end{Ex}

We next elaborate on samples of simple convex combinations of functions.

\begin{Ex}\label{Ex:Combinations_in_Figure}
{\rm
{\bf a.}~ 
Consider the {\em canonical} degree one function 
\mbox{${\color{blue}f_1(s)=\frac{2}{5}+\frac{\frac{21}{10}a_1}{s+a_1}}$,}
from Figure \ref{Fig:Degree_One_HP}.
Take also a zero degree function \mbox{$f_5(s)\equiv\frac{2}{5}$.}
\quad
It is now easy to verify that with \mbox{${\scriptstyle\alpha}=\frac{4}{9}$}
\[
{\scriptstyle\alpha}{\color{blue}f_1(s)}+{\scriptstyle(1-\alpha)}f_5(s)_{|_{a_1=a_2}}=
\begin{smallmatrix}{\color{red}\frac{2}{5}}\end{smallmatrix}+
{\color{red}\frac{{\scriptstyle\frac{14}{15}}a_2}{s+a_2}}
={\color{red}f_2(s)},
\]
where $ {\color{red}f_2(s)}$ is from Figure \ref{Fig:Degree_One_HP} as well.\\
Note that ${\color{blue}f_1(s)}$, ${\color{red}f_2(s)}$ and $f_5(s)$ are all
$\mathcal{HP}_{\beta}$ functions with ${\scriptstyle\beta}=\frac{20}{29}$.
\bigskip

{\bf b.}~ 
Consider again the {\em canonical} $\mathcal{HP}_{\beta}$ function 
\mbox{${\color{blue}f_1(s)=\frac{2}{5}+\frac{\frac{21}{10}a_1}{s+a_1}}$,}
from the previous item. From Eqs.
\eqref{eq:Canonical_Degree_One}
\eqref{eq:Inverse_Caninical_Degree_One} we know that its inverse takes the form
\[
f_6(s):=({\color{blue}f_1(s)})^{-1}=\begin{smallmatrix}\frac{5}{2}\end{smallmatrix}-
\frac{\frac{21}{10}a_6}{s+a_6}\quad{\rm with}\quad\begin{smallmatrix}a_6\end{smallmatrix}
=\begin{smallmatrix}\frac{25}{4}a_1\end{smallmatrix}~.
\]
Moreover, $f_6(s)$ is, like ${\color{blue}f_1(s)}$, a {\em canonical}
$\mathcal{HP}_{\beta}$ function with ${\scriptstyle\beta}=\frac{20}{29}$.
\smallskip

As before, take a zero degree function \mbox{$f_7(s)\equiv\frac{5}{2}$.}
It is easy to verify that with \mbox{${\scriptstyle\alpha}=\frac{5}{6}$}
\[
{\scriptstyle\alpha}f_6(s){\scriptstyle(1-\alpha)}f_7(s)_{|_{a_1=a_2}}=
\begin{smallmatrix}{\color{orange}\frac{5}{2}}\end{smallmatrix}
-{\color{orange}\frac{{\scriptstyle\frac{7}{4}}a_3}{s+a_3}}
={\color{orange}f_3(s)},
\]
where $ {\color{orange}f_3(s)}$ is from Figure \ref{Fig:Degree_One_HP} as well.\\
Note that ${\color{orange}f_3(s)}$, $f_6(s)$ and $f_7(s)$ are all
$\mathcal{HP}_{\beta}$ functions with ${\scriptstyle\beta}=\frac{20}{29}$.
\bigskip

{\bf c.}~ Return now to the function
\mbox{${\color{red}f_2(s)=\frac{2}{5}+\frac{{\scriptstyle\frac{14}{15}}a_2}{s+a_2}}$}
from item {\bf a.} of this example and from Figure \ref{Fig:Degree_One_HP}. Note now
that if one wishes to find a $\mathcal{HP}_{\beta}$ function, of degree one, which
at $s=0$ attains the value of
\mbox{${\color{red}f_2}(0)=\begin{smallmatrix}\frac{4}{3}\end{smallmatrix}$} and at
$s=\infty$ the value of
\mbox{$\frac{1}{{\color{red}f_2}(0)}=\begin{smallmatrix}\frac{3}{4}\end{smallmatrix}$,}
it must be {\em canonical}. Furthermore, it turns out to be equal to
\mbox{${\color{teal}f_4(s)=\frac{3}{4}-\frac{\frac{7}{12}a_4}{s+a_4}}$,} from Figure
\ref{Fig:Degree_One_HP} as well.
}
$\T$
\end{Ex}
\smallskip

An example of a convex combination of {\em canonical} $\mathcal{HP}_{\color{blue}\beta}$
functions, will be given in Example \ref{Ex:Convex_Degree_Two} below.

\subsection{Left vs. Right {\em canonical} $\mathcal{HP}_{\color{blue}T}$ Functions}
\label{counter-123}

Recall that {\em Right} and {\em Left} $\mathcal{HP}_{\color{blue}T}$ functions were
introduced in Eqs. \eqref{eq:HP_Delta}, \eqref{eq:Left_HP_T} respectively,
\[
\begin{matrix}
{\rm Right}&F(s)+(F(s))^*\succcurlyeq\begin{smallmatrix}{\color{blue}T}\end{smallmatrix}
+(F(s))^{\color{red}*}\begin{smallmatrix}{\color{blue}T}\end{smallmatrix}F(s)
\\~\\
{\rm Left:}&F(s)+(F(s))^*\succcurlyeq\begin{smallmatrix}{\color{blue}T}\end{smallmatrix}
+F(s)\begin{smallmatrix}{\color{blue}T}\end{smallmatrix}(F(s))^{\color{red}*}
\end{matrix}
\quad\quad\forall s\in\C_R~.
\]
In this subsection we look into the difference between these sets. First,
the following claim appeared in \cite[item (ii) of Theorem 2.8]{AlpayLew2025a}.

\begin{Pn}
Given, \mbox{$I_m\succ{\color{blue}T}\succcurlyeq 0$.} When
\mbox{$F(s)\in\mathcal{HP}_{\color{blue}T}$}, it is equivalent to having the function
\mbox{$\begin{smallmatrix}(I_m-{\color{blue}T}^2)^{-\frac{1}{2}}\end{smallmatrix}
(F(s^*))^*\begin{smallmatrix}(I_m-{\color{blue}T}^2)^{\frac{1}{2}}\end{smallmatrix}$}
in \mbox{$\mathcal{HP}_{\color{blue}T}$,} as well.
\end{Pn}

Hence, one can say that if \mbox{$F(s){\color{blue}T}\equiv{\color{blue}T}F(s)$,} for all
$s\in\C_R$, then \mbox{$F(s)=(F(s^*))^*$} so this function is both Left and Right
Hyper-Positive. However, in general this is not the case.
\smallskip

In turns out that, it is enough to consider the {\em canonical} case where $F(s)$ is of
degree zero, i.e. on $i\R$, $F(s)$ is a constant matrix. Recall that ``Hyper-Lyapunov
Matrix Inclusions" were introduced in \cite{Lewk2024a}. 

\begin{Pn}\label{Pn:Left_Right_Hyper_Lyapunov}
For \mbox{$I_m\succ{\color{blue}T}\succ 0$,} consider the following  $\C^{m\times m}$
matrices\begin{footnote}{To ease the reading, we (artificially) denote the Right and Left
elements by $A_R$ and $A_L$, respectively.}\end{footnote} $A_R$ and $A_L$.

\begin{itemize}
\item[({\rm Right})~]{} A matrix \mbox{$A_R\in\C^{m\times m}$} satisfies the equation,
\begin{equation}\label{eq:Right_Hyper_Lyapunov}
A_R+{A_R}^*={\color{blue}T}+{A_R}^{\color{red}*}{\color{blue}T}A_R~,
\end{equation}
if and only if, for some \mbox{$UU^*=I_m=U^*U$,} this $A_R$ can be written as
\begin{equation}\label{eq:A_R}
A_R={\color{blue}T}^{-1}+{\color{blue}T}^{-\frac{1}{2}}U({\color{blue}T}^{-1}-
{\color{blue}T})^{\frac{1}{2}}.
\end{equation}
\item[({\rm Left})~]{}
A matrix \mbox{$A_L\in\C^{m\times m}$} satisfies the equation
\begin{equation}\label{eq:Left_Hyper_Lyapunov}
A_L+{A_L}^*={\color{blue}T}+A_L{\color{blue}T}{A_L}^{\color{red}*},
\end{equation}
if and only if, for some \mbox{$UU^*=I_m=U^*U$,} this $A_L$ can be written as
\begin{equation}\label{eq:A_L}
A_L={\color{blue}T}^{-1}+({\color{blue}T}^{-1}-{\color{blue}T})^{\frac{1}{2}}
U{\color{blue}T}^{-\frac{1}{2}}.
\end{equation}
\end{itemize}
\end{Pn}

{\bf Proof :}~
(Right)\\
Eq. \eqref{eq:A_R} means that the matrix \mbox{$U={\color{blue}T}^{\frac{1}{2}}
(A_R-{\color{blue}T}^{-1})({\color{blue}T}^{-1}-{\color{blue}T})^{
-\frac{1}{2}}$} is unitary, namely
\[
\begin{matrix}
U^*U=\overbrace{({\color{blue}T}^{\frac{1}{2}}(A_R-{\color{blue}T}^{-1})
({\color{blue}T}^{-1}-{\color{blue}T})^{-\frac{1}{2}})^*}^{U^*}
\overbrace{({\color{blue}T}^{\frac{1}{2}}(A_R-{\color{blue}T}^{-1})
({\color{blue}T}^{-1}-{\color{blue}T})^{-\frac{1}{2}})}^U\\~\\
=({\color{blue}T}^{-1}-{\color{blue}T})^{-\frac{1}{2}}({A_R}^*-{\color{blue}T}^{-1})
{\color{blue}T}(A_R-{\color{blue}T}^{-1}){\color{blue}T}^{-1}-
{\color{blue}T})^{-\frac{1}{2}}\\~\\=({\color{blue}T}^{-1}-
{\color{blue}T})^{-\frac{1}{2}}(\underbrace{{A_R}^*{\color{blue}T}
A_R-A_R-{A_R}^*}_x+{\color{blue}T}^{-1})({\color{blue}T}^{-1}-
{\color{blue}T})^{-\frac{1}{2}}.
\end{matrix}
\]
Now by Eq. \eqref{eq:Right_Hyper_Lyapunov} $x$ should be equal to $-{\color{blue}T}$,
which is equivalent to having $U^*U=I_m$, so this part is established.
\smallskip

(Left)\\
Eq. \eqref{eq:A_L}
means that the matrix \mbox{$U=({\color{blue}T}^{-1}-{\color{blue}T})^{-\frac{1}{2}}(A_L-
{\color{blue}T}^{-1}){\color{blue}T}^{\frac{1}{2}}$} is unitary, namely
\[
\begin{matrix}
UU^*=\overbrace{({\color{blue}T}^{-1}-{\color{blue}T})^{-\frac{1}{2}}
(A_L-{\color{blue}T}^{-1}){\color{blue}T}^{\frac{1}{2}}}^U
\overbrace{({\color{blue}T}^{-1}-{\color{blue}T})^{-\frac{1}{2}}
(A_L-{\color{blue}T}^{-1}){\color{blue}T}^{\frac{1}{2}})^*}^{U^*}\\~\\=
({\color{blue}T}^{-1}-{\color{blue}T})^{-\frac{1}{2}}(A_L-{\color{blue}T}^{-1})
{\color{blue}T}({A_L}^*-{\color{blue}T}^{-1})({\color{blue}T}^{-1}-
{\color{blue}T})^{-\frac{1}{2}}\\~\\=({\color{blue}T}^{-1}-
{\color{blue}T})^{-\frac{1}{2}}(\underbrace{A_LT{A_L}^*-A_L-{A_L}^*}_y+
{\color{blue}T}^{-1})({\color{blue}T}^{-1}-{\color{blue}T})^{-\frac{1}{2}}.
\end{matrix}
\]
Next by Eq. \eqref{eq:Left_Hyper_Lyapunov} $y$ should be equal to $-{\color{blue}T}$,
which is equivalent to having $UU^*=I_m$, so the claim is established.
\qed
\smallskip

\begin{Rk}
{\rm
In fact, Eq. \eqref{eq:A_R} can be obtained by substituting in Eq. \eqref{eq:Affine_G_o_F}
\mbox{$F(s)\equiv A_R$} and \mbox{$G_o(s)\equiv U$.}
}
$\T$
\end{Rk}

Technically, we have the following.

\begin{Cy}
A matrix $A\in\C^{m\times m}$ satisfies both Eqs. \eqref{eq:Right_Hyper_Lyapunov} and
\eqref{eq:Left_Hyper_Lyapunov}, if and only if, this $A$ commutes with ${\color{blue}T}$.
\end{Cy}

{\bf Proof :}
From Eqs. \eqref{eq:A_R} and \eqref{eq:A_L} it is clear that $A$ commutes with
${\color{blue}T}$, if and only if, the matrices $U$ and ${\color{blue}T}$ commute.
Thus, we next examine commutativity of $U$ and ${\color{blue}T}$.
\smallskip

If the matrices $U$ and ${\color{blue}T}$ commute, then Eqs. \eqref{eq:A_R} and
\eqref{eq:A_L} degenerate to,
\[
A={\color{blue}T}^{-1}+U({\color{blue}T}^{-2}-I_m)^{\frac{1}{2}}.
\]
For the converse, assume now that $A_L=A_R$, namely
\[
A_L={\color{blue}T}^{-1}+({\color{blue}T}^{-1}-{\color{blue}T})^{\frac{1}{2}}U
{\color{blue}T}^{-\frac{1}{2}}={\color{blue}T}^{-1}+{\color{blue}T}^{-\frac{1}{2}}
U({\color{blue}T}^{-1}-{\color{blue}T})^{\frac{1}{2}}=A_R
\]
This means that
\[
({\color{blue}T}^{-1}-{\color{blue}T})^{\frac{1}{2}}U{\color{blue}T}^{-\frac{1}{2}}
={\color{blue}T}^{-\frac{1}{2}}U({\color{blue}T}^{-1}-{\color{blue}T})^{\frac{1}{2}},
\]
and thus
\begin{equation}\label{eq:Comutativity_Condition}
(I_m-{\color{blue}T}^{2})^{\frac{1}{2}}U=U (I_m-{\color{blue}T}^{2})^{\frac{1}{2}}.
\end{equation}
Now, without loss of generality, one can write 
\[
{\color{blue}T}=V^*\left(\begin{smallmatrix}t_1&&~&&~\\~\\~&&\ddots&&~\\~\\~&&~&&t_m
\end{smallmatrix}\right)V\quad\quad\begin{smallmatrix}V^*V=I_m=VV^*
\\~\\t_1,~\ldots~,~t_m\in(0,~1),\end{smallmatrix}
\]
and hence \mbox{$(I_m-{\color{blue}T}^{2})^{\frac{1}{2}}=V^*D_{\color{blue}T}V$,} where
\mbox{$D_{\color{blue}T}:=\left(\begin{smallmatrix}\sqrt{1-t_1^2}&&~&&~\\~\\~&&\ddots&&~
\\~\\~&&~&&\sqrt{1-t_m^2}\end{smallmatrix}\right)$.} Namely $D_{\color{blue}T}$ commutes
with $VUV^*$, which is equivalent to having
$\left(\begin{smallmatrix}t_1&&~&&~\\~\\~&&\ddots&&~\\~\\~&&~&&t_m\end{smallmatrix}\right)$
commuting with $VUV^*$, which in turn means that indeed ${\color{blue}T}$ and $U$ commute,
so the claim is established.
\qed

\section{Realization of {\em canonical} $\mathcal{HP}_{\color{blue}T}$  functions}
\label{Sec:Realization}
\setcounter{equation}{0}

\subsection{Proof of part B of Theorem \ref{Tm:Kyp_Hyper_Pos_W}}~
\label{Subsec:Structure}

Part (i) is immediate from the definition of {\em canonical} Hyper-Positive functions,
see also \cite[Theorem 3.4]{AlpayLew2024}.
\bigskip

(ii)~ Recall (see e.g. \cite[Eq. (4.2)]{AlpayLew2025a}) that the quadratic form of Eq.
\eqref{eq:HP_Delta_KYP} is,
\begin{equation}\label{eq:Quadratic_Kyp_HP}
\underbrace{\left(\begin{smallmatrix}A&&B\\~\\C&&D\\~\\I_n&&0\\~\\0&&I_m\end{smallmatrix}
\right)^*}_{\left(\begin{smallmatrix}R_F\\~\\I_{n+m}\end{smallmatrix}\right)^*}
\underbrace{\left(\begin{smallmatrix}~0&&~0&&-H&&~0\\~\\~0&&-{\color{blue}T}&&~0&&~I_m
\\~\\-H&&~0&&~0&&~0\\~\\~0&&~I_m&&~0&&-{\color{blue}T}\end{smallmatrix}\right)}_V
\underbrace{\left(\begin{smallmatrix}A&&B\\~\\C&&D\\~\\I_n&&0\\~\\0&&I_m\end{smallmatrix}
\right)}_{\left(\begin{smallmatrix}R_F\\~\\I_{n+m}\end{smallmatrix}\right)}\in
\overline{\mathbf P}_{n+m}~.
\end{equation}
Thus, the quadratic form of Eq. \eqref{eq:KYP_Canonical_T} amounts to having {\em zero}
right-hand side, i.e.
\[
\left(\begin{smallmatrix}R_F\\~\\I_{n+m}\end{smallmatrix}\right)^*V\left(
\begin{smallmatrix}R_F\\~\\I_{n+m}\end{smallmatrix}\right)=0_{(n+m)\times(n+m)}~.
\]
Multiply the last relation by $({R_F}^{-1})^*$ and ${R_F}^{-1}$ from the left and from
the right, respectively, to obtain
\[
\left(\begin{smallmatrix}I_{n+m}\\~\\{R_F}^{-1}\end{smallmatrix}\right)^*V
\left(\begin{smallmatrix}I_{n+m}\\~\\{R_F}^{-1}\end{smallmatrix}\right)=
0_{(n+m)\times(n+m)}~.
\]
Since the two $2(n+m)\times{2(n+m)}$ matrices: $\left(\begin{smallmatrix}0&&I_{n+m}
\\~\\I_{n+m}&&0\end{smallmatrix}\right)$ and the above $V$ commute, this is equivalent
to having
\[
\left(\begin{smallmatrix}{R_F}^{-1}\\~\\I_{n+m}\end{smallmatrix}\right)^*V\left(
\begin{smallmatrix}{R_F}^{-1}\\~\\I_{n+m}\end{smallmatrix}\right)=0_{(n+m)\times(n+m)}~,
\]
hence the claim is established.
\qed
\smallskip

\begin{Rk}
{\rm
The right-hand side of Eqs. \eqref{eq:HP_Delta_KYP}, \eqref{eq:KYP_Canonical_T} may have
the following physical interpretation.
\smallskip

J.C. Willems quantified the ``dissipation rate" of a state-space equation,
\[
\dot{x}=Ax+Bu,\quad\quad\quad y=Cx+Du,\quad {\rm where}\quad \begin{smallmatrix}
x\in\R^n\\~\\ u, y\in\R^m,\end{smallmatrix}
\]
see Eq.  \eqref{eq:Realization}. In \cite[Theorem 4]{Will1972b} it is stated
that this system is {\em dissipative} with respect to the supply rate
\mbox{$\langle u,~y\rangle$} (an inner product of the input and the output), and the
(quadratic) dissipation rate is \mbox{$\|Mx+Nu\|_2^2$,} with parameters
$M\in\C^{r\times n}$ and $N\in\C^{r\times m}$ for some $r$.
\smallskip

Now, in the current framework, employ Eq. \eqref{eq:HP_Delta_KYP} for
\mbox{$I_m\succ{\color{blue}T}\succ 0$} and take $r=m$, to obtain
\mbox{$\left(\begin{smallmatrix}M^*\\~\\N^*\end{smallmatrix}\right)
\left(\begin{smallmatrix}M&&N \end{smallmatrix}\right)
=\left(\begin{smallmatrix}C&&D\\~\\0_{n\times m}&&I_m\end{smallmatrix}\right)^*
\left(\begin{smallmatrix}{\color{blue}T}&&0
\\~\\0&&{\color{blue}T}\end{smallmatrix}\right)
\left(\begin{smallmatrix}C&&D\\~\\0_{n\times m}&&I_m\end{smallmatrix}\right)
$.}
}
$\T$
\end{Rk}

\subsection{Proof of Corollary \ref{Cy:Parametrization_Of_Realizations}}
\label{Subsec:Parametrization_of_Realization}

(i)~ Recall that the original Eq. \eqref{eq:KYP_Canonical_T} reads,
\[
\left(\begin{smallmatrix}-H&&0\\~\\0&&I_m\end{smallmatrix}\right)R_F+{R_F}^*
\left(\begin{smallmatrix}-H&&0\\~\\0&&I_m\end{smallmatrix}\right)=
\left(\begin{smallmatrix}C~&&D\\~\\0_{m\times n}&&I_m\end{smallmatrix}\right)^*
\left(\begin{smallmatrix}{\color{blue}T}&&0\\~\\0&&{\color{blue}T}\end{smallmatrix}\right)
\left(\begin{smallmatrix}C~&&D\\~\\0_{m\times n}&&I_m\end{smallmatrix}\right),
\]
for some $H\in\mathbf{P}_n$. Consider now the following change of coordinates
\[
R_F~\longrightarrow~\left(\begin{smallmatrix}H^{-\frac{1}{2}}&0\\~\\0~&I_m\end{smallmatrix}
\right)R_F\left(\begin{smallmatrix}H^{\frac{1}{2}}&0\\~\\0&I_m\end{smallmatrix}\right)
={\footnotesize\left(\begin{array}{c|c}A&B\\ \hline C&D\end{array}\right)}.
\]
Writing down Eq. \eqref{eq:KYP_Canonical_T} explicitly, in the new coordinates, yields
\begin{equation}\label{eq:Explicit_Canonical_HP_T_Kyp}
\left(\begin{smallmatrix}-A-A^*&&C^*-B\\~\\C-B^*&&D+D^*\end{smallmatrix}\right)=
\left(\begin{smallmatrix}C^*{\color{blue}T}C&&C^*{\color{blue}T}D\\~\\D^*{\color{blue}T}C&&
D^*{\color{blue}T}D+{\color{blue}T}\end{smallmatrix}\right).
\end{equation}
First, the lower-right equation is \mbox{$D+D^*=D^*{\color{blue}T}D+{\color{blue}T}$.}
Note that in the (degenerate) scalar case the solution is \mbox{$d=\frac{1}{t}
(1-e^{i\theta}{\scriptstyle\sqrt{1-t^2}})$,} with \mbox{$\begin{smallmatrix}\theta
\end{smallmatrix}\in[0,~2\pi)$.}
And for \mbox{$D\in\C^{m\times m}$} this means that for some \mbox{$UU^*=I_m=U^*U$,}
\[
D={\color{blue}T}^{-\frac{1}{2}}(I_m-U(I_m-{\color{blue}T}^2)^{\frac{1}{2}})
{\color{blue}T}^{-\frac{1}{2}}={\color{blue}T}^{-1}-{\color{blue}T}^{-\frac{1}{2}}
U({\color{blue}T}^{-1}-{\color{blue}T})^{\frac{1}{2}}.
\]
Next, substituting this $D$ in the upper-right equation,
\mbox{$C^*-B=C^*{\color{blue}T}D$,} yields,
\[
B=C^*(I_m-TD)=C^*{\color{blue}T}^{\frac{1}{2}}U({\color{blue}T}^{-1}-
{\color{blue}T})^{\frac{1}{2}}.
\]
Now the upper-left equation is
\begin{equation}\label{eq:Lyapunov_Canonical_Realization}
-A-A^*=C^*{\color{blue}T}C,
\end{equation}
and thus
\[
A=-{\scriptstyle\frac{1}{2}}C^*{\color{blue}T}C+Y\quad\quad\begin{smallmatrix}
Y\in\C^{n\times n}\\~\\Y+Y^*=0.\end{smallmatrix}
\]
This means that \mbox{${\rm spec}(A)\subset\overline{\C}_L$.} However, we need to guarantee
that $A$ is in fact Hurwitz stable. This is next addressed.
\smallskip

Note that Eq. \eqref{eq:Lyapunov_Canonical_Realization}  suits the framework of a Lyapunov
equation, where the right-hand side is {\em semi}-definite, which means that the pair
$(A, C^*{\color{blue}T}C)$ must be observable, see e.g.
\cite[Theorem 2.4.7, Remark 2.4.9]{HornJohnson2}.
\smallskip

Let now $Y\in\C^{n\times n}$ and $C\in\C^{m\times n}$ be arbitrary, and
\mbox{${\color{blue}T}\succ 0$.}
Then the two following pairs share the same observable subspace,
\[
(Y,~C)\quad\quad\quad(Y-{\scriptstyle\frac{1}{2}}C^*{\color{blue}T}C,~C^*{\color{blue}T}C).
\]
Indeed, by the PBH eigenvectors test (see e.g. \cite[Theorem 3.3]{Zhou1996}),
\mbox{$0\not=v\in\C^n$} belongs to the orthogonal complement of the observable subspace
associated with a pair \mbox{$(Y, C)$,} if and only if,
for some \mbox{$\begin{smallmatrix}\lambda\end{smallmatrix}\in\C$,}
\[
(Y-\begin{smallmatrix}\lambda\end{smallmatrix}I_n)v=0\quad{\rm and}\quad Cv=0.
\]
This in turn is equivalent to having (with the same \mbox{$0\not=v\in\C^n$} and
the same \mbox{$\begin{smallmatrix}\lambda\end{smallmatrix}\in\C$),}
\[
\left(~(Y-{\scriptstyle\frac{1}{2}}C^*{\color{blue}T}C)-
\begin{smallmatrix}\lambda\end{smallmatrix}I_n\right)v=0
\quad{\rm and}\quad C^*{\color{blue}T}Cv=0.
\]
Hence, this part of the claim is established.
\bigskip

(ii)~ 
Substituting in Eq. \eqref{eq:Realization_Canonical}
\mbox{${\color{blue}T}=\begin{smallmatrix}{\color{blue}\beta}\end{smallmatrix}I_m$}
results in,
\[
R_F=\left({\footnotesize\begin{array}{c|c}{\scriptstyle-\frac{\color{blue}\beta}{2}}
C^*C+Y&\begin{smallmatrix}\sqrt{1-{\color{blue}\beta}^2}\end{smallmatrix}C^*U\\ \hline
C&{\scriptstyle\frac{1}{\color{blue}\beta}}(I_m-
{\scriptstyle\sqrt{1-{\color{blue}\beta}^2}}U)\end{array}}\right).
\]
Now, to obtain the {\em balanced} realization $\hat{R}_F$ from Eq.
\eqref{eq:Realization_Balanced_Canonical}, apply the change of coordinates
\[
\hat{R}_F=\left(\begin{smallmatrix}(1-{\color{blue}\beta}^2)^{-\frac{1}{4}}I_n&0\\0&I_m
\end{smallmatrix}\right)R_F\left(\begin{smallmatrix}
(1-{\color{blue}\beta}^2)^{\frac{1}{4}}I_n&0\\0&I_m\end{smallmatrix}\right),
\]
so this item is established.
\smallskip

(iii)~ Recall now that balanced realization means, see e.g. \cite[Section 3.9]{Zhou1996},
one has that \mbox{$H_o=\hat{H}_{\rm cont}=\hat{H}_{\rm obs}$} for some
$H_o\in\mathbf{P}_n$, namely
\begin{equation}\label{eq:Balanced_Realization}
-(H_oA^*+AH_o)=BB^*\quad\quad{\rm and}\quad\quad-(H_oA+A^*H_o)=C^*C.
\end{equation}
Substituting the data from Eq. \eqref{eq:Realization_Balanced_Canonical} results in
\[
-(H_o({\scriptstyle-\frac{\color{blue}\beta}{2}}C^*C+Y)+
({\scriptstyle-\frac{\color{blue}\beta}{2}}C^*C+Y)^*H_o)=
\begin{smallmatrix}\sqrt{1-{\color{blue}\beta}^2}\end{smallmatrix}C^*C=
-(H_o({\scriptstyle-\frac{\color{blue}\beta}{2}}C^*C+Y)^*+
({\scriptstyle-\frac{\color{blue}\beta}{2}}C^*C+Y)H_o)).
\]
Namely, the balanced Gramians are 
\mbox{$H_o=\begin{smallmatrix}\frac{\sqrt{1-{\color{blue}\beta}^2}}{\color{blue}\beta}
\end{smallmatrix}I_n$~.} Next, recall that the Hankel singular values
associated with a rational function, can be obtained form the controllability and the
observability Gramians, see e.g. \cite[Section 3.9]{Zhou1996}. Here, indeed the Hankel
singular values are just the entries along the diagonal of the balanced Gramian $H_o$,
as in Eq. \eqref{eq:Equal_Hankel_Singular_Values}, so this part is established.
\smallskip

(iv)~ 
Equal Hankel singular valued. ~The assumption implies that here, $H_{\rm cont}$ and
$H_{\rm obs}$, the controllability and observability Gramians, associated with $\Phi(s)$,
are so that their product satisfy
\mbox{$H_{\rm cont}H_{\rm obs}=\begin{smallmatrix}\alpha\end{smallmatrix}I_n$,}
for some $\begin{smallmatrix}\alpha\end{smallmatrix}>0$. Thus, $H_o$, the balanced Gramians
are \mbox{$H_o=\begin{smallmatrix}\sqrt{\alpha}\end{smallmatrix}I_n$.} This implies 
$\Phi(s)$ admits a {\em balanced} realization of the form,
\[
R_{\Phi}=\left({\footnotesize\begin{array}{c|c}{\scriptstyle-\frac{1}{2\sqrt{\alpha}}}C^*C+Y&
C^*U\\ \hline
C& *
\end{array}}\right),
\]
where the parameters are,
\[
\begin{smallmatrix}U\in\C^{m\times m}&&U^*U=I_m&&&&Y\in\C^{n\times n}&&Y+Y^*=0\\~\\
C\in\C^{m\times n}&&{\rm is~of~a~full~rank}&&&&{\rm the~pair}~~Y,~C&&{\rm is~observable}.
\end{smallmatrix}
\]
Indeed, one must have
\[
\underbrace{\begin{smallmatrix}\sqrt{\alpha}\end{smallmatrix}I_n}_{H_o}(A+A^*)=-BB^*=-C^*C.
\]
Now, in principle, one can always find $D$ so that the {\em balanced} realization in Eq.
\eqref{eq:Realization_Balanced_Canonical} is obtained, and the proof in complete.
\qed
\smallskip

We now illustrate item (ii) of Corollary \ref{Cy:Parametrization_Of_Realizations}

\begin{Ex}\label{Ex:Realization_Canonical}
{\rm
Substitute in the left-hand side of Eqs. \eqref{al:Phi_3}, \eqref{al:Phi_4}
\mbox{$c=d=a,$} and \mbox{$\gamma=\delta=a,$} to obtain a scalar {\em canonical}
$\mathcal{HP}_{\color{blue}\beta}$ function, of McMillan degree two,
\[
\begin{matrix}
\phi_{3,4}(s)&=&{\scriptstyle\frac{1}{\color{blue}\beta}}
\pm\frac{\scriptstyle\sqrt{1-{\scriptstyle{\color{blue}\beta}^2}}}
{\scriptstyle\color{blue}\beta}\frac{(s-a)^2}{(s+a)^2}
={\scriptstyle\frac{1}{\color{blue}\beta}}
\pm\frac{\scriptstyle\sqrt{1-{\scriptstyle{\color{blue}\beta}^2}}}
{\scriptstyle\color{blue}\beta}\mp
{\scriptstyle\frac{\sqrt{1-
{\scriptstyle{\color{blue}\beta}^2}}}{\color{blue}\beta}}
\frac{4as}
{(s+a)^2}
\end{matrix}
\quad\quad
\begin{smallmatrix}a>0\\~\\{\color{blue}\beta}\in(0,~1).\end{smallmatrix}
\]
It turns out that the corresponding balanced realizations are, 
\[
R_{\phi_{3,4}}=\left({\footnotesize\begin{array}{rr|c}-a&-2a&
\mp\sqrt{2a}\frac{(1-{\scriptstyle\color{blue}\beta}^2)^{\frac{1}{4}}}
{\scriptstyle\sqrt{\color{blue}\beta}}\\0&-a&\mp\sqrt{2a}
\frac{(1-{\scriptstyle\color{blue}\beta}^2)^{\frac{1}{4}}}
{\scriptstyle\sqrt{\color{blue}\beta}}\\ \hline
\sqrt{2a}\frac{(1-{\scriptstyle\color{blue}\beta}^2)^{\frac{1}{4}}}
{\scriptstyle\sqrt{\color{blue}\beta}}
&
\sqrt{2a}\frac{(1-{\scriptstyle\color{blue}\beta}^2)^{\frac{1}{4}}}
{\scriptstyle\sqrt{\color{blue}\beta}}
&\frac{1}{\color{blue}\beta}
\pm\frac{\sqrt{1-{\scriptstyle\color{blue}\beta}^2}}
{\color{blue}\beta}
\end{array}}\right).
\]
In this case, in Eq. \eqref{eq:Realization_Canonical}, the upper-left block is
comprized of
\mbox{${C_3}^*C_3={C_4}^*C_4=\begin{smallmatrix}\frac{2a}{\color{blue}\beta}
\sqrt{1-{\color{blue}\beta}^2}\end{smallmatrix}\left(\begin{smallmatrix}1&&1
\\~\\1&&1\end{smallmatrix}\right)$} and 
\mbox{$Y_{3,4}=\begin{smallmatrix}\frac{2a}{\color{blue}\beta}
\sqrt{1-{\color{blue}\beta}^2}\end{smallmatrix}\left(\begin{smallmatrix}~0&&1
\\~\\-1&&0\end{smallmatrix}\right)$.} Note that each pair \mbox{$Y_3, C_3~$,}
and \mbox{$Y_4, C_4~$,} is observable.\\
Finally, here \mbox{$U_{3,4}=\mp 1$.}\quad
Note now that substituting here 
\[
-(H_o{A_{3,4}}^*+A_{3,4}H_o)=B_3{B_3}^*=B_4{B_4}^*\quad{\rm and}\quad
-(H_oA_{3,4}+{A_{3,4}}^*H_o)={C_3}^*C_3={C_4}^*C_4~,
\]
results in the balanced Gramians \mbox{$H_o=
\begin{smallmatrix}\frac{\sqrt{1-{\color{blue}\beta}^2}}
{\color{blue}\beta}\end{smallmatrix}I_n~$.}
}
$\T$
\end{Ex}

\subsection{Proof of part B of Theorem \ref{Tm:Convexity_Sets_Of_Realizations}}~
\label{Subsec:Convexity_of_Realizations}

(i)~ By assumption, Eq. \eqref{eq:KYP_Canonical_T} holds, for $j=0,~1$ i.e.
\[
\underbrace{
\left(\begin{smallmatrix}-H&&0\\~\\0&&I_m\end{smallmatrix}\right)
\left(\begin{smallmatrix}A_j&&B_j\\~\\C_j&&D_j\end{smallmatrix}\right)
+
\left(\begin{smallmatrix}A_j&&B_j\\~\\C_j&&D_j\end{smallmatrix}\right)^*
\left(\begin{smallmatrix}-H&&0\\~\\0&&I_n\end{smallmatrix}\right)
}_{{\rm Left}_j}
=
\underbrace{
\left(\begin{smallmatrix}C_j~&&D_j\\~\\0_{m\times n}&&I_m\end{smallmatrix}\right)^*
\left(\begin{smallmatrix}{\color{blue}T}&&0\\~\\0
&&{\color{blue}T}\end{smallmatrix}\right)
\left(\begin{smallmatrix}C_j~&&D_j\\~\\0_{m\times n}&&I_m\end{smallmatrix}
\right)}_{{\rm Right}_j}.
\]
Next, for \mbox{$\begin{smallmatrix}\alpha\end{smallmatrix}\in[0,~1]$,} denote
\[
R_{\alpha}=\left({\footnotesize\begin{array}{c|c}A_{\alpha}&B_{\alpha}\\
\hline C_{\alpha}&D_{\alpha}\end{array}}\right)
=\left({\footnotesize\begin{array}{c|c}
\begin{smallmatrix}\alpha\end{smallmatrix}A_1+\begin{smallmatrix}
(1-\alpha)\end{smallmatrix}A_0&\begin{smallmatrix}\alpha\end{smallmatrix}
B_1+\begin{smallmatrix}(1-\alpha)\end{smallmatrix}B_0\\
\hline\begin{smallmatrix}\alpha\end{smallmatrix}C_1+
\begin{smallmatrix}(1-\alpha)\end{smallmatrix}
C_0&\begin{smallmatrix}\alpha\end{smallmatrix}D_1+\begin{smallmatrix}(1-\alpha)
\end{smallmatrix}D_0\end{array}}\right)
\]
and separately write each side of Eq. \eqref{eq:KYP_Canonical_T} 
\[
{\rm Left}_{\alpha}:=\left(\begin{smallmatrix}-H&&0\\~\\0&&I_m
\end{smallmatrix}\right)R_{\alpha}+{R_{\alpha}}^*
\left(\begin{smallmatrix}-H&&0\\~\\0&&I_m\end{smallmatrix}\right)
\]
and
\[
{\rm Right}_{\alpha}:=\left(\begin{smallmatrix}C_{\alpha}~&&D_{\alpha}
\\~\\0_{m\times n}&&I_m\end{smallmatrix}\right)^*
\left(\begin{smallmatrix}{\color{blue}T}&&0
\\~\\0&&{\color{blue}T}\end{smallmatrix}\right)
\left(\begin{smallmatrix}C_{\alpha}~&&D_{\alpha}
\\~\\0_{m\times n}&&I_m\end{smallmatrix}\right).
\]
Note now that on the one hand,
\[
{\rm Left}_{\alpha}=\begin{smallmatrix}\alpha\end{smallmatrix}{\rm Left}_1+
\begin{smallmatrix}(1-\alpha)\end{smallmatrix}{\rm Left}_0
=\begin{smallmatrix}\alpha\end{smallmatrix}{\rm Right}_1+\begin{smallmatrix}
(1-\alpha)\end{smallmatrix}{\rm Right}_0~.
\]
On the other hand, a straightforward computation yields,
\[
\begin{matrix}
{\rm Right}_{\alpha}&=&\begin{smallmatrix}\alpha\end{smallmatrix}
{\rm Right}_1+\begin{smallmatrix}(1-\alpha)\end{smallmatrix}
{\rm Right}_0-\begin{smallmatrix}\alpha(1-\alpha)\end{smallmatrix}&
\left(\begin{smallmatrix}C_1-C_0~&&D_1-D_0\\~\\
0_{m\times n}&&0_{m\times m}\end{smallmatrix}\right)^*
\left(\begin{smallmatrix}{\color{blue}T}&&0\\~\\0&&{\color{blue}T}\end{smallmatrix}
\right)\left(\begin{smallmatrix}C_1-C_0~&&D_1-D_0\\~\\0_{m\times n}&&0_{m\times m}
\end{smallmatrix}\right)\\~\\~&=&\begin{smallmatrix}\alpha\end{smallmatrix}
{\rm Right}_1+\begin{smallmatrix}(1-\alpha)\end{smallmatrix}
{\rm Right}_0-\begin{smallmatrix}\alpha(1-\alpha)\end{smallmatrix}&
\underbrace{\left(\begin{smallmatrix}(C_1-C_0)^*\\~\\(D_1-D_0)^*
\end{smallmatrix}\right)\begin{smallmatrix}{\color{blue}T}\end{smallmatrix}
\left(\begin{smallmatrix}C_1-C_0&&D_1-D_0\end{smallmatrix}\right)}_W.
\end{matrix}
\]
Thus equality holds, if and only if, 
\mbox{$
W=0_{(n+m)\times(n+m)}$.}
\smallskip

Recall now, see e.g. \cite[Theorem 1.3.22]{HJ1}, that
matrices
\[
W:=\left(\begin{smallmatrix}C_1^*-C_0^*\\~\\D_1^*-D_0^*\end{smallmatrix}\right)
\begin{smallmatrix}{\color{blue}T}\end{smallmatrix}
\left(\begin{smallmatrix}C_1-C_0&&D_1-D_0\end{smallmatrix}\right)\in
\overline{\mathbf P}_{n+m}\quad{\rm and}\quad
X:=\left(\begin{smallmatrix}C_1-C_0&&D_1-D_0\end{smallmatrix}\right)
\begin{smallmatrix}{\color{blue}T}\end{smallmatrix}
\left(\begin{smallmatrix}C_1^*-C_0^*\\~\\D_1^*-D_0^*\end{smallmatrix}\right)\in
\overline{\mathbf P}_m
\]
share the same positive eigenvalues. Thus the above condition, $W=0$, is
equivalent to $X=0$ in Eq. \eqref{eq:Convex_Realization_Canonical}. Hence
this part is established.
\smallskip

Assume now that
for some \mbox{$I_m\succ{\color{blue}T}\succ 0$,} $R_{F_0}$ and $R_{F_1}$
are two realizations of {\em canonical} $\mathcal{HP}_{\color{blue}T}$
functions, of the same McMillan degree. Thus, they share the structure
described in Eq. \eqref{eq:Realization_Canonical}. Next, assuming that
$R_{F_{\alpha}}$ is a realization of a {\em canonical}
$\mathcal{HP}_{\color{blue}T}$ function, implies that \mbox{$D_0=D_1$}
and thus $U$ is identical in both realizations. Since also,
\mbox{$C_0=C_1=C$,} it implies that \mbox{$B_0=B_1$} and in addition
\mbox{$A_0+A_0^*=A_1+A_1^*=C^*{\color{blue}T}C$.} Namely, in the terminology
of Eq.  \eqref{eq:Realization_Canonical}, $A_0$ and $A_1$ differ only in the
$Y$ part. Thus, this part of the proof is complete.
\bigskip

(ii)~
If $X\in\overline{\mathbf P}_m$ is non-zero, it has say $q$ positive eigenvalue,
for some $q\in[1,~m]$. Then also $W\in\overline{\mathbf P}_{n+m}$ has $q$
positive eigenvalue, and thus using the notation of part (i),
\[
\begin{matrix}
{\rm Right}_{\alpha}&=&
\left(\begin{smallmatrix}C_{\alpha}~&&D_{\alpha}
\\~\\0_{m\times n}&&I_m\end{smallmatrix}\right)^*
\left(\begin{smallmatrix}{\color{blue}T}&&0
\\~\\0&&{\color{blue}T}\end{smallmatrix}\right)
\left(\begin{smallmatrix}C_{\alpha}~&&D_{\alpha}
\\~\\0_{m\times n}&&I_m\end{smallmatrix}\right)
&
-\begin{smallmatrix}\alpha(1-\alpha)\end{smallmatrix}
\underbrace{\left(\begin{smallmatrix}(C_1-C_0)^*\\~\\(D_1-D_0)^*
\end{smallmatrix}\right)\begin{smallmatrix}{\color{blue}T}\end{smallmatrix}
\left(\begin{smallmatrix}C_1-C_0&&D_1-D_0\end{smallmatrix}\right)}_
{W\in\overline{\mathbf P}_{n+m}}
\\~\\~&\succcurlyeq&
\left(\begin{smallmatrix}C_{\alpha}~&&D_{\alpha}
\\~\\0_{m\times n}&&I_m\end{smallmatrix}\right)^*
\left(\begin{smallmatrix}{\color{blue}T}&&0
\\~\\0&&{\color{blue}T}\end{smallmatrix}\right)
\left(\begin{smallmatrix}C_{\alpha}~&&D_{\alpha}
\\~\\0_{m\times n}&&I_m\end{smallmatrix}\right),
&~
\end{matrix}
\]
so indeed for \mbox{${\scriptstyle\alpha}\in(0,~1)$,} $F_{\alpha}$ is a
non-{\em canonical} $\mathcal{HP}_{\color{blue}T}$ function.
Thus the proof is complete.
\qed
\bigskip

Recall that by item (i) of part {\bf A} Theorem
\ref{Tm:Convexity_Sets_Of_Realizations}, the set of \mbox{$(n+m)\times(n+m)$}
realizations of $\mathcal{HP}_{\color{blue}T}$ functions, satisfying
Eq. \eqref{eq:HP_Delta_KYP} with a prescribed $H\in\mathbf{P}_n$, is convex. In
contrast, item (i) part {\bf B} of Theorem \ref{Tm:Convexity_Sets_Of_Realizations},
shows that in the subclass of {\em canonical} functions, the conditions for
convexity of realizations, are far more restrictive.
\smallskip

\subsection{Convex Combination of Realizations}
\label{Subsec:Conex_Realizations}

Convex combination of $\mathcal{HP}_{\color{blue}T}$ rational functions
({\em canonical} and non-{\em canonical}) was addressed in Theorem
\ref{Tm:Set_Of_HP_Functions}.
Parallel discussion in the state-space framework, was given in Theorem
\ref{Tm:Convexity_Sets_Of_Realizations}. We here compare these two aspects.

\begin{Rk}\label{Rk:Sum_Functions_vs_Arrays})
{\rm
Consider the summation of a pair of $m\times m$-valued functions $F_1$, $F_2$
(taking a convex combination is a particular case). Technically, in the
framework of rational functions, the McMillan degree typically increases (up
to the sum of the original degrees). In contrast, while summing up the
corresponding realization arrays $R_{F_1}$, $R_{F_2}$, they must be of the
same dimensions, say the same McMillan degree $n$. Then, the sum may be viewed
as a realization is of a $m\times m$-valued
function of McMillan degree (of at most) $n$.
}
\end{Rk}
$\T$

To simplify the discussion, we now introduce an alternative formulation of the
KYP-type results.

Here are the details. Let $R_F\in\R^{(n+m)\times(n+m)}$ be a realization array as in Eq.
\eqref{eq:Realization}. A functions $F(s)$ is in $\mathcal{HP}_{\color{blue}T}$, with
\mbox{$I_m\succ{\color{blue}T}\succcurlyeq 0$,}
if there exist and let $H\in\mathbf{P}_n$ satisfying Eq.  \eqref{eq:HP_Delta_KYP}, i.e.
\begin{equation}\label{eq:HP_Delta_KYP_Again}
\underbrace{
\left(\begin{smallmatrix}-H&&0\\~\\0&&I_m\end{smallmatrix}\right)R_F+{R_F}^*
\left(\begin{smallmatrix}-H&&0\\~\\0&&I_m\end{smallmatrix}\right)}_{\rm Left}\succcurlyeq
\underbrace{
\left(\begin{smallmatrix}C~&&D\\~\\0_{m\times n}&&I_m\end{smallmatrix}\right)^*
\left(\begin{smallmatrix}{\color{blue}T}&&0\\~\\0&&{\color{blue}T}\end{smallmatrix}\right)
\left(\begin{smallmatrix}C~&&D\\~\\0_{m\times n}&&I_m\end{smallmatrix}\right)}_{\rm Right}.
\end{equation}
Now, item {\bf A} (i) of Theorem \ref{Tm:Convexity_Sets_Of_Realizations} asserts that for
prescribed $H$, ${\color{blue}T}$, the set of all \mbox{$(n+m)\times(n+m)$} realization
arrays $R_F$ satisfying Eq. \eqref{eq:HP_Delta_KYP_Again}, is convex.
\smallskip

Now, let us denote,
\[
\left(\begin{smallmatrix}-H^{\frac{1}{2}}&&0\\~\\0&&I_m\end{smallmatrix}\right)
R_F
\left(\begin{smallmatrix}H^{-\frac{1}{2}}&&0\\~\\0&&I_m\end{smallmatrix}\right)
=\left(\begin{smallmatrix}-H^{\frac{1}{2}}AH^{-\frac{1}{2}}&&-H^{\frac{1}{2}}B\\~\\
CH^{-\frac{1}{2}}&&D\end{smallmatrix}\right)
=
\underbrace
{\left(\begin{smallmatrix}\hat{A}&&\hat{B}\\~\\ \hat{C}&&D\end{smallmatrix}\right)
}_{
\hat{R}_F
}
\quad\quad{\rm and}\quad\quad
\underbrace{
\left(\begin{smallmatrix}0_{n\times n}&&0\\~\\0&&{\color{blue}T}
\end{smallmatrix}\right)
}_{\Theta}.
\]
\[
\hat{R}_F:=
\left(\begin{smallmatrix}-H^{\frac{1}{2}}&&0\\~\\0&&I_m\end{smallmatrix}\right)
R_F
\left(\begin{smallmatrix}H^{-\frac{1}{2}}&&0\\~\\0&&I_m\end{smallmatrix}\right)
=\left(\begin{smallmatrix}\hat{A}&&\hat{B}\\~\\ \hat{C}&&D\end{smallmatrix}\right)
\quad\quad{\rm and}\quad\quad\Theta:=\left(\begin{smallmatrix}0_{n\times n}&&0
\\~\\0&&{\color{blue}T}\end{smallmatrix}\right).
\]
For matrices, the Hyper-Lyapunov inclusions were introduced in \cite{Lewk2024a}. We
now cast in this framework, a KYP result 

\begin{La}
Eq. \eqref{eq:HP_Delta_KYP_Again} can be equivalently written as the Hyper-Lyapunov
inclusion,
\begin{equation}\label{eq:Kyp_HP_as_Hyper_Lyap_2}
\hat{R}_F+{\hat{R}_F}^*\succcurlyeq\Theta+{\hat{R}_F}^*\Theta\hat{R}_F~.
\end{equation}
\end{La}

{\bf Proof :}~ Indeed, 
multiplying Eq. \eqref{eq:HP_Delta_KYP_Again} by
$\left(\begin{smallmatrix}H^{-\frac{1}{2}}&&0\\~\\0&&I_m\end{smallmatrix}\right)$,
from the left and from the right,
and 
denoting $J:=\left(\begin{smallmatrix}-I_n&&0\\~\\~0&&1_m\end{smallmatrix}\right)$,
yields
\[
\begin{matrix}
\left(\begin{smallmatrix}H^{-\frac{1}{2}}&&0\\~\\0&&I_m\end{smallmatrix}\right)
{\rm Left}
\left(\begin{smallmatrix}H^{-\frac{1}{2}}&&0\\~\\0&&I_m\end{smallmatrix}\right)
&=&
\underbrace{
\left(\begin{smallmatrix}H^{\frac{1}{2}}&&0\\~\\0&&I_m\end{smallmatrix}\right)
JR_F
\left(\begin{smallmatrix}H^{-\frac{1}{2}}&&0\\~\\0&&I_m\end{smallmatrix}\right)
}_{\hat{R}_F}
+
\underbrace{
\left(\begin{smallmatrix}H^{-\frac{1}{2}}&&0\\~\\0&&I_m\end{smallmatrix}\right)
{R_F}^*J
\left(\begin{smallmatrix}H^{\frac{1}{2}}&&0\\~\\0&&I_m\end{smallmatrix}\right)
}_{{\hat{R}_F}^*}.
\end{matrix}
\]
Now,
\[
\begin{matrix}
\left(\begin{smallmatrix}H^{-\frac{1}{2}}&&0\\~\\0&&I_m\end{smallmatrix}\right){\rm Right}
\left(\begin{smallmatrix}H^{-\frac{1}{2}}&&0\\~\\0&&I_m\end{smallmatrix}\right)&=&
\left(\begin{smallmatrix}\hat{C}&&D\\~\\0_{m\times n}&&I_m\end{smallmatrix}\right)^*
\left(\begin{smallmatrix}{\color{blue}T}&&0\\~\\0&&{\color{blue}T}\end{smallmatrix}\right)
\left(\begin{smallmatrix}\hat{C}&&D\\~\\0_{m\times n}&&I_m\end{smallmatrix}\right)
\\~\\&=&\underbrace{\left(\begin{smallmatrix}0_{m\times m}&&0\\~\\
0&&{\color{blue}T}\end{smallmatrix}\right)}_{\Theta}+\left(\begin{smallmatrix}
\hat{C}&&D\\~\\0_{m\times n}&&0_{m\times m}\end{smallmatrix}\right)^*
\left(\begin{smallmatrix}{\color{blue}T}&&0\\~\\0&&0_{m\times m}\end{smallmatrix}\right)
\left(\begin{smallmatrix}\hat{C}&&D\\~\\0_{m\times n}&&0_{m\times m}\end{smallmatrix}
\right)\\~\\&=&\underbrace{\left(\begin{smallmatrix}0_{m\times m}&&0\\~\\0&&
{\color{blue}T}\end{smallmatrix}
\right)}_{\Theta}+\left(\begin{smallmatrix}0_{m\times n}&&0_{m\times m}\\~\\ \hat{C}
&&D\end{smallmatrix}\right)^*\underbrace{\left(\begin{smallmatrix}0_{m\times m}&&0\\~\\
0&&{\color{blue}T}\end{smallmatrix}\right)}_{\Theta}\left(\begin{smallmatrix}
0_{m\times n}&&0_{m\times m}\\~\\ \hat{C}&&D\end{smallmatrix}\right)\\~\\&=&
\underbrace{\left(\begin{smallmatrix}0_{m\times m}&&0\\~\\0&&{\color{blue}T}
\end{smallmatrix}\right)}_{\Theta}+{\hat{R}_F}^*\underbrace{\left(\begin{smallmatrix}
0_{m\times m}&&0\\~\\0&&{\color{blue}T}
\end{smallmatrix}\right)}_{\Theta}{\hat{R}_F},
\end{matrix}
\]
so the claim is established.
\qed
\bigskip

We conclude this work with two examples.

\begin{Ex}
{\rm 
Consider the following pair of non-{\em canonical} $\mathcal{HP}_{\color{blue}\beta}$
functions from Figure \ref{Fig:Degree_One_HP}, 
\mbox{${\color{red}f_2(s)=\begin{smallmatrix}\frac{2}{5}\end{smallmatrix}
+\begin{smallmatrix}\frac{\frac{14}{15}a_2}{s+a_2}\end{smallmatrix}}$}
and
\mbox{${\color{orange}f_3(s)=\begin{smallmatrix}\frac{5}{2}\end{smallmatrix}
-\begin{smallmatrix}\frac{\frac{7}{4}a_3}{s+a_3}\end{smallmatrix}}$.}
In Example \ref{Ex:Combinations_in_Figure} we showed how can each be obtained as a convex
combination of a pair {\em canonical} $\mathcal{HP}_{\color{blue}\beta}$ functions.
\smallskip

We here mimic the procedure with {\em realization arrays}. Furthermore, as minimal
realization of a given function, is non-unique, in each case we present two
ways of doing that. Here are the details.

{\bf a.}~
Recall that in item {\bf a.} of Example \ref{Ex:Combinations_in_Figure} we showed
that the function
\[
{\color{red}f_2(s)=\begin{smallmatrix}\frac{2}{5}\end{smallmatrix}+\frac{
\begin{smallmatrix}\frac{14}{15}
\end{smallmatrix}a_2}{s+a_2}}\quad\quad{\color{red}a_2}>0,
\]
can be obtained as a convex combination of the {\em canonical} function
\[
{\color{blue}f_1(s)=\begin{smallmatrix}\frac{2}{5}\end{smallmatrix}+
\frac{\begin{smallmatrix}\frac{21}{10}\end{smallmatrix}a_1}{s+a_1}}~
\quad\quad{\color{blue}a_1}>0,
\]
along with the zero degree function,
\[
f_5(s)\equiv\begin{smallmatrix}\frac{2}{5}\end{smallmatrix}~.
\]
{\bf (i)}~
Consider now the realization setup:
\[
R_{\color{blue}f_1}={\footnotesize\left(\begin{array}{c|c}-a_1&\sqrt{\frac{21}{10}a_1}
\\ \hline 
\sqrt{\frac{21}{10}a_1}
&\frac{2}{5}\end{array}\right)}
\quad{\rm and}\quad
\hat{R}_{\color{blue}f_1}=
\left(\begin{smallmatrix}-1&0\\~0&1\end{smallmatrix}\right)
R_{\color{blue}f_1}
\left(\begin{smallmatrix}-1&0\\~0&1\end{smallmatrix}\right)
=
{\footnotesize\left(\begin{array}{c|c}-a_1&
-\sqrt{\frac{21}{10}a_1}
\\ \hline 
-\sqrt{\frac{21}{10}a_1}
&\frac{2}{5}\end{array}\right)},
\]
are two balanced realizations of the same {\em canonical} function
${\color{blue}f_1(s)}$.
\smallskip

Consider now a convex combination of these realizations, i.e.
for ${\scriptstyle\alpha}\in[0,~1]$,
\[
\begin{matrix}
{\scriptstyle\alpha}{R}_{\color{blue}f_1}+(1-{\scriptstyle\alpha})
\hat{R}_{\color{blue}f_1}&=&{\footnotesize\left(\begin{array}{c|c}
-a_1&(2{\scriptstyle\alpha}-1)\sqrt{\frac{21}{10}a_1}\\ \hline
(2{\scriptstyle\alpha}-1)\sqrt{\frac{21}{10}a_1}&\frac{2}{5}\end{array}\right)}.
\end{matrix}
\]
Specifically, 
for \mbox{${\scriptstyle\alpha}=\frac{5}{6}$} and $a_1=a_2$,
the combination \mbox{${\scriptstyle\alpha}{R}_{
\color{blue}f_1}+(1-{\scriptstyle\alpha})\hat{R}_{\color{blue}f_1}$} results in a
 balanced realization of the same ${\color{red}f_2(s)}$, i.e.
\[
\begin{matrix}
{\scriptstyle\alpha}{R}_{\color{blue}f_1}+(1-{\scriptstyle\alpha})
\hat{R}_{\color{blue}f_1}_{|_{\alpha=\frac{5}{6}}}
&=&
{\footnotesize\left(\begin{array}{c|c}-a_1&\sqrt{\frac{14}{15}a_1}\\
\hline\sqrt{\frac{14}{15}a_1}&\frac{2}{5}\end{array}\right)}_{|_{a_1=a_2}}
=R_{\color{red}f_2}~.
\end{matrix}
\]
\smallskip

\noindent
{\bf (ii)}~
We next present another realization of the {\em canonical}
$\mathcal{HP}_{\color{blue}\beta}$ functions ${\color{blue}f_1(s)}$ and $f_5(s)$.
The realization of ${\color{blue}f_1(s)}$ is minimal (but non-balanced), while
the realization of the zero degree function
$f_5(s)$, is non-minimal, tailored to our aim.
\[
\begin{matrix}
{\color{blue}f_1(s)}=\begin{smallmatrix}\frac{2}{5}\end{smallmatrix}+
\begin{smallmatrix}\frac{\frac{21}{10}a_1}{s+a_1}
\end{smallmatrix}
&& R_{\color{blue}f_1}=\left({\footnotesize\begin{array}{c|c}-a_1&\frac{21}{10}a_1
\\ \hline 1&\frac{2}{5}\end{array}}\right)\\~\\
f_5(s)\equiv\begin{smallmatrix}\frac{2}{5}\end{smallmatrix}
&&R_{f_5}=\left({\footnotesize\begin{array}{c|c}-a_5&0\\
\hline1&\frac{2}{5}\end{array}}\right).
\end{matrix}
\]
Consider now a convex combination of these realizations, i.e.
for ${\scriptstyle\alpha}\in[0,~1]$,
\[
\begin{matrix}
{
{\scriptstyle\alpha}{R}_{\color{blue}f_1}
}_{|_{a_1=a_2}}
+{(1-{\scriptstyle\alpha})R_{f_5}
}_{|_{a_5=a_2}}&=&
{\footnotesize\left(\begin{array}{c|c}-a_2&{\scriptstyle\alpha}\frac{21}{10}a_2
\\ \hline 1&\frac{2}{5}\end{array}\right)}.
\end{matrix}
\]
Taking \mbox{${\scriptstyle\alpha=\frac{4}{9}}$} results in a minimal (non-balanced)
realization of the same ${\color{red}f_2(s)}$, i.e.
\[
\begin{matrix}
{{\scriptstyle\frac{4}{9}}{R}_{\color{blue}f_1}}_{|_{a_1=a_2}}
+{{\scriptstyle\frac{5}{9}}R_{f_5}}_{|_{a_5=a_2}}&=&
{\footnotesize\left(\begin{array}{c|c}-a_2&\frac{14}{15}a_2
\\ \hline 1&\frac{2}{5}\end{array}\right)}
=R_{\color{red}f_2}~.
\end{matrix}
\]
\smallskip

{\bf b.~}
Recall that in item {\bf b.} of Example \ref{Ex:Combinations_in_Figure} we first computed
\[
f_6(s)=({\color{blue}f_1(s)})^{-1}
=
f_6(s)=\begin{smallmatrix}\frac{5}{2}\end{smallmatrix}-\begin{smallmatrix}
\frac{\frac{21}{10}a_6}{s+a_6}\end{smallmatrix}~,
\]
and then took a convex combination of this function along with
the zero degree function
\[
f_7(s)\equiv\begin{smallmatrix}\frac{5}{2}\end{smallmatrix}~,
\]
to obtain
\[
{\color{orange}f_3(s)=\begin{smallmatrix}\frac{5}{2}\end{smallmatrix}
-\frac{\begin{smallmatrix}\frac{7}{4}
\end{smallmatrix}a_3}{s+a_3}}\quad\quad{\color{orange}a_3}>0.
\]
{\bf (i)}~
As before, a pair of corresponding balanced realizations of $f_6(s)$
is given by,
\[
R_{f_6}={\footnotesize\left(\begin{array}{c|c}-a_6&-\sqrt{\frac{21}{10}a_6}
\\ \hline\sqrt{\frac{21}{10}a_6}&\frac{5}{2}\end{array}\right)}\quad{\rm and}
\quad\hat{R}_{f_6}=\left(\begin{smallmatrix}-1&0\\~0&1\end{smallmatrix}\right)
R_{\color{blue}f_6}\left(\begin{smallmatrix}-1&0\\~0&1\end{smallmatrix}\right)
={\footnotesize\left(\begin{array}{c|c}-a_6&\sqrt{\frac{21}{10}a_6}\\ \hline
-\sqrt{\frac{21}{10}a_6}&\frac{5}{2}\end{array}\right)}.
\]
Consider now a convex combination of these realizations, i.e. for
${\scriptstyle\alpha}\in[0,~1]$,
\[
{\scriptstyle\alpha}R_{f_6}+(1-{\scriptstyle\alpha})\hat{R}_{f_6}
={\footnotesize\left(\begin{array}{c|c}-a_6&
(1-2{\scriptstyle\alpha})\sqrt{\frac{21}{10}a_6}\\ \hline 
(2{\scriptstyle\alpha}-1)\sqrt{\frac{21}{10}a_6}
&\frac{5}{2}\end{array}\right)}
_{|_{\alpha\approx 0.956,~a_6=a_3}}~.
\]
Specifically, for \mbox{${\scriptstyle\alpha}={\scriptstyle\frac{1}{2}}
(1+{\scriptstyle\frac{\sqrt{5}}{\sqrt{6}}})
\approx 0.956$} and $a_6=a_3$ the combination 
${\scriptstyle\alpha}R_{f_6}+(1-{\scriptstyle\alpha})\hat{R}_{f_6}$ results in a
balanced realization of the same ${\color{orange}f_3(s)}$, i.e.
\[
{{\scriptstyle\alpha}R_{f_6}+(1-{\scriptstyle\alpha})
\hat{R}_{f_6}}_{|_{\alpha\approx 0.956,~a_6=a_3}}
={\footnotesize\left(\begin{array}{c|c}-a_3&-\frac{\sqrt{7a_3}}{2}\\
\hline\frac{\sqrt{7a_3}}{2}&\frac{5}{2}\end{array}\right)}=R_{\color{orange}f_3}~.
\]
\smallskip

\noindent
{\bf (ii)}~
As before, we now present another realization of the {\em canonical}
$\mathcal{HP}_{\color{blue}\beta}$ functions $f_6(s)$ and the zero degree
(and thus {\em canonical}) function $f_7(s)$. The realization of
$f_6(s)$ is minimal (but non-balanced), while the realization of 
the zero degree function $f_5(s)$, is non-minimal, tailored to our aim.
\[
\begin{matrix}
f_6(s)=\begin{smallmatrix}\frac{5}{2}\end{smallmatrix}-\begin{smallmatrix}
\frac{\frac{21}{10}a_6}{s+a_6}\end{smallmatrix}&&R_{f_6}=\left({\footnotesize
\begin{array}{c|c}-a_6&-\frac{21}{10}a_6\\ \hline 1&\frac{5}{2}\end{array}}\right)
\\~\\
f_7(s)=\begin{smallmatrix}\frac{5}{2}\end{smallmatrix}&&R_{f_7}=\left(
{\footnotesize\begin{array}{c|c}-a_7&0\\ \hline 1&\frac{5}{2}\end{array}}\right).
\end{matrix}
\]
Consider now a convex combination of these realizations, i.e. for
${\scriptstyle\alpha}\in[0,~1]$,
\[
\begin{smallmatrix}\alpha\end{smallmatrix}{R_{f_6}}_{|_{a_6=a_3}}+
\begin{smallmatrix}(1-\alpha)\end{smallmatrix}{R_{f_7}}_{|_{a_7=a_3}}=
\left({\footnotesize\begin{array}{c|c}-a_3&-\alpha\frac{21}{10}a_3\\
\hline 1&\frac{5}{2}\end{array}}\right).
\]
Taking
\mbox{${\scriptstyle\alpha=\frac{5}{6}}$} results in a minimal (non-balanced) realization
of the same ${\color{orange}f_3(s)}$, i.e.
\[
\begin{smallmatrix}\frac{5}{6}\end{smallmatrix}{R_{f_6}}_{|_{a_6=a_3}} +
\begin{smallmatrix}\frac{1}{6}\end{smallmatrix}{R_{f_7}}_{|_{a_7=a_3}}=
\left({\footnotesize\begin{array}{c|c}-a_3&-\frac{7}{4}a_3\\ \hline 1&\frac{5}{2}
\end{array}}\right)=R_{\color{orange}f_3}~.
\]
}
\end{Ex}
$\T$
\smallskip

We know that convex combination of a pair of {\em canonical} $\mathcal{HP}_T$
functions typically yields a non-{\em canonical} function. This is true in both
frameworks: of rational functions and of realization arrays.
We conclude this work by illustrating that with scalar functions of degree two.

\begin{Ex}\label{Ex:Convex_Degree_Two}
{\rm
Let $\phi_3(s)$ and $\phi_4(s)$ be the degree two {\em canonical}
$\mathcal{HP}_{\color{blue}\beta}$,
\mbox{$\begin{smallmatrix}\beta\in[0,~1)\end{smallmatrix}$,} functions as
in Eqs.  \eqref{al:Phi_3} and \eqref{al:Phi_4}. To avoid increase of degree
under convex combination (see Remark \ref{Rk:Sum_Functions_vs_Arrays})
assume that \mbox{$\gamma=c>0$} and \mbox{$\delta=d>0$.} Then, one can
compactly write,
\begin{equation}\label{eq:Phi_3_and_Phi_4}
{\phi}_{3,4}(s)=\begin{smallmatrix}\frac{1}{\beta}\end{smallmatrix}
\pm\begin{smallmatrix}\frac{\sqrt{1-{\beta}^2}}{\beta}\end{smallmatrix}
\frac{(s-c)(s-d)}{(s+c)(s+d)}~.
\end{equation}
It turns out, that in this case,
the convex combination
\begin{equation}\label{eq:Convex_Canonical_Functions_Degree_Two}
{\scriptstyle\alpha}{\phi}_3(s)+{\scriptstyle(1-\alpha)}{\phi}_4(s)=
\begin{smallmatrix}\frac{1}{\beta}\end{smallmatrix}+{\scriptstyle(2\alpha-1)}
\begin{smallmatrix}\frac{\sqrt{1-{\beta}^2}}{\beta}\end{smallmatrix}
\frac{(s-c)(s-d)}{(s+c)(s+d)}~,
\quad\quad
\begin{smallmatrix}\alpha\in[0,~1]\end{smallmatrix},
\end{equation}
is a $\mathcal{HP}_{\color{blue}\beta}$ function. For
\mbox{${\scriptstyle\alpha=\frac{1}{2}}~$,} it is of degree zero. Else, it is of
degree two. For $\begin{smallmatrix}\alpha\in(0,~1)\end{smallmatrix},$ it is
non-{\em canonical}. 
\smallskip

Next, we focus on realization arrays. Here, one needs to distinguish between
two cases.\\
{\bf (i)}~ For $d\not=c$, a balanced realization of the functions in Eq.
\eqref{eq:Phi_3_and_Phi_4}, may be written as,
\[
R_{\phi_{3,4}}
=\left({\footnotesize\begin{array}{cc|c}-c&~~0&
\pm{\scriptstyle\sqrt{2c}\frac{\sqrt{d+c}}{\sqrt{d-c}}\frac{(1-
{\scriptstyle\color{blue}\beta}^2)^{\frac{1}{4}}}{\sqrt{\color{blue}\beta}}}
\\~~0&-d&\pm{i}{\scriptstyle\sqrt{2d}\frac{\sqrt{d+c}}{\sqrt{d-c}}\frac{(1-
{\scriptstyle\color{blue}\beta}^2)^{\frac{1}{4}}}{\sqrt{\color{blue}\beta}}}
\\ \hline
{\scriptstyle\sqrt{2c}\frac{\sqrt{d+c}}{\sqrt{d-c}}\frac{(1-
{\scriptstyle\color{blue}\beta}^2)^{\frac{1}{4}}}{\sqrt{\color{blue}\beta}}}
&i{\scriptstyle\sqrt{2d}\frac{\sqrt{d+c}}{\sqrt{d-c}}\frac{(1-
{\scriptstyle\color{blue}\beta}^2)^{\frac{1}{4}}}{\sqrt{\color{blue}\beta}}}
&{\scriptstyle\frac{1}{\color{blue}\beta}\pm\frac{\sqrt{1-
{\scriptstyle\color{blue}\beta}^2}}{\color{blue}\beta}}
\end{array}}\right).
\]
Hence, for $\begin{smallmatrix}\alpha\in[0,~1]\end{smallmatrix}$ a convex
combination of these realizations arrays is given by,
\[
{\scriptstyle\alpha}R_{{\phi}_3}+{\scriptstyle(1-\alpha)}R_{{\phi}_4}
=\left({\footnotesize\begin{array}{cc|c}-c&~~0&
{\scriptstyle(2\alpha-1)\sqrt{2c}\frac{\sqrt{d+c}}{\sqrt{d-c}}\frac{(1-
{\scriptstyle\color{blue}\beta}^2)^{\frac{1}{4}}}{\sqrt{\color{blue}\beta}}}
\\~~0&-d&i{\scriptstyle(2\alpha-1)\sqrt{2d}\frac{\sqrt{d+c}}
{\sqrt{d-c}}\frac{(1-
{\scriptstyle\color{blue}\beta}^2)^{\frac{1}{4}}}{\sqrt{\color{blue}\beta}}}
\\ \hline
{\scriptstyle(2\alpha-1)\sqrt{2c}\frac{\sqrt{d+c}}{\sqrt{d-c}}\frac{(1-
{\scriptstyle\color{blue}\beta}^2)^{\frac{1}{4}}}{\sqrt{\color{blue}\beta}}}
&i{\scriptstyle(2\alpha-1)\sqrt{2d}\frac{\sqrt{d+c}}{\sqrt{d-c}}\frac{(1-
{\scriptstyle\color{blue}\beta}^2)^{\frac{1}{4}}}{\sqrt{\color{blue}\beta}}}
&{\scriptstyle\frac{1}{\color{blue}\beta}+{\scriptstyle(2\alpha-1)}
\frac{\sqrt{1-{\scriptstyle\color{blue}\beta}^2}}{\color{blue}\beta}}
\end{array}}\right).
\]
For all \mbox{${\scriptstyle\alpha\in[0,~1]}$,} this can be viewed as a balanced
realization of the $\mathcal{HP}_{\color{blue}\beta}$ rational function, 
\[
\begin{smallmatrix}
\frac{1}{\color{blue}\beta}
\end{smallmatrix}
+
\begin{smallmatrix}
(2\alpha-1)\frac{\sqrt{1-{\color{blue}\beta}^2}}
{\color{blue}\beta}
\end{smallmatrix}
-
\begin{smallmatrix}
(2\alpha-1)^2\frac{\sqrt{1-{\color{blue}\beta}^2}}
{\color{blue}\beta}
\end{smallmatrix}
\frac{2(c+d)s}{(s+c)(s+d)}~,
\]
which for \mbox{${\scriptstyle\alpha\in(0,~1)}$} is a non-{\em canonical}.
\bigskip

{\bf (ii)}~ For $d=c$, a balanced realization of the functions in Eq.
\eqref{eq:Phi_3_and_Phi_4} may be written as,
\[
R_{\phi_{3,4}}=\left({\footnotesize\begin{array}{cc|c}-c&-2c&
\mp{\scriptstyle\sqrt{2c}\frac{(1-
{\scriptstyle\color{blue}\beta}^2)^{\frac{1}{4}}}
{\scriptstyle\sqrt{\color{blue}\beta}}}
\\~~0&-c&
\mp{\scriptstyle\sqrt{2c}\frac{(1-
{\scriptstyle\color{blue}\beta}^2)^{\frac{1}{4}}}
{\scriptstyle\sqrt{\color{blue}\beta}}}
\\ 
\hline
{\scriptstyle\sqrt{2c}
\frac{(1-{\scriptstyle\color{blue}\beta}^2)^{\frac{1}{4}}}
{\scriptstyle\sqrt{\color{blue}\beta}}}
&
{\scriptstyle\sqrt{2c}
\frac{(1-{\scriptstyle\color{blue}\beta}^2)^{\frac{1}{4}}}
{\scriptstyle\sqrt{\color{blue}\beta}}}
&
{\scriptstyle
\frac{1}{\color{blue}\beta}
\pm
\frac{\sqrt{1-{\scriptstyle\color{blue}\beta}^2}}
{\color{blue}\beta}}
\end{array}}\right).
\]
Hence, for $\begin{smallmatrix}\alpha\in[0,~1]\end{smallmatrix}$ a convex
combination of these realizations arrays is given by,
\[
{\scriptstyle\alpha}
R_{\phi_3}
{\scriptstyle(1-\alpha)}
R_{\phi_4}
=\left({\footnotesize\begin{array}{cc|c}-c&-2c&
{\scriptstyle(1-2\alpha)
\sqrt{2c}
\frac
{(1-{\scriptstyle\color{blue}\beta}^2)^{\frac{1}{4}}}
{\scriptstyle\sqrt{\color{blue}\beta}}
}
\\~~0&-c&
{\scriptstyle(1-2\alpha)
\sqrt{2c}
\frac{(1-{\scriptstyle\color{blue}\beta}^2)^{\frac{1}{4}}}
{\scriptstyle\sqrt{\color{blue}\beta}}
}
\\ 
\hline
{\scriptstyle
\sqrt{2c}
\frac{(1-{\scriptstyle\color{blue}\beta}^2)^{\frac{1}{4}}}
{\scriptstyle\sqrt{\color{blue}\beta}}
}
&
{\scriptstyle
\sqrt{2c}
\frac{(1-{\scriptstyle\color{blue}\beta}^2)^{\frac{1}{4}}}
{\scriptstyle\sqrt{\color{blue}\beta}}
}
&
{\scriptstyle
\frac{1}{\color{blue}\beta}
+{\scriptstyle(2\alpha-1)}
\frac{\sqrt{1-{\scriptstyle\color{blue}\beta}^2}}
{\color{blue}\beta}
}
\end{array}}\right).
\]
For all \mbox{${\scriptstyle\alpha\in[0,~1]}$,} this can be viewed as a
realization of the function,
\[
\begin{smallmatrix}
\frac{1}{\color{blue}\beta}
\end{smallmatrix}
+
\begin{smallmatrix}
(2\alpha-1)\frac{\sqrt{1-{\color{blue}\beta}^2}}
{\color{blue}\beta}
\end{smallmatrix}
\frac{(s-c)^2}{(s+c)^2}~,
\]
which is identical to the function in Eq.
\eqref{eq:Convex_Canonical_Functions_Degree_Two}, when $d=c$.\\
For \mbox{${\scriptstyle\alpha\in(0,~1)}$} this 
realization is not balanced, and the resulting function 
is a non-{\em canonical}.
}
\end{Ex}
$\T$

\begin{center}
Acknowledgment
\end{center}

The authors wish to express their appreciation and gratitude to the referees
for providing them with a thorough, constructive review, improving the final version
of this work.

\end{document}